\documentclass[12pt]{article}
\usepackage{enumerate}
\usepackage{natbib}

\RequirePackage[colorlinks,citecolor=blue,linkcolor=blue,urlcolor=blue,pagebackref]{hyperref}
\usepackage{latexsym}
\usepackage{amssymb}
\usepackage{amsmath}
\usepackage{amsthm}
\usepackage{mathrsfs}
\usepackage{stmaryrd}
\usepackage{url}
\usepackage{longtable}
\usepackage{color}
\usepackage{dsfont}
\usepackage[T1]{fontenc}
\usepackage{lmodern}
\usepackage{mathtools}
\usepackage{tabularx,booktabs}
\usepackage{multirow}
\usepackage{caption}
\usepackage{subcaption}
\usepackage{xcolor}
\usepackage{bbm}
\usepackage{tensor}
\usepackage{float}
\usepackage{graphicx}
\usepackage[latin1]{inputenc}
\usepackage{enumitem}
\usepackage{amsopn}
\usepackage{comment}
\usepackage{breqn}

\newcommand{\blind}{1}

\numberwithin{equation}{section}
\theoremstyle{plain}
\newtheorem{Theorem}{Theorem}[section]
\newtheorem{Corollary}[Theorem]{Corollary}
\newtheorem{Lemma}[Theorem]{Lemma}
\newtheorem{Proposition}[Theorem]{Proposition}

\newtheorem{Assumption}{Assumption}
\renewcommand{\theAssumption}{\Alph{Assumption}}
\makeatletter
\newcommand{\settheoremtag}[1]{
	\let\oldtheAssumption\theAssumption
	\renewcommand{\theAssumption}{#1}
	\g@addto@macro\endAssumption{
		\global\let\theAssumption\oldtheAssumption}
}
\makeatother
\theoremstyle{remark}
\newtheorem{Definition}[Theorem]{Definition}
\newtheorem{Remark}[Theorem]{Remark}
\newtheorem*{Example}{Example}

\newcommand{\btab}{\begin{tab}}
	\newcommand{\etab}{\end{tab}}
\newcommand{\barr}{\begin{array}}
	\newcommand{\earr}{\end{array}}
\newcommand{\beq}{\begin{equation}}
	\newcommand{\eeq}{\end{equation}}
\newcommand{\bdis}{\begin{displaymath}}
	\newcommand{\edis}{\end{displaymath}\noindent}

\newcommand{\bbn}{\mathbb{N}}

\newcommand{\bbr}{\mathbb{R}}

\newcommand{\bbe}{\mathbb{E}}
\newcommand{\bbq}{\mathbb{Q}}

\newcommand{\bbf}{\mathbb{F}}
\newcommand{\bbl}{\mathbb{L}}

\newcommand{\bbs}{\mathbb{S}}
\newcommand{\bbt}{\mathbb{T}}
\newcommand{\bone}{\mathds 1}

\newcommand{\R}{\mathbb{R}}

\newcommand{\N}{\mathbb{N}}
\newcommand{\E}{\mathbb{E}}
\newcommand{\F}{\mathbb{F}}

\renewcommand{\P}{\mathbb{P}}

\newcommand{\lims}{\stackrel{\mathrm{st}}{\longrightarrow}}
\newcommand{\limL}{\stackrel{L^1}{\Longrightarrow}}
\newcommand{\limp}{\stackrel{\mathbb{P}}{\longrightarrow}}

\newcommand{\nto}{{n\to\infty}}

\newcommand{\lec}{\lesssim}

\newcommand{\calc}{\mathcal{C}}

\newcommand{\calf}{\mathcal{F}}

\newcommand{\calw}{\mathcal{W}}

\newcommand{\calz}{\mathcal{Z}}

\newcommand{\al}{{\alpha}}
\newcommand{\la}{{\lambda}}
\newcommand{\La}{{\Lambda}}
\newcommand{\eps}{{\varepsilon}}

\newcommand{\ga}{{\gamma}}
\newcommand{\Ga}{{\Gamma}}

\newcommand{\si}{{\sigma}}
\newcommand{\om}{{\omega}}
\newcommand{\Om}{{\Omega}}
\renewcommand{\phi}{\varphi}

\newcommand{\ov}{\overline}
\newcommand{\un}{\underline}
\newcommand{\wh}{\widehat}
\newcommand{\wt}{\widetilde}

\newcommand{\dd}{\mathrm{d}}

\newcommand{\pd}{\partial}

\newcommand{\Del}{\Delta_n}
\newcommand{\Delh}{\sqrt{\Delta_n}}

\DeclareMathOperator{\var}{{\mathrm{Var}}}
\DeclareMathOperator{\cov}{{\mathrm{Cov}}}

\newcommand{\ra}{\rightarrow}
\newcommand{\ind}{\mathds{1}}

\newcommand{\wti}{\widetilde}

\newcommand{\Den}{\Delta_n}
\newcommand{\DenOneHalf}{\Delta_n^{\frac{1}{2}}}
\newcommand{\DenMinOneHalf}{\Delta_n^{-\frac{1}{2}}}

\newcommand{\tr}{\mathrm{tr}}
\newcommand{\di}{\mathrm{dis}}

\newcommand{\Deni}{\Delta^n_i}
\newcommand{\uDeni}{\un{\Delta}^n_i}
\newcommand{\DenH}{\Delta_n^H}

\newcommand{\LeinsKonv}{\stackrel{L^1}{\Longrightarrow}}
\newcommand{\stabKonv}{\stackrel{\mathrm{st}}{\Longrightarrow}}

\newcommand{\bv}{\big \vert}

\newcommand{\Sig}{\Sigma}

\DeclareMathOperator*{\argmin}{arg\,min}

\newcommand{\bthm}{\begin{Theorem}}
	\newcommand{\ethm}{\end{Theorem}}

\newcommand{\bcor}{\begin{Corollary}}
	\newcommand{\ecor}{\end{Corollary}}

\newcommand{\blem}{\begin{Lemma}}
	\newcommand{\elem}{\end{Lemma}}

\newcommand{\bprop}{\begin{Proposition}}
	\newcommand{\eprop}{\end{Proposition}}

\newcommand{\bdf}{\begin{Definition}}
	\newcommand{\edf}{\end{Definition}}

\newcommand{\bex}{\begin{Example}}
	\newcommand{\eex}{\end{Example}}

\newcommand{\brem}{\begin{Remark}}
	\newcommand{\erem}{\end{Remark}}

\newcommand{\bass}{\begin{Assumption}}
	\newcommand{\eass}{\end{Assumption}}

\newcommand{\bpr}{\begin{proof}}
	\newcommand{\epr}{\end{proof}}

\newcommand{\benu}{\begin{enumerate}}
	\newcommand{\eenu}{\end{enumerate}}

\newcommand{\bit}{\begin{itemize}}
	\newcommand{\eit}{\end{itemize}}

\begin{document}

\date{}

\def\spacingset#1{\renewcommand{\baselinestretch}%
{#1}\small\normalsize} \spacingset{1}


\if1\blind
{
 \title{\bf When Frictions are Fractional: \\ Rough Noise in High-Frequency Data}
  \author{Carsten H. Chong\thanks{
    We thank Yacine A\"it-Sahalia,  Torben Andersen, Patrick Cheridito, Jean Jacod, Fabian Mies, Serena Ng, Mark Podolskij, Walter Schachermayer, Viktor Todorov and participants at various conferences and seminars for valuable comments and suggestions, which greatly improved the paper. We would also like to thank the Editor, an Associate Editor and a referee, whose detailed comments on earlier drafts led to a substantial improvement of the paper. The second author is partially supported by the Deutsche Forschungsgemeinschaft, project number KL 1041/7-2.}\hspace{.2cm}\\
    Department of Information Systems, Business Statistics and\\  Operations Management,\\ The Hong Kong University of Science and Technology\\
    and \\
    Thomas Delerue \\
   Institute of Epidemiology, Helmholtz Munich\\
	and \\
	Guoying Li \\
	Department of Statistics, Columbia University}
  \maketitle
} \fi

\if0\blind
{
  \bigskip
  \bigskip
  \bigskip
  \begin{center}
    {\LARGE\bf When Frictions are Fractional:\\ \bigskip Rough Noise in High-Frequency Data}
\end{center}
  \medskip
} \fi

\bigskip
\begin{abstract}
The analysis of high-frequency financial data is often impeded by the presence of noise. This article is motivated by intraday return data in which market microstructure noise appears to be \emph{rough}, that is, best captured by a continuous-time stochastic process that locally behaves as fractional Brownian motion. Assuming that the underlying efficient price process follows a continuous Itô semimartingale, we derive consistent estimators and asymptotic confidence intervals for the roughness parameter of the noise and the integrated price and noise volatilities, in all cases where these quantities are identifiable. In addition to desirable features such as serial dependence of increments, compatibility between different sampling frequencies and diurnal effects, the rough noise model can further explain divergence rates in volatility signature plots that vary considerably over time and between assets.
\end{abstract}

\noindent%
\vfill

\newpage
\spacingset{1} 
	\section{Introduction}\label{sec:intro}
	A classical statistical  inference problem consists in   separating a signal $X$ from a noise term $Z$ when only their sum   
\beq\label{eq:Y} Y=X + Z \eeq
is observed. This paper analyzes a particular instance of this problem in which the signal term $X$ is a continuous It\^o semimartingale of the form
\beq\label{eq:X} X_t = X_0 + \int_0^t a_s\, \dd s+\int_0^t \si_s \,\dd B_s, \eeq 
the noise term $Z$ is a rough fractional process to be specified in Section~\ref{sec:model} below, and the observations $\{Y_{i\Den}:i=0,\dots,[T/\Den]\}$ are recorded on a regularly spaced time grid  where $\Den\to0$ and  $T>0$ is fixed. 

The motivation for this problem comes from the statistical analysis of high-frequency financial data, where $X$ models the efficient logarithmic price of an asset (i.e., its economic value in a frictionless setting) and $Z$ denotes microstructure noise, which in practice arises due to bid--ask bounces, transaction costs and other market frictions. In this context, a key quantity of interest is the \emph{integrated (price) volatility}
$C_T=\int_0^T \si_s^2\,\dd s$. In the absence of noise, estimating $C_T$ is a straightforward matter: given observations $\{X_{i\Del}:i=0,\dots,[T/\Del]\}$,   the \emph{realized variance (RV)} defined by $\sum_{i=1}^{[T/\Del]} (\Delta^n_i X)^2$, where $\Delta^n_i X= X_{i\Del}-X_{(i-1)\Del}$, is a consistent estimator of $C_T$ as $\Del\to0$ \citep{JP}.

However, in practice, RV typically explodes as the sampling frequency increases, indicating the presence of noise. 
In order to deal with noisy observations, 
 a common approach in the literature is to model $Z$ at the observation times  $i\Del$ as
\beq\label{eq:nonvan} Z_{i\Del} = \eps^n_i, \eeq
where  $(\eps^n_i)_{i=1}^{[T/\Del]}$ is a discrete time series for each $n$, and to construct noise-robust estimators of $C_T$ based on that assumption.\footnote{\spacingset{1}\footnotesize Examples for $\eps^n_i$ include rounding noise \citep{Delattre97, Li07, Robert10, Rosenbaum09}, white noise   \citep{Bandi06,BN08,Podolskij09,Zhang05}, AR- or MA-type noise \citep{AitSahalia11,Da19,Hansen06}, and certain non-parametric extensions thereof \citep{Jacod09,Jacod17,Li21}. 
}
The current paper is motivated by statistical properties found in certain samples of high-frequency financial data that cannot be explained by such discrete noise models. For instance, if the noise $Z$ is independent of $X$ and takes the form \eqref{eq:nonvan}, where $\eps=(\eps^n_i)_{i=1}^{[T/\Den]}$ is a stationary time series with a distribution that does not depend on $n$, it is a simple consequence of the law of large numbers (LLN) that
$$  \Den\sum_{i=1}^{[T/\Delta_n]} (\Delta^n_i Y)^2 \limp 2\var(\eps)(1- r(1)),$$
where $r$ is the autocorrelation function (ACF) of $\eps^n_i$. In particular, the RV of the observed process $Y$ should be of order $\Den^{-1}$. However, as we can see in the volatility signature plot of  Figure~\ref{fig1} (a), the divergence rate of RV in real samples can be much slower. 
An almost equivalent way of illustrating this observation is to consider \emph{variance plots}, in which the sample variance of increments of $Y$ is computed as a function of the sampling frequency. In our data sample, we observe shrinking price increments; see Figure~\ref{fig1} (b). By contrast, in discrete  noise models, 
\beq\label{eq:var} \var(\Delta^n_i Y) \sim \var(\eps^n_i-\eps^n_{i-1}) = 2\var(\eps)(1-r(1)), \eeq
so asymptotically noise increments do not shrink. 
\begin{figure}[tb!]
	\centering
	\begin{subfigure}{.5\textwidth}
		\centering
		\includegraphics[width=0.95\linewidth]{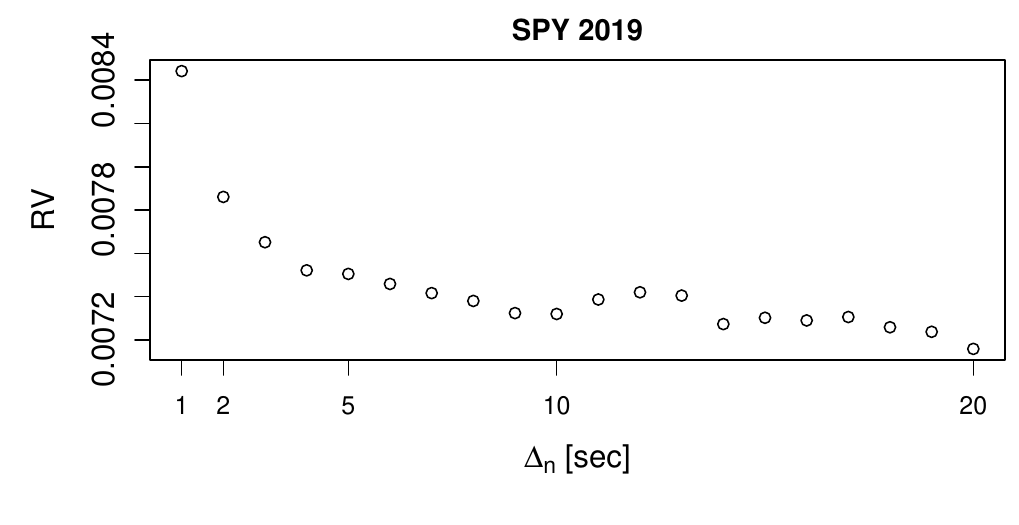}
		\includegraphics[width=0.95\linewidth]{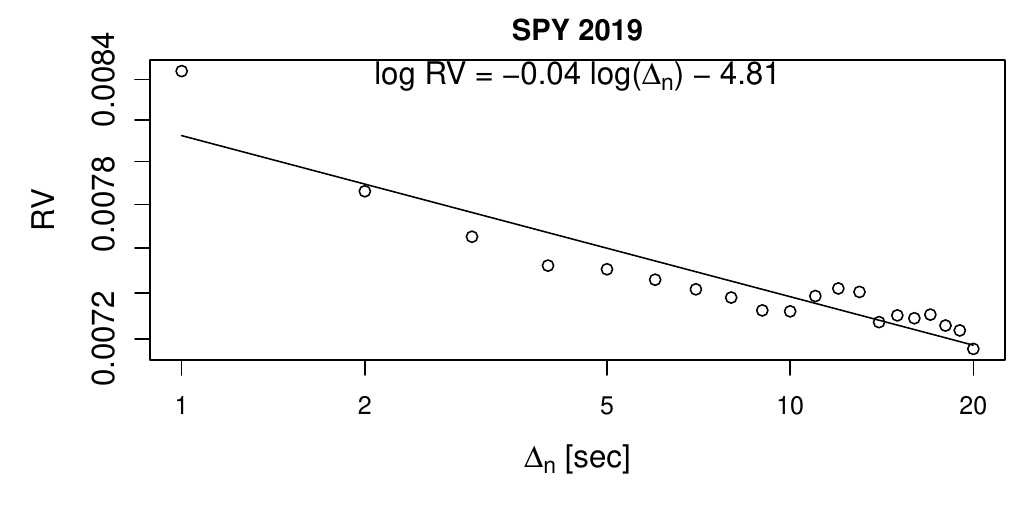}
		\includegraphics[width=0.95\linewidth]{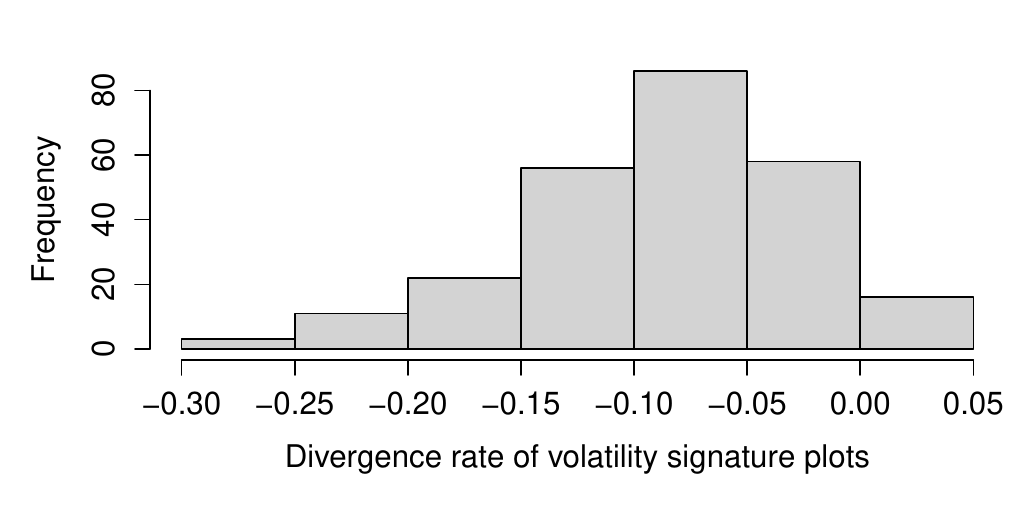}
		\caption{Volatility signature plots}
	\end{subfigure}%
	\begin{subfigure}{.5\textwidth}
		\centering
	\includegraphics[width=0.95\linewidth]{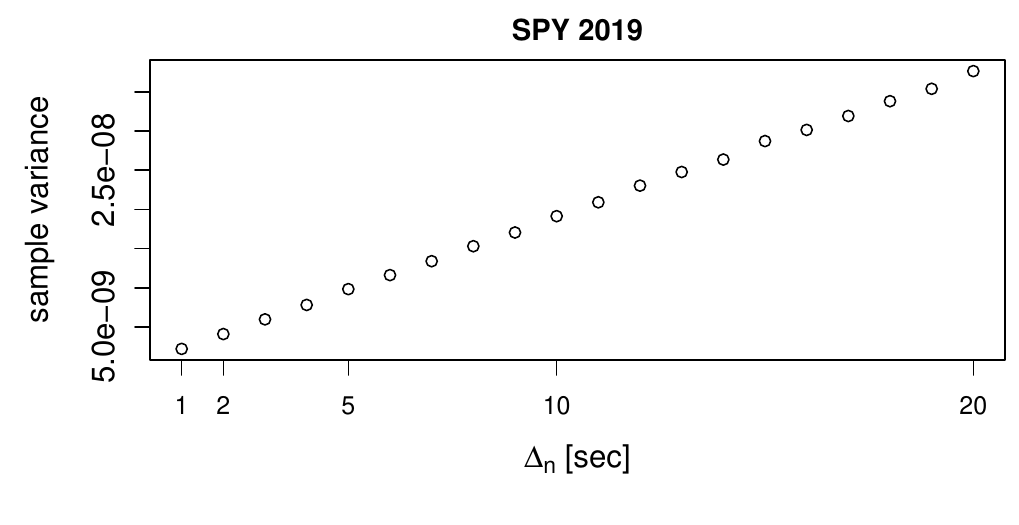}
\includegraphics[width=0.95\linewidth]{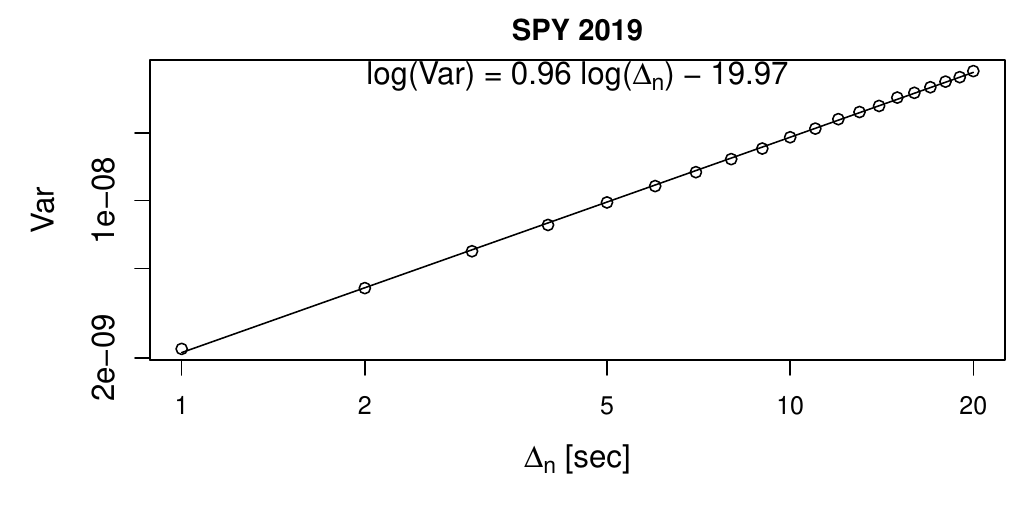}
\includegraphics[width=0.95\linewidth]{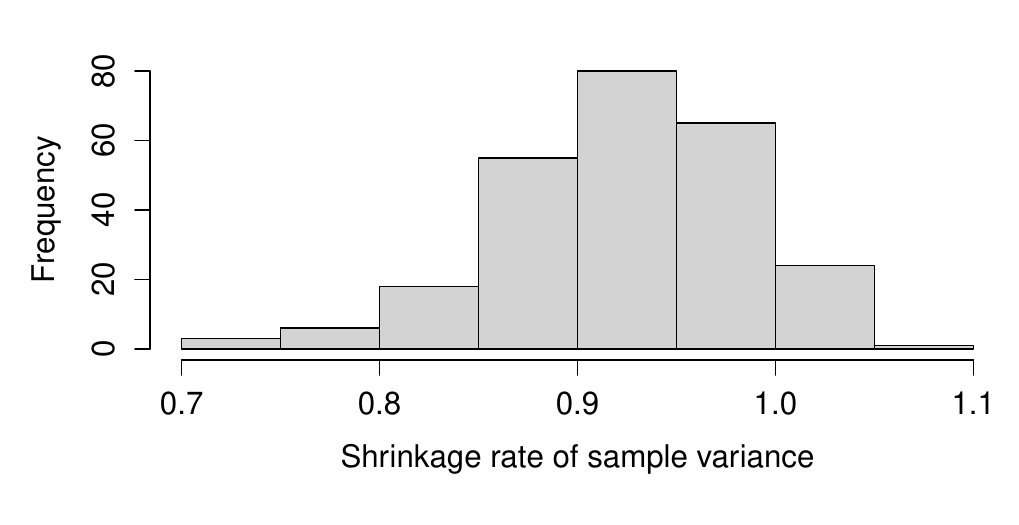}
			\caption{Variance plots}
	\end{subfigure}
	\caption{(a) Volatility signature plot and (b) variance plot for 2019 SPY transaction data (top). The same plots on a log--log scale  (middle)  reveal a divergence rate of $-0.04$ for RV and a shrinkage rate of increments of $0.96$ for the whole year. (The divergence rate of RV is $\alpha$ if $\mathrm{RV}\sim C_1\Den^{\al}$ for some $C_1>0$; the shrinkage rate of increments is $\beta$ if $\var(\Delta^n_i Y)\sim C_2\Den^{\beta}$ for some $C_2>0$.) The histograms (bottom) show the daily divergence rates in volatility signature plots and  the daily  shrinkage rates of price increments  in 2019 SPY transaction data. Each data point corresponds to one trading day.}\label{fig1}
\end{figure}

To corroborate these findings, we perform additional analyses on our data sample, the results of which are shown in Figure~\ref{fig:test}. Panel (a) shows 
\cite{Jacod17}'s point estimators and 95\%-confidence bands for $\var(\eps)$, 
 indicating that $\var(\eps)$ is significantly different from $0$. This suggests that
\beq\label{eq:1} \text{$\var(Z_{i\Den})=\var(\eps^n_i)$ is  bounded away from $0$.} \eeq
Panel (b) shows \cite{Jacod17}'s point estimators and 95\%-confidence intervals for the  first-order autocorrelation $r(1)$ of the noise. As
we can see, there is a   high correlation between $\eps^n_i$ and $\eps^n_{i+1}$, which is not significantly different from $1$ on most days. 
By \eqref{eq:var}, it follows that $\var(\eps^n_i-\eps^n_{i-1})\approx 0$, which we interpret as
\beq\label{eq:2} \var(\Delta^n_i Z)= \var(\eps^n_i-\eps^n_{i-1}) \to 0.\eeq

\begin{figure}[tb!]
	\centering
	\begin{subfigure}{.49\textwidth}
	\centering
	\includegraphics[width=0.95\linewidth]{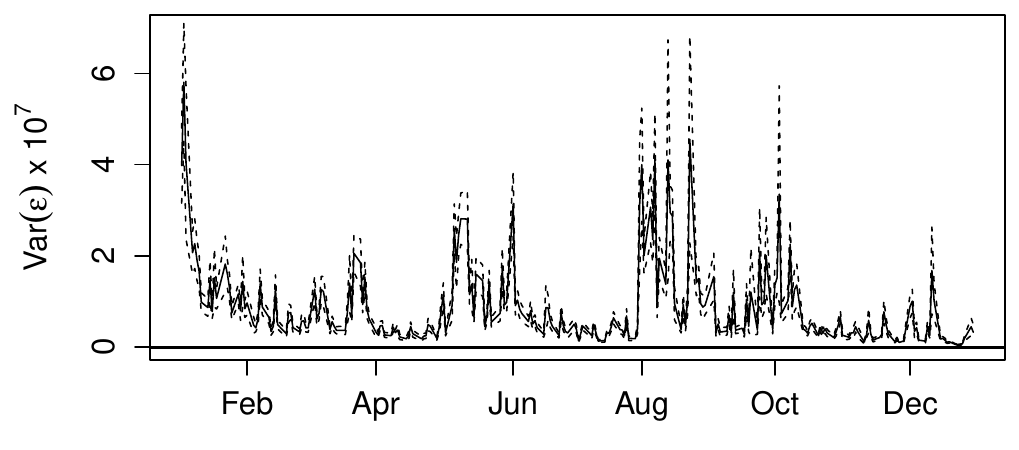}

\vspace{-0.4\baselineskip}
	\caption{}
\end{subfigure}
	\begin{subfigure}{.49\textwidth}
	\centering
	\includegraphics[width=0.95\linewidth]{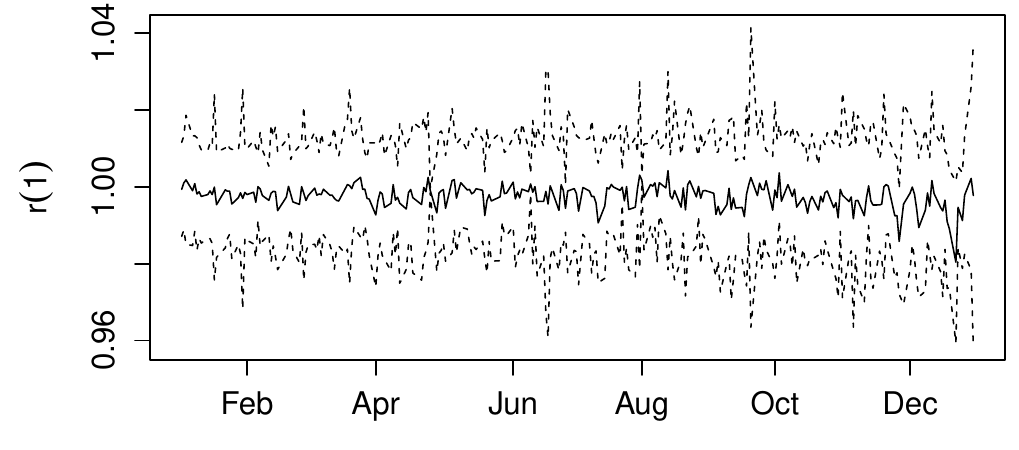}

\vspace{-0.4\baselineskip}
	\caption{}
\end{subfigure}
	\caption{
		(a) Estimators of $\var(\eps)$ and (b) estimators of $r(1)$ including 95\%-confidence intervals. 
		The analysis is based on 2019 SPY transaction data sampled at $\Den=1\,\mathrm{sec}$.}\label{fig:test}
\end{figure}

In conclusion, there is strong empirical evidence that market microstructure noise  in  our data sample is \emph{non-shrinking} (because of \eqref{eq:1}) but with \emph{shrinking increments} (because of \eqref{eq:2}).\footnote{\spacingset{1} \footnotesize We neither claim that this is universally true 
	nor  that noise has shrinking increments in tick time, which is the primary concern in \cite{Jacod17}. It is well known (see e.g., \cite{Hansen06}) that the distributional properties of noise can be very different in calendar time compared to tick time.} To our best knowledge, all microstructure noise models that have been considered so far in the literature are either non-shrinking with non-shrinking increments (as in \eqref{eq:nonvan}) or shrinking with (necessarily) shrinking increments (e.g., \cite{AitSahalia19,Da19,Kalnina08}). 
The goal of this work is to fill in this gap.

To this end,  we introduce in Section~\ref{sec:model} a non-shrinking microstructure noise model with shrinking increments that we call  the \emph{rough noise model}. We   establish a CLT for variation functionals of the resulting semimartingale plus noise process   in  Section~\ref{sec:main}. A major challenge here is the subtle interplay between the semimartingale and the noise process, leading to (a potentially large number of) intermediate limits between the LLN and the CLT. While the LLN and the CLT limits only depend on the noise, these intermediate limits depend on coefficients of the semimartingale and of the noise process at the same time. In Section~\ref{sec:est}, we then combine the CLT of  Section~\ref{sec:main} with a generalized method of moments (GMM) approach to derive consistent and asymptotically mixed normal estimators of integrated price volatility, integrated noise volatility and the roughness of the noise.
Section~\ref{sec:sim} contains a simulation and Section~\ref{sec:emp} shows an empirical study of SPY transaction data. Section~\ref{sec:disc} concludes. The supplement contains more details about modeling noise in continuous time (Appendix~\ref{sec:noisecont}), a multivariate extension of the CLT for mixed semimartingales (Appendix~\ref{sec:mult}) and its proof (Appendix~\ref{Sect:proof}), the proof of the results in Section~\ref{sec:est} (Appendix~\ref{SecC}) and further empirical results (Appendix~\ref{sec:quote}).

\section{Model}\label{sec:model}

Both the noise and the efficient price process are defined on a filtered probability space $(\Om, \mathcal{F}, \F=(\mathcal{F}_t)_{t \geq 0}, \mathbb{P})$ satisfying the usual conditions.
A   natural way of implementing the shrinking increments property \eqref{eq:2} observed in our data is to model the noise process $Z$ in  continuous time. Indeed, if $Z_{i\Den}$ does not change much on average from $i$ to $i+1$, this  implies some form of continuity (e.g., in probability) between them. Therefore, $\{Z_{i\Den}:i\in\N\}$, at least for large $n$, essentially determines a continuous-time process $(Z_t)_{t\geq0}$. A continuous-time noise model further has the advantage that it is compatible  between different sampling frequencies, a property that is typically hard to satisfy for colored noise models with non-shrinking increments (see Section~7.1.2 in \cite{AitSahalia14}). We give more econometric background on about modeling noise in continuous time  in Appendix~\ref{sec:noisecont}.

 \settheoremtag{(Z)}
 \bass\label{ass:Z}
 The process $(Z_t)_{t\geq0}$ is given by
 \beq\label{eq:Z} Z_t=Z_0 + \int_0^t g(t-s)\rho_s\,\dd W_s,\qquad t\geq0,\eeq
 where $W$ is a  standard $\F$-Brownian motion 
 and $(\rho_t)_{t\geq0}$ is an $\F$-adapted locally bounded   process. The kernel $g\colon(0,\infty)\to\R$ is of the form
 \begin{equation} \label{kernel:g}
 	g(t) = K_H^{-1} t^{H - \frac{1}{2}} + g_0(t)
 \end{equation}
for some $H\in (0,\frac 12)$, $K_H  = {\sqrt{2H\sin(\pi H)\Ga(2H)}}/{\Ga(H+\frac12)}$ is a normalization constant
   and $g_0 \colon [0,\infty)   \to \bbr$ is a smooth function with $g_0(0) = 0$.\footnote{\spacingset{1} \footnotesize The condition $H<\frac12$ is not restrictive for the purpose of modeling microstructure noise:
   	if $H=\frac12$, then $Z$ has the same smoothness as Brownian motion, so in general,  there will be no way to discern $Z$ from the efficient price process $X$; if $H>\frac12$, then $Z$ is smoother than $X$ and RV remains a consistent estimator of $C_T$. The normalization $K_H$ is chosen in such a way that $\E[(Z_{t+\delta}-Z_t)^2]/\delta^{2H} \to 1$ as $\delta\to0$ if $\rho\equiv1$.}
 \eass

\begin{Remark} 
	By the Wold--Karhunen representation theorem \cite[Theorem~XII.5.3]{Doob53}, every second-order stationary process, up to deterministic or finite-variation components, has the form $\int_{0}^t G(t-s) \,\dd M_s$
	for some kernel $G\in L^2((0,\infty))$ and some process $(M_t)_{t\geq0}$ with second-order stationary and  orthogonal increments. Therefore, if $Z$ is stationary,  \eqref{eq:Z} is quite a  natural assumption on the noise process. Due to the presence of $\rho$, the process $Z$ in \eqref{eq:Z} does not need to be stationary in general.
\end{Remark}

In the special case where $g_0\equiv0$ and $\rho_s\equiv \rho$ is a constant, $Z$ is---up to a term of finite variation---simply a multiple of \emph{fractional Brownian motion (fBM)}. If further $X_t=\si B_t$ with constant volatility $\si$, then the resulting observed process $Y_t=\si B_t +\rho Z_t$ is  a \emph{mixed fractional Brownian motion (mfBM)}  as introduced by \cite{Cheridito01}. Our model for the observed price process, as the sum of $X$ in \eqref{eq:X} and $Z$ in \eqref{eq:Z}, can be viewed as a non-parametric generalization of mfBM that allows for stochastic volatility in both  its Brownian and its fractional component. We do keep the parameter $H$, though, which we refer to as the \emph{roughness parameter} of $Z$ (or $Y$).\footnote{\spacingset{1}\footnotesize Fractional processes are also used  in \cite{Mandelbrot97,Bayraktar04,Bianchi18} to model asset prices. In these works, the primary interest is short-/long-range dependence, which is determined by   the behavior of $g$ at $t=\infty$. Our interest, by contrast, is the behavior of $g$ around $t=0$, which governs the local regularity, or \emph{roughness}, of the fractional process. Since  our model does not specify the behavior of $g$ at $t=\infty$ (due to the presence of $g_0$ in \eqref{kernel:g}), we refer to $H$  as the roughness parameter of $Z$. 
}
	 In analogy with mfBM, we call 
\begin{equation} \label{mix:SM:mod}
	Y_t=X_t+Z_t=Y_0 +\int_0^t a_s\, \dd s+\int_0^t \si_s \,\dd B_s+ \int_{0}^{t} g(t-s) \rho_s \, \dd W_s,\qquad t\geq0,
\end{equation}
 the observed price process in our model,
 a \emph{mixed semimartingale}.

\brem
 In recent years, there has been growing interest in \emph{rough volatility models} \citep{Gatheral18}, where $\si$   is modeled by a rough  process. In this paper, by contrast, we are concerned with roughness of observed prices, caused by  microstructure noise. Roughness on the price level and roughness on the volatility level imply  different features of asset returns and must  be modeled and analyzed separately. For instance, if $Y_t=X_t=\int_0^t \si_s\,\dd B_s$, without noise but with a rough  $\si$, RV will \emph{not} explode in volatility signature plots. In fact, in the absence of  noise, the asymptotics   of RV do \emph{not} depend on the roughness of volatility \cite[Theorem~5.4.2]{Jacod12}. Therefore, the empirical findings discussed so far and below can neither be explained by nor do they indicate rough volatility. 
\erem

On an abstract level, the statistical problem we are facing in this paper is a deconvolution problem: given a semimartingale process $X$ and rough process $Z$, how can we recover the two (or certain   components of the two, such as volatility) based on observing their sum $Y=X+Z$.
The next result, which follows from \cite[Corollary 2.2]{vanZanten07}, puts a constraint on the identifiability of the (smoother) semimartingale signal:
\bprop\label{prop:imposs} Assume that $Y$ is an mfBM, that is, $Y=X+Z$ where $X=\si B$ and $Z= \rho B^H$ for some  $\rho,\si\in(0,\infty)$, $B$ is a Brownian motion  and $B^H$ is an independent fBM with Hurst parameter $H\in(0,\frac12)$.  For any $T>0$,  the laws of $(Y_t)_{t\in[0,T]}$ and $(Z_t)_{t\in[0,T]}$ are mutually equivalent if   $H\in(0,\frac14)$ and mutually singular if $H\in [\frac14,\frac12)$.
\eprop
In other words,  if $H\in(0,\frac14)$, due to the roughness of the noise, there is no way to consistently estimate $\si$  on a finite time interval. This is conceptually similar to the fact that the finite-variation part of a semimartingale cannot be estimated consistently in finite time if there is a Brownian component. We will comment on possible pathways to estimate $\si$ if $H<\frac14$ in Section~\ref{sec:disc}.

\brem\label{rem:preav} The case of white noise, which formally corresponds to $H=0$ in terms of roughness, is special in this context:  it is rougher than $Z$ in \eqref{eq:Z}, but $C_T=\int_0^T \si_s^2\,\dd s$ can still be recovered  through subsampling \citep{Zhang05} or pre-averaging \citep{Jacod09}. Indeed, if $k_n$ is an increasing sequence and $Z$ is a white noise, then $k_n^{-1}\sum_{j=0}^{k_n} Y_{(i+j)\Del} \approx X_{i\Del}$ by the law of large numbers. By contrast, if $H\in(0,\frac12)$, the process $Z$ in \eqref{eq:Z} is \emph{continuous} (and so is $Y$ in \eqref{eq:Y}), which implies that $k_n^{-1}\sum_{j=0}^{k_n} Y_{(i+j)\Del} \approx Y_{i\Del}$, so pre-averaging does not remove the noise part at all! 
Therefore, while classical noise-robust volatility estimators work well if $Z$ is a modulated white noise, they become inconsistent for $C_T$ if $H\in(0,\frac12)$.   
\erem

\section{Central limit theorem for variation functionals}\label{sec:main}

Our estimators are based on limit theorems for power variations and related functionals in an infill asymptotic setting. To keep notation simple, we only discuss the one-dimensional case here; a multivariate extension of Theorem~\ref{thm:CLT:mixedSM-1} is stated and proved in Appendix~\ref{sec:mult}.  Given a test function $f \colon \bbr \to \bbr$, our  goal is to establish a CLT for \emph{normalized variation functionals}  
\begin{equation*}
	V^n_f(Y, t) = \Delta_n \sum_{i = 1}^{[t/\Del]} f \bigg( \frac{ {\Delta}_i^n Y}{\Den^H} \bigg),
\end{equation*}
where $\Delta_i^n Y = Y_{i \Delta_n}- Y_{(i-1) \Delta_n}$.
For semimartingales, this is a well studied topic; see \cite{AitSahalia14} and \cite{Jacod12} for in-depth treatments of this subject. For fractional Brownian motion or moving-average processes as in \eqref{eq:Z}, the theory is similarly well understood; see \cite{BN11} and \cite{Brouste18}. Surprisingly, it turns out that the mixed case is more complicated than the ``union'' of the purely semimartingale and the  purely fractional case. For instance, as we elaborate in Remark~\ref{rem:CLT3}, already for power variations of even order, we may have a large number of higher-order bias terms.
Our CLT will be proved  under the following set of  assumptions. 

\settheoremtag{(CLT)}
\begin{Assumption} \label{Ass:A-1} The observation process $Y$ is given by the sum of $X$ from \eqref{eq:X} and $Z$ from \eqref{eq:Z} with the following specifications:
	\begin{enumerate}
		\item[(i)]
		The function $f \colon \bbr  \to \bbr $ is even and infinitely differentiable. Moreover, all its derivatives (including $f$ itself) have at most polynomial growth.
		\item[(ii)] Both $B$ and $W$ are independent  standard $\bbf$-Brownian motions, the drift $a$ is locally bounded and $\F$-adapted, and $\si$ is an $\F$-adapted locally bounded  process such that for every $T>0$, there is $K_1\in(0,\infty)$ with
				\begin{equation} \label{reg:cond:rho:B3}
			\bbe \Big[ 1\wedge  \vert \si_t - \si_s   \vert  \Big] \leq K_1 \vert t -s \vert^{\frac{1}{2}},\qquad s,t\in[0,T].
		\end{equation}
		\item[(iii)] The noise volatility process $\rho$ from \eqref{eq:Z} takes the form
		\begin{equation} \label{repr:si}
			\rho_t = \rho^{(0)}_t + \int_{0}^{t} \wti{b}_s \, \dd s + \int_{0}^{t} \wti{\rho}_s \, \dd \wt W_s,\qquad t \geq 0.
		\end{equation}
		In \eqref{repr:si},  $\wt W$ is standard $\bbf$-Brownian motion   that is jointly Gaussian with $(B,W)$;  $\wti{b}$ is   locally bounded and $\F$-adapted;  $\rho^{(0)}$ is an $\F$-adapted locally bounded   process such that for all $T>0$,
			\begin{equation} \label{mom:ass:si:rho:0}
				\bbe \Big[ 1\wedge\vert \rho_t^{(0)} - \rho_s^{(0)}  \vert  \Big]  \leq K_2\vert t - s \vert^{\ga},\qquad s,t\in[0,T],
			\end{equation}
			for some $\ga \in  ( \frac{1}{2}, 1  ]$ and $K_2\in(0,\infty)$; and $\wti{\rho}$ is an $\F$-adapted locally bounded   process   such that for all $T>0$, there exist   $\eps > 0$ and $K_3\in(0,\infty)$ with	
			\begin{equation} \label{reg:ass:si:ti}
				\bbe \Big[ 1\wedge\vert \wti{\rho}_t - \wti{\rho}_s  \vert   \Big]  \leq K_3 \vert t -s \vert^{\eps},\qquad s,t\in[0,T].
			\end{equation}
		\item[(iv)] We have \eqref{kernel:g} with $H\in(0,\frac12)$ and some $g_0\in C^\infty([0,\infty))$ with $g_0(0) = 0$.	
	\end{enumerate}
\end{Assumption}

The following CLT is our   main  technical result.  We write  $\mu_f(v)=\mathbb{E}[f( Z)]$ and $\ga_f(v,q) = \cov(f( Z),f( Z'))$,  where $(Z,Z')$ follows a centered bivariate normal distribution with $\var(Z)=\var(Z')=v$ and $\cov(Z,Z')=q$. Moreover, we 
define
\begin{equation} \label{num:Ga}
	\Ga^H_0 = 1 \qquad \textrm{and} \qquad \Ga^H_r = \frac{1}{2} \Big( (r + 1)^{2H} - 2 r^{2H} + (r-1)^{2H} \Big), \quad r \geq 1,
\end{equation}
and use $\stackrel{\mathrm{st}}{\Longrightarrow}$ (resp., $\limL$) to denote functional stable convergence in law (resp., convergence in $L^1$) in the space of c\`adl\`ag functions  equipped with the local uniform topology.  In the  parametric setup of an mfBM, the CLT for the test function  $f(x)=x^2$ was  obtained by \cite{Dozzi15}.

\begin{Theorem} \label{thm:CLT:mixedSM-1}
	Grant Assumption~\ref{Ass:A-1} and
	let $N(H)=[1/(2-4H)]$ for $H\in(0,\frac12)$. Then  
	\begin{equation} \label{CLT:2-1} 
			\Delta_n^{- \frac{1}{2}} \Bigg\{ V^n_f(Y,t)
			- \int_{0}^{t} \mu_f (\rho_s^2 ) \, \dd s   	- \sum_{j=1}^{N(H)} \frac{\Den^{j (1 - 2H)} }  {j!}\int_{0}^{t}   \mu_f^{(j)} (\rho_s^2) \si_s^{2j} \, \dd s \Bigg \} 
			\stackrel{\mathrm{st}}{\Longrightarrow} \mathcal{Z},
	\end{equation}
	where $\mu_f^{(j)}$ denotes the $j$th derivative of $\mu_f$ and $\mathcal{Z} = (\mathcal{Z}_t)_{t \geq 0}$ is a  continuous process defined on a very good filtered extension $(\ov{\Om}, \ov{\mathcal{F}}, (\ov{\mathcal{F}}_t)_{t \geq 0}, \ov{\mathbb{P}})$ of   $(\Om, \mathcal{F}, (\mathcal{F}_t)_{t \geq 0}, \mathbb{P})$ which, conditionally on  $\mathcal{F}$, is a centered Gaussian process with independent increments and such that the conditional variance function 
	$\calc_t = \ov{\bbe}[\calz _t ^2 \mid \calf]$  is given by
	\begin{equation} \label{cov:fct:lim:CLT-1}
		\calc_t =\int_{0}^{t} \bigg \{  \ga_f(\rho_s^2, \rho_s^2) 
		+2 \sum_{r=1}^{\infty}  \ga_f  (\rho_s^2, \rho_s^2\Ga^H_r) \bigg \} \,\dd s.
	\end{equation}
\end{Theorem}

This result can be extended to a multivariate setting; see   Theorem~\ref{thm:CLT:mixedSM}  in Appendix~\ref{sec:mult}.

\brem\label{rem:CLT1}  It suffices to require $f$ be $2(N(H)+1)$-times continuously differentiable with derivatives of at most polynomial growth. An assumption  as in \eqref{repr:si} is standard for CLTs in high-frequency statistics. But here we need it for $\rho$ (instead of $\si$), as the noise process dominates the efficient price process in the limit $\Del\to0$. Condition \eqref{reg:cond:rho:B3} on $\si$ is satisfied if, for example, $\si$ is itself a continuous It\^o semimartingale. These assumptions do exclude the case of rough volatility. For quadratic functionals (as considered in Corollary~\ref{cor:CLT:lags} below), we conjecture that Assumption~\ref{Ass:A-1} can be relaxed to allow for rough (price and noise) volatility, if further structural assumptions are made concerning   volatility of volatility (e.g., if both $\wt \rho$ and the volatility process of $\si$ are again processes of fractional type); cf.\ \cite{CHLRS23}. To keep the exposition simple, we do not   consider such  an extension here. 
\erem

\brem\label{rem:CLT2} Both the LLN limit $V_f(Y,t)=\int_0^t \mu_f(\rho_s^2)\,\dd s$
and the conditional variance process $\calz$ are driven by the rough component $Z$. In other words, if $\si\equiv0$ (i.e., in the pure fractional case), we would have \eqref{CLT:2-1} without the $\sum_{j=1}^{N(H)}$-expression; see \cite{BN11}. Even if $\si\not\equiv0$, in the case where $H<\frac14$, no additional terms are present because $N(H)=0$. This is in line with Proposition~\ref{prop:imposs}, which states that it is impossible to consistently estimate $C_t=\int_0^t\si^2_s\,\dd s$ if $H<\frac14$. If $H\in(\frac14,\frac12)$, the ``mixed'' terms in the $\sum_{j=1}^{N(H)}$-expression will allow us to estimate $C_t$.
\erem

\brem\label{rem:CLT3} In the special case  where $f(x)=x^{2p}$ for some $p\in\N$,  \eqref{CLT:2-1} reads
\[
\Delta_n^{- \frac{1}{2}} \Bigg\{ V^n_f(Y,t)-\mu_{2p}\int_0^t\rho_s^{2p}\,\dd s
- \sum_{j=1}^{N(H)} \Den^{j (1 - 2H)}  \mu_{2p}{p \choose j}
\int_{0}^{t} \rho^{2p-2j}_s \si^{2j}_s\, \dd s \Bigg \} 
\stackrel{\mathrm{st}}{\Longrightarrow} \mathcal{Z},
\]
where $\mu_{2p}$ is the moment of order $2p$ of a standard normal variable. Typically, one is interested in estimating only one of the terms in the sum $\sum_{j=1}^{N(H)}$ at a time (e.g., $\int_0^t \si^{2p}_s\,\dd s$ corresponding to $j=p$). All other terms (e.g., $j\neq p$) have to be considered as higher-order bias terms in this case. The appearance of (potentially many, if $N(H)$ is large) bias terms for test functions as simple as powers  of even order neither happens in the pure semimartingale nor in the pure fractional setting.
\erem

The proof of Theorem~\ref{thm:CLT:mixedSM-1} is deferred to Appendix~\ref{Sect:proof} in the supplementary material. 
In addition to the usual steps that are common to  CLTs in high-frequency statistics, there are two new challenges   in the present setting:
\benu
\item[(i)] The observation process $Y$ is \emph{not} a semimartingale (and not even close to one). This is because  the rough component $Z$ \emph{dominates} the efficient price process $X$ in the limit as $\Del\to0$. 
In particular, the increments of $Y$ remain conditionally \emph{dependent} as $\Del\to0$.
\item[(ii)] If $H$ is close to (but smaller than) $\frac12$, the semimartingale part is only marginally smoother than the noise part. So for the CLT, there will be an intricate interplay between the efficient price process and the noise process.
\eenu

To overcome the first challenge, we employ a multiscale analysis: by suitably truncating the increments of $Y$, we can restore, to some degree (not on the finest scale $\Del$ but on some intermediate scale $\theta_n\Del$ where $\theta_n\to\infty$), asymptotic conditional independence between increments of $Y$ (see Lemma~\ref{lem:app:2}). This in turn gives  $V^n_f(Y,t)$, as a process in $t$, a semimartingale-like structure on this intermediate scale, which is sufficient for deriving the CLT when we center by appropriate conditional expectations (see \eqref{CLT:fBM}).
 However, because increments are still correlated on the finest scale, the limiting process is not the usual one for semimartingales but the one for  fractional Brownian motion (see \eqref{cov:fct:lim:CLT}, in particular). 
Regarding the second challenge above, we find, to our surprise, that the semimartingale component \emph{never} enters the CLT limit of $V^n_f(Y,t)$ when centered by conditional expectations (see Lemma~\ref{approx:lem:1}), no matter how close $H$ is to $\frac12$.
By contrast, it does affect the limit behavior of these conditional expectations (Lemmas~\ref{lem:approx:6}--\ref{app:ZGW}), producing an $H$-dependent number of higher-order bias terms that neither appear in the pure semimartingale nor in the pure fractional setting.

\section{Estimating the roughness parameter and integrated price and noise volatilities}\label{sec:est}

In this section, we   develop an estimation procedure for the roughness parameter of the noise and the integrated price (if $H>\frac14$) and noise volatilities, that is, for $H$,
$ C_t=\int_{0}^{t} \si_s^2 \, \dd s$ and $\Pi_t= \int_{0}^{t} \rho_s^2 \, \dd s$.
To avoid additional bias terms (cf.\ Remark~\ref{rem:CLT3}), we use quadratic functionals only, that is, we consider 
$f_r(x) = x_1 x_{r + 1}$ for $x = (x_1, \ldots, x_{r + 1}) \in \bbr^{r + 1}$ and $r \in \bbn_0$ 
and the associated variation functionals
$
		V^{n}_{r, t}  = V^{n}_{f_r}(Y, t) = 
		\Den^{1-2H} \sum_{k = 1}^{[t/\Den] - r} 
		 \Delta_k^n Y\Delta_{k+r}^n Y
$. (This is a multivariate variation functional as considered  in Appendix~\ref{sec:mult}.)
Note that $V^n_{r,t}$ is not a statistic as it depends on the unknown  parameter $H$. Therefore, we introduce $\wh{V}^{n}_{t}=(\wh{V}^{n}_{0, t}, \ldots,\wh{V}^{n}_{R, t})^T$, a  non-normalized version of $V^n_{r,t}$ that is a statistic: 
\begin{equation*} 
	\begin{split}
		\wh{V}^{n}_{r, t} & = \wh{V}^{n}_{f_r}(Y, t) = 
		\sum_{k = 1}^{[t/\Den] - r} 
		\Delta_k^n Y \Delta_{k+r}^n Y, \qquad r \in \bbn_0.
	\end{split}
\end{equation*}
Clearly, we have $\Delta_n^{1-2H} \wh{V}^{n}_{r, t} = V^{n}_{r, t}$, so a multivariate extension of Theorem~\ref{thm:CLT:mixedSM-1}  (see Theorem~\ref{thm:CLT:mixedSM} in the appendix) immediately yields:
\begin{Corollary} \label{cor:CLT:lags}
	Let $\wh V^n_t = (\wh V^{n}_{0, t},  \ldots,\wh V^{n}_{R, t})^T$ for a fixed but arbitrary  $R\in\N_0$. For   $H\in(0,\frac{1}{2})$, 
	\begin{equation} \label{CLT:var:func:lags}
		\Delta_n^{- \frac{1}{2}} \bigg\{ \Delta_n^{1-2H} \wh{V}^{n}_{t} - \Ga^H \int_{0}^{t} \rho^2_s \, \dd s - e_1 \int_{0}^{t} \si_s^2 \, \dd s \,\Delta_n^{1-2H}\,\bone_{[\frac14,\frac12)}(H) \bigg\}
		\stackrel{\mathrm{st}}{\Longrightarrow} \mathcal{Z},
	\end{equation}
	where $\Ga^H = (\Ga_0^H,\ldots,\Ga_R^H)^T$, $e_1 = (1,0,\ldots,0)^T \in \bbr^{1+R}$ and $\calz$ is   an $\bbr^{R+1}$-valued continuous process defined on  $(\ov{\Om}, \ov{\mathcal{F}}, (\ov{\mathcal{F}}_t)_{t \geq 0}, \ov{\mathbb{P}})$  which, conditionally on  $\mathcal{F}$, is a centered Gaussian process with independent increments such that  for all $r,r'=0,\dots,R$,  
\begin{equation}\begin{split}	\label{covMat:lags} 
	 \calc^H_{r,r'}(t) &=\ov{\bbe}[\calz^{r}_t \calz^{r'}_t \mid \calf]= \calc^H_{r,r'} \int_{0}^{t} \rho_s^4 \, \dd s,\qquad \calc^H_{r,r'}=v^{H,0}_{r,r'}+\sum_{k=1}^\infty (v^{H,k}_{r,r'}+v^{H,k}_{r',r}),\\
		v^{H,k}_{r,r'}&=	\cov(\Delta B^H_{1+k}\Delta B^H_{1+k+r}, \Delta B^H_1\Delta B^H_{1+r'})=\Ga^H_k\Ga^H_{\lvert r-r'+k\rvert} + \Ga^H_{\lvert k-r'\rvert}\Ga^H_{k+r},
		\end{split}
 \end{equation}
where $\Delta B^H_i=B^H_i - B^H_{i-1}$ for a standard fractional Brownian motion $B^H$.
\end{Corollary}

\subsection{Asymptotically mixed normal estimators}

The simplest estimator for $H$ is  obtained by calculating the rate of divergence in volatility signature plots, that is, by regressing $\log \Del$ on $\log \wh V^n_{0,t}$  (see also \cite{Rosenbaum11} for a more general but related concept). However, as noted by \cite{Dozzi15} in their Remark~3.1, already in an mfBM model, this   estimator only has a logarithmic rate of convergence. Indeed, as our simulation study in Section~\ref{sec:sim} shows, this estimator systematically overestimates $H$ unless $H$ is very close to  $\frac12$.
In the pure fractional case, rate-optimal estimators are given by  so-called change-of-frequency  or autocorrelation estimators \citep{BN11,Corcuera13}. Both  extract information about $H$ by considering the ratio of (different combinations of) $\wh V^n_{r,t}$ for  different values of $r$. For example, the simplest autocorrelation estimator is 
\beq\label{eq:acf} \wt H^n_{\text{acf}}=\frac12\bigg[1+\log_2\bigg(\frac{\wh V^n_{1,t}}{\wh V^n_{0,t}}+1\bigg)\bigg], \eeq
which is based on the fact that $\wh V^n_{1,t}/\wh V^n_{0,t} = V^n_{1,t}/V^n_{0,t}\limp \Ga^H_1=2^{2H-1}-1$. But due to the bias term that appears in \eqref{CLT:var:func:lags} when $r=0$, the   convergence rate worsens and  becomes suboptimal when \eqref{eq:acf} is applied to mixed semimartingales. A simple way to circumvent this problem is to consider ratios of $\wh V^n_{r,t}$ for two different values of $r\neq0$. This indeed leads to   estimators of $H$ with rate of convergence $\Den^{-1/2}$, and the first rate-optimal estimator of $H$ in the case of mfBM, which is constructed in Theorem~3.2 of \cite{Dozzi15}, is exactly of this type. However, these estimators  suffer from a fundamental identification problem: if the observed price  is simply  $Y=\si B$ for some constant $\si>0$ (i.e., there is no noise), then, by standard CLTs for Brownian motion, the ratio $\wh V^n_{r_1,t} / \wh V^n_{r_2,t} $ (for $r_1\neq r_2$ with $r_1\neq0$ and $r_2\neq0$) converges stably in law to the ratio $Z_1/Z_2$ of two independent centered   normal random variables. Because $Z_1/Z_2$ has a density supported on $\R$, estimators based on such ratios can generate estimates from any non-empty open interval with positive probability. Thus, with such estimators, it is impossible to tell whether there is evidence of rough noise or whether a small estimate of $H$ is simply the result of chance. Even if $H$ is less than  but close to $\frac12$, the finite-sample variance is so large that in their Remark~3.2, \cite{Dozzi15} do not recommend using such estimators in practice.

Our strategy, by contrast, uses all lags $r=0,\dots, R$ for some finite $R\in\N$ and is a two-step procedure. In a first step, we use the   statistic 
\begin{equation}\label{eq:test} 
	\wh T^n = \frac{\wh V^n_{1,t}}{\sqrt{\sum_{k=1}^{[t/\Den]-1} (\Delta^n_k Y\Delta^n_{k+1}Y)^2}}
\end{equation} 
to test for the presence of noise. If there is no noise (i.e., if $Y=X$ is a semimartingale), then $\wh T^n\lims  N(0,1)$; cf.\ Theorem 8 of \cite{ALTZ23}. If there is noise (i.e., $H<\frac12$ and $\Pi_t\neq0$), then it is easy to see that $\wh T^n \to -\infty$ at a rate of $\Den^{-1/2}$. Therefore, if $\wh T^n > -q_n$ where $q_n = q\log \Den^{-1}$ for some $q>0$, we set 
\begin{equation}\label{eq:H1/2} 
	\wh H^n = \tfrac12,\quad \wh \Pi^n_t=0,\quad\wh C^n_t = \wh V^n_{0,t}.
\end{equation} In the absence of noise, this happens with probability converging to $1$. 

If $\wh T^n \leq -q_n$,  we construct  an estimator $	\wh\theta^n_t=(\wh H^n, \wh\Pi^n_t,\wh C^n_t) $ of $\theta_t=(H,\Pi_t, C_t)$
using a generalized method of moments (GMM) approach \citep{Hansen82},  by solving
\begin{equation}\label{eq:esteq} 
	 \argmin_{\theta=(H,\Pi,C)} \Bigl\{ \bigl\lVert  \wh\calw_n^{1/2}(\wh V^n_t-\Den^{2H-1}\Pi\Ga^H-Ce_1)\bigr\rVert_2^2\quad\text{subject to}~ \Pi,C\geq0,~ H\in(0,\tfrac12]\Bigr\},
\end{equation}
or rather
\begin{equation}\label{eq:F} 
	F_n(\theta)=\nabla_\theta\lVert \wh{\mathcal{W}}_n^{1/2}(\wh V^n_t-\Den^{2H-1}\Pi\Ga^H-Ce_1)\rVert^2_2 = 0\qquad\text{on } (0,\tfrac12]\times[0,\infty)^2,
\end{equation}  
where $\wh\calw_n$ is a (possibly random) sequence of symmetric positive definite matrices in $\R^{(R+1)\times(R+1)}$ and $\lVert\cdot\rVert_2$ denotes the Euclidean norm. The main theorem of this paper is the following.

\begin{Theorem}\label{thm:main} Suppose that (ii)--(iv) of Assumption~\ref{Ass:A-1} are satisfied. Further suppose that $R\geq2$ and that $\wh\calw_n\limp \calw$, where $\calw\in\R^{(R+1)\times(R+1)}$ is a deterministic symmetric positive definite matrix. 
\begin{enumerate}
	\item[(i)] If $H\in(\frac14,\frac12)$ and $\Pi_t,C_t>0$ almost surely, then there exists a sequence of estimators $	\wh\theta^n_t=(\wh H^n, \wh\Pi^n_t,\wh C^n_t) $ of $\theta_t=(H,\Pi_t, C_t)$ such that $\P(F_n(\wh \theta^n_t)=0)\to1$ and
	\begin{equation}\label{eq:conv-theta} 
		D_n(t)^{-1}(\wh \theta^n_t-\theta_t) \lims  (E(t)E(t)^T)^{-1}E(t)\calw^{1/2}\calz 
	\end{equation}
where  $\calz$ is the same process as in Corollary~\ref{cor:CLT:lags} and
\begin{equation*}
D_n(t)	=\begin{pmatrix}\Den^{1/2} & 0 & 0\\ 2\Den^{1/2}\lvert\log \Den \rvert\Pi_t & \Den^{1/2} & 0\\ 0&0&\Den^{2H-1/2}\end{pmatrix}\quad\text{and}\quad E(t)=(\Pi_t\partial_H\Ga^H,\Ga^H,e_1)^T\calw^{1/2}.
\end{equation*}
In the last line,   $\partial_H \Ga^H$ is the entrywise derivative of $\Ga^H$ with respect to $H$. 
\item[(ii)]  If $H\in(0,\frac14)$ and $\Pi_t>0$ almost surely, then there is a sequence of estimators $	\wh\theta^{\prime n}_t=(\wh H^n, \wh\Pi^n_t) $ of $\theta'_t=(H,\Pi_t)$ such that $\P(F'_n(\wh\theta^{\prime n}_t)=0)\to1$ and 
	\begin{equation}\label{eq:conv-theta-2} 
	D'_n(t)^{-1}(\wh\theta^{\prime n}_t-\theta'_t) \lims (E'(t)E'(t)^T)^{-1}E'(t)\calw^{1/2}\calz, 
\end{equation}
where $\calz$ is the same process as in Corollary~\ref{cor:CLT:lags},
\begin{equation}\label{eq:F-2} 
	F'_n(\theta')=F'_n(H,\Pi)=\nabla_{\theta'}\lVert \wh{\mathcal{W}}_n^{1/2}(\wh V^n_t-\Den^{2H-1}\Pi\Ga^H)\rVert^2_2
\end{equation}  
 and
\begin{equation*}
	D'_n(t)	=\begin{pmatrix}\Den^{1/2} & 0 \\ 2\Den^{1/2}\lvert\log \Den \rvert\Pi_t & \Den^{1/2} \end{pmatrix}\quad\text{and}\quad E'(t)=(\Pi_t\partial_H\Ga^H,\Ga^H)^T\calw^{1/2}. 
\end{equation*}
\item[(iii)] In the setup of (i) (resp., (ii)), the sequences $(\wh\theta^n_t)_{n\in\N}$ (resp.,  $(\wh\theta^{\prime n}_t)_{n\in\N}$) are locally unique in the sense that if $\wt \theta^n_t$ also satisfies $\P(F_n(\wt \theta^n_t)=0)\to1$ and $\P(\lVert \wt \theta^n_t - \theta_t\rVert\leq 1/(\log \Den)^2) \to1$ (resp., $\P(F'_n(\wt \theta^n_t)=0)\to1$ and $\P(\lVert \wt \theta^n_t - \theta'_t\rVert\leq 1/(\log \Den)^2) \to1$), then $\P(\wh\theta^n_t = \wt\theta^n_t)\to1$ (resp., $\P(\wh\theta^{\prime n}_t = \wt\theta^n_t)\to1$).
Moreover, in the situation considered in (ii), if $\wh \theta^n_t=(\wh H^n,\wh\Pi^n_t,\wh C^n_t)$  satisfies $\P(F_n(\wh \theta^n_t)=0)\to1$, then \eqref{eq:conv-theta-2} continues to hold with  $\ov\theta^n_t=(\wh H^n,\wh\Pi^n_t)$ instead of $\wh\theta^{\prime n}_t$.
\end{enumerate}
\end{Theorem}

 The proof can be found in Appendix~\ref{SecC} in the supplement and uses the theory of estimating equations  \citep{JS18,MP23} to derive \eqref{eq:conv-theta} and \eqref{eq:conv-theta-2} from Corollary~\ref{cor:CLT:lags}.
The rates of convergence of  $\wh H^n$, $\wh \Pi^n_t$ and $\wh C^n_t$ (if $H>\frac14$) are $\Den^{-1/2}$, $\Den^{-1/2}/\lvert \log \Den\rvert$ and $\Den^{1/2-2H}$, respectively. 
The additional logarithmic factor in estimating $ \Pi_t$ is due to the fact that $H$ is unknown and already appears in the pure fractional setting \citep{Brouste18}.  The rate of convergence of $\wh C^n_t$ decreases with $H$ and $C_t$ can no longer be consistently estimated if $H<\frac14$. Therefore, the previous theorem yields a quantitative version of Proposition~\ref{prop:imposs}. 

If $H<\frac14$, there is no way to estimate $C_t$ consistently on a finite time interval. This is why the case $H<\frac14$ in (ii) has to be stated  separately from   part (i) where $H>\frac14$. However, as a consequence of the last part of Theorem~\ref{thm:main}, there is no need in practice to know or distinguish whether $H<\frac14$ or $H>\frac14$, and we always recommend  solving \eqref{eq:esteq} to obtain estimates of $H$, $\Pi_t$ and $C_t$.  If $H>\frac14$, we know that the resulting estimators are asymptotically mixed normal. If $H<\frac14$, we still have asymptotic normality for the estimators of $H$ and $\Pi_t$ but the estimator of $C_t$ is no longer consistent. 

\subsection{Feasible implementation}

In order to obtain a consistent estimator of the asymptotic variance of $\wh H^n$, $\wh \Pi^n_t$ and $\wh C^n_t$ (if $H\in (\frac14,\frac12)$), we proceed analogously to \cite{LX16} and define 
\begin{equation}\label{eq:acov} \begin{split}
	\wh \Sig_n &= \wh\Sig^{(0)}_n + \sum_{\ell=1}^{\ell_n} K(\ell,\ell_n)(\wh \Sig^{(\ell)}_n + (\wh\Sig^{(\ell)}_n)^T),  \\
	 \wh\Sig^{(\ell)}_n &= \Den\sum_{i=\ell+1}^{[t/\Den]-R} \eta^{(i)}(\eta^{(i-\ell)})^T \in\R^{(R+1)\times(R+1)},\qquad \eta^{(i)}=(\eta^{(i)}_0,\dots,\eta^{(i)}_R)^T,\\
	 \eta^{(i)}_r &= \Delta^n_i Y \Delta^n_{i+r}Y - \wh m^{n,r}_i,\qquad \wh m^{n,r}_i=\frac1{k_n}\sum_{j=0}^{k_n-1} \Delta^n_{i+j} Y\Delta^n_{i+j+r}Y,\\
	\wh\zeta_n& = (\Den^{2\wh H^n}\wh\Pi^n_t (\partial_H \Ga^{\wh H^n}-2\lvert\log \Den\rvert \Ga^{\wh H^n}),\Den^{2\wh H^n}\Ga^{\wh H^n}, \Den e_1) \in \R^{(R+1)\times 3},
	\end{split}
\end{equation}
where $K$ is a deterministic kernel function and $k_n$ and $\ell_n$ are integer sequences.

\begin{Corollary}\label{cor:feas} Assume   the conditions of Theorem~\ref{thm:main}  and that $K$ is uniformly bounded with $K(\ell,\ell_n)\to1$ for every fixed $\ell\geq1$. Further suppose that $k_n$ and $\ell_n$ increase to infinity such that $\ell_n/\sqrt{k_n}\to0$ and $\ell_n\sqrt{k_n\Den}\to0$.	If we denote the
  diagonal elements of the $3\times 3$-matrix
	\begin{equation}\label{eq:avar} 
	\mathbb{V}_n=\Den	(\wh \zeta_n^T\wh {  \mathcal{W}}_n \wh \zeta_n)^{-1}\wh\zeta_n^T \wh{ \mathcal{W}}_n\wh\Sig_n\wh { \mathcal{W}}_n \wh\zeta_n(\wh \zeta_n^T\wh {  \mathcal{W}}_n \wh \zeta_n)^{-1}
	\end{equation}
by $\mathbb{V}_n^H$, $\mathbb{V}_n^\Pi$ and $\mathbb{V}_n^C$ and the distribution function of the standard normal law by $\Phi$, then for any $\ga\in(0,1)$,
\begin{equation*} 
	[\wh H^n \pm \Phi^{-1}((1-\ga)/2)\sqrt{\mathbb{V}_n^H}],\quad [\wh \Pi^n_t \pm \Phi^{-1}((1-\ga)/2)\sqrt{\mathbb{V}_n^\Pi}],\quad [\wh C^n_t \pm \Phi^{-1}((1-\ga)/2)\sqrt{\mathbb{V}_n^C}] 
 \end{equation*}
are, respectively, asymptotic $\ga$-confidence intervals for $H$, $\Pi_t$ and $C_t$ (if $H\in(\frac14,\frac12)$).
\end{Corollary}

\subsection{Finite-sample considerations}\label{sec:fsc}

As $H$ approaches $\frac12$, distinguishing volatility from a marginally rougher noise term becomes increasingly difficult. In this case, it can happen in finite samples that $\wh \Pi^n_t$ yields a better approximation of $C_t$, while $\wh C^n_t$ yields a better approximation of $\Pi_t$.  As $\Pi_t$ and $C_t$ are not separable in the limit   $H=\frac12$, there is no way this can be prevented in general. However, if one is willing to incorporate a priori information such as the assumption that $\Pi_t$ is smaller than $C_t$ (which in our application is supported by previous empirical results of \cite{AY09}), one can restrict the minimization problem \eqref{eq:esteq} to solutions where $\Pi\leq C$, which by design eliminates the mix-ups mentioned before.\footnote{\spacingset{1} \footnotesize The assumption $\Pi_t\leq C_t$ does \emph{not} imply that   noise only accounts for a small proportion of the log-return variance. Indeed, this proportion is given by $\Den^{2H-1}\Pi_t/(\Den^{2H-1}\Pi_t+C_t)$, which can be large because of $\Den^{2H-1}$ even if $\Pi_t\leq C_t$.}
Thus, we implement estimators $(H^n, \Pi^n_t, C^n_t)$ obtained as follows:
\begin{enumerate}
	\item[(i)] If $\wh T^n >-q_n$, we set $(H^n, \Pi^n_t, C^n_t)=(\wh H^n, \wh \Pi^n_t, \wh C^n_t)$ from \eqref{eq:H1/2}.
	\item[(ii)] Otherwise, we compute $(H^n, \Pi^n_t, C^n_t)$ by solving 
	\begin{equation*}
		\argmin_{\theta=(H,\Pi,C)} \Bigl\{ \bigl\lVert \wh\calw_n^{1/2}(\wh V^n_t-\Den^{2H-1}\Pi\Ga^H-Ce_1)\bigr\rVert_2^2\quad\text{subject to}~ C\geq\Pi\geq0,~ H\in(0,\tfrac12]\Bigr\}.
	\end{equation*}
\end{enumerate}
As long as $\Pi_t\leq C_t$ if $H>\frac14$, we have $\P( H^n=\wh H^n,\ \Pi^n_t = \wh\Pi^n_t,\ C^n_t=\wh C^n_t \text{ if } H>\frac14)\to1$, which means that $(H^n, \Pi^n_t, C^n_t)$ is only a finite-sample adjustment of $(\wh H^n, \wh\Pi^n_t, \wh C^n_t)$ and the asymptotic results of Theorem~\ref{thm:main} and Corollary~\ref{cor:feas} continue to hold for $(H^n, \Pi^n_t, C^n_t)$.

\section{Monte Carlo simulation}\label{sec:sim}

We evaluate the performance of $  H^n$, $  \Pi^n$ and $  C^n$ when applied to  the mfBM model
$$Y_t= X_t+Z_t=\si B_t +\rho B^H_t,\qquad t\in[0,T],$$
where $\si=0.02$,  $B$ and $B^H$ are independent and $T=5$ trading days, each of which consists of $6.5$ hours or $n=23{,}400$ seconds. Accordingly, we choose $\Del=1/n=1/23{,}400$. The values of $H$ will be taken from the set
\beq\label{Hvalues} H\in \{0.05, 0.1, 0.15, 0.2, 0.25, 0.3, 0.35, 0.4, 0.45\}. \eeq
We also include ``$H=0.5$'' (i.e.,  $\rho=0$) and ``$H=0$,'' in which case $(B^0_t)_{t\in[0,T]}$ is a  centered Gaussian white noise with variance $1/2$ (such that increments have variance $1$). The value of $\rho$ is chosen dependent on $H$ such that noise accounts for $1/3$  of the log-return variance (i.e., such that $\rho^2 \Den^{2H-1} / (\rho^2 \Den^{2H-1} + \si^2) =\frac13$). This choice roughly matches the empirical results of \cite{AY09}.

We choose $q_n=1.645$ as the 95\%    standard normal quantile, which corresponds to an initial test for the presence of noise using \eqref{eq:test} at a 5\%-level. If ``no noise'' is rejected, we compute an estimate   $H^n$ following the procedure described in Section~\ref{sec:fsc} using five days of simulated data. We choose  $R=10$, which  corresponds to considering autocovariances up to a lag of ten seconds. Furthermore, we choose $k_n=300\approx 2\Den^{-1/2}$, which corresponds to computing the local autocovariances $\wh m^n_i$ in \eqref{eq:acov} over 5-minute intervals. For the computation of $\wh \Sig_n$, we take the Parzen kernel $K(\ell,\ell_n)=k(\ell/(\ell_n+1))$, where $k(x)=(1-6x^2+6x^3)\bone_{\{x\leq 1/2\}}+2(1-x)^3\bone_{\{x>1/2\}}$ and $\ell_n$ is selected according to the optimal procedure of \cite{NW94}. This guarantees that $\wh \Sig_n$ is positive semidefinite in finite samples and that the optimal $\ell_n$, which is of order $\Den^{-1/5}$, satisfies the rate conditions of Corollary~\ref{cor:feas} if $k_n$ is of order $\Den^{-1/2}$. 
The weight matrix $\wh \calw_n$ is chosen as $\wh W_n = (\wh\Sig_n)^{-1}$ in order to obtain an optimal GMM procedure.
 
\begin{table}[ht!]
	\setlength{\tabcolsep}{0.09cm}
	\spacingset{1}
	\begin{center}\small 
		\caption{Bias,  SE and RMSE of $H^n$}\label{tab:H}
		\begin{tabularx}{\textwidth}{Xccccccccccc}
			\toprule
			$H$    & 0 & 0.05 &0.10 &0.15&0.20&0.25&0.30&0.35&0.40&0.45&0.50  \\
			\hline

			Bias &0.0091& -0.0005& -0.0007 &-0.0007 &-0.0008 &-0.0010 &-0.0012 &-0.0016 &-0.0042 &-0.0255 &-0.0184\\
			SE & 0.0122 &0.0218 &0.0222 &0.0226 &0.0235 &0.0251 &0.0280 &0.0333 &0.0420 &0.0625 &0.0890  \\ 
			RMSE&0.0153 &0.0218 &0.0222 &0.0226 &0.0235 &0.0251 &0.0280 &0.0334 &0.0423& 0.0675 &0.0912\\
			\bottomrule
		\end{tabularx}
	\end{center}
\end{table}
As we can see from Table~\ref{tab:H}, the resulting estimator $H^n$ is essentially unbiased, except when $H$ is very close to $0$ or $0.5$ (but even in this case, the bias is very small). As a result, the main contribution to the root-mean-square error (RMSE) of $H^n$ is the standard error (SE), which is increasing in $H$. This shows that given the same signal-to-noise ratio, estimating $H$ becomes more difficult as $H$ approaches $0.5$. This is reasonable as it is impossible  in the limit as $H\to0.5$ to distinguish a fractional from a semimartingale process. At $H=0.5$, the RMSE of $H^n$ is completely due to the roughly 5\% of cases where $\wh T^n$ from \eqref{eq:test} falsely detects the presence of noise.

Next, we study the distribution of the   pivotal quantity $(H^n-H)/\sqrt{\mathbb{V}^H_n}$. As Figure~\ref{fig:stand.H} shows, for all considered values of $H$ except  $H\in\{0.05,0.40,0.45\}$, the sample quantiles of $(H^n-H)/\sqrt{\mathbb{V}^H_n}$ match standard normal quantiles quite well, confirming the finite-sample reliability of the distributional approximations in  Theorem~\ref{thm:main} and Corollary~\ref{cor:feas} for inferential purposes. If $H=0.05$ (resp., $H\in\{0.40,0.45\}$), the low (resp., high) quantiles of $(H^n-H)/\sqrt{\mathbb{V}^H_n}$  are essentially flat, while higher (resp., lower) quantiles   are approximately   normal. This, of course, is due to the fact that $H^n \in [0,0.5]$ by construction. 
\begin{figure}[ht!]
	\centering
	
	\includegraphics[width=\textwidth]{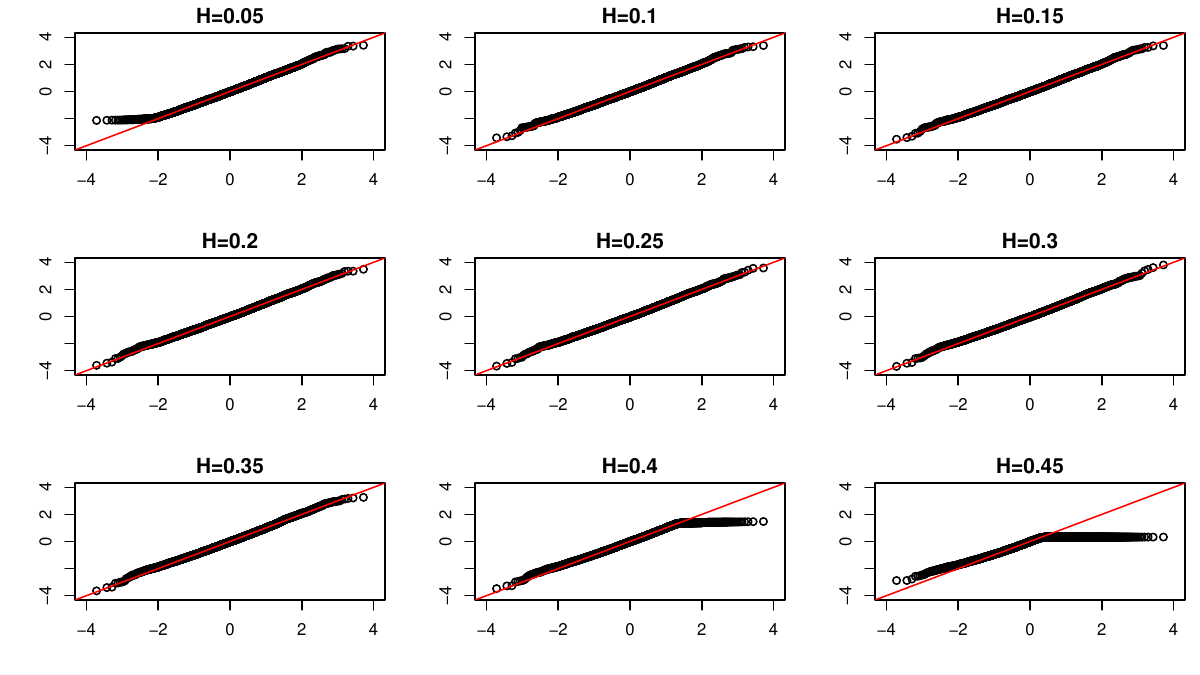}

	\caption{Sample quantiles of $(H^n-H)/\sqrt{\mathbb{V}^H_n}$ against standard normal quantiles.}\label{fig:stand.H}
\end{figure}

Finally, in Figure~\ref{fig:H}, we   compare our estimator $H^n$ 
with three alternatives:  the  estimator $ \wt H^n_{\text{DMS}}=\frac12 (1+\log_{2+} [({\wh V^{n/4}_{0,t}-\wh V^{n/2}_{0,t}})/({\wh V^{n/2}_{0,t}-\wh V^{n}_{0,t}}) ])$ of \cite{Dozzi15},
where  $\log_{2+} x=\log_2 x$ if $x>0$ and $\log_{2+}x=0$ otherwise; the  estimator $ \wt H^n_{\text{VS}}=\tfrac12(\wt\beta^n_{\text{VS}}+1)$ based on volatility signature plots, 
where $\wt\beta^n_{\text{VS}}$ is the slope estimate in  a  linear regression of $\log \wh V^{n/i}_{0,t}$ on $\log i$ for $i=1,\dots,20$; 
and the autocorrelation estimator $\wt H^n_{\text{acf}}$ from \eqref{eq:acf}. 
\begin{figure}[ht!]
	\centering
	
	\includegraphics[width=0.45\textwidth]{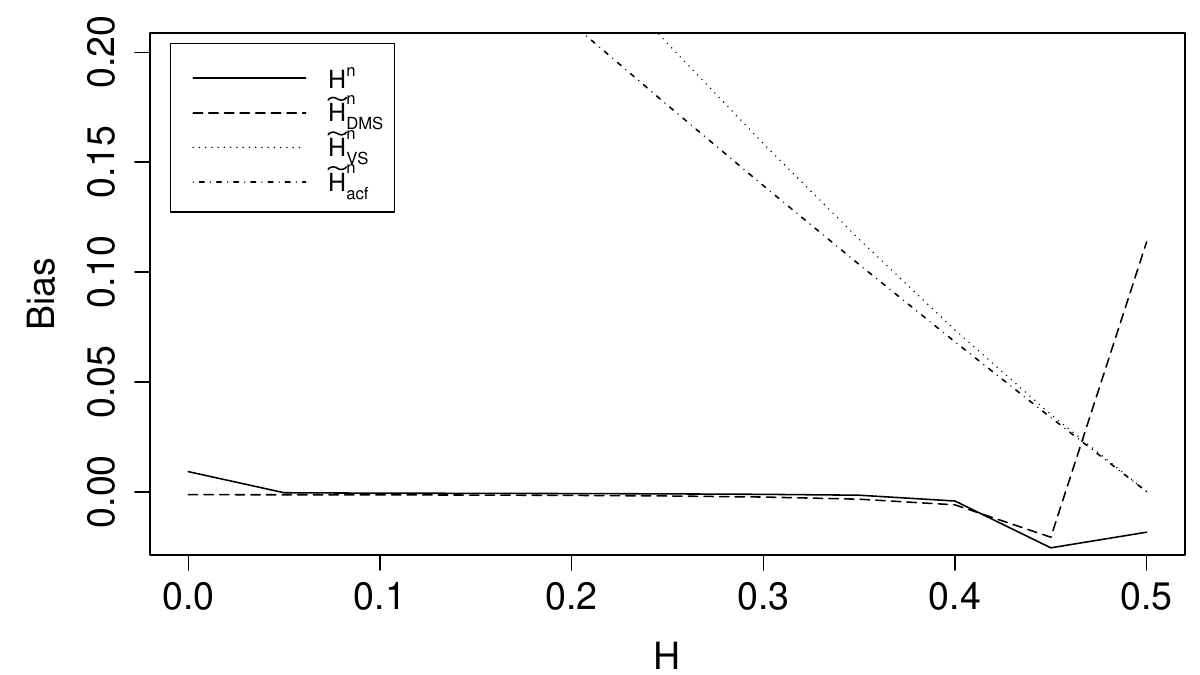}
	\includegraphics[width=0.45\textwidth]{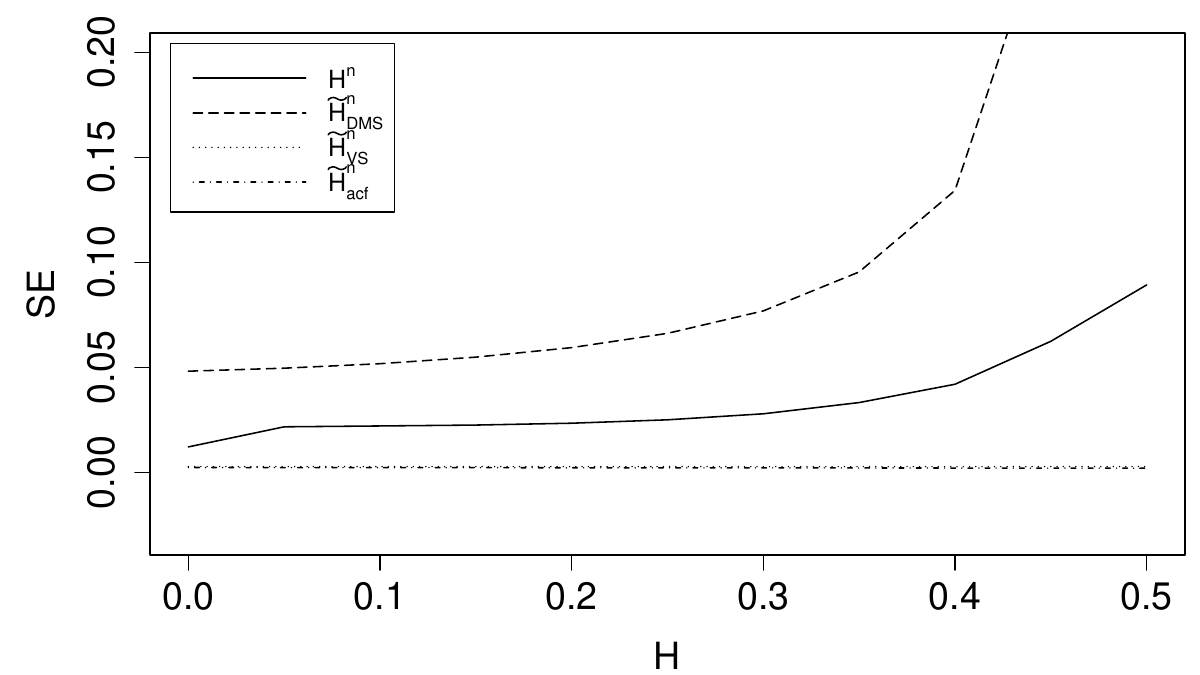}
	\includegraphics[width=0.45\textwidth]{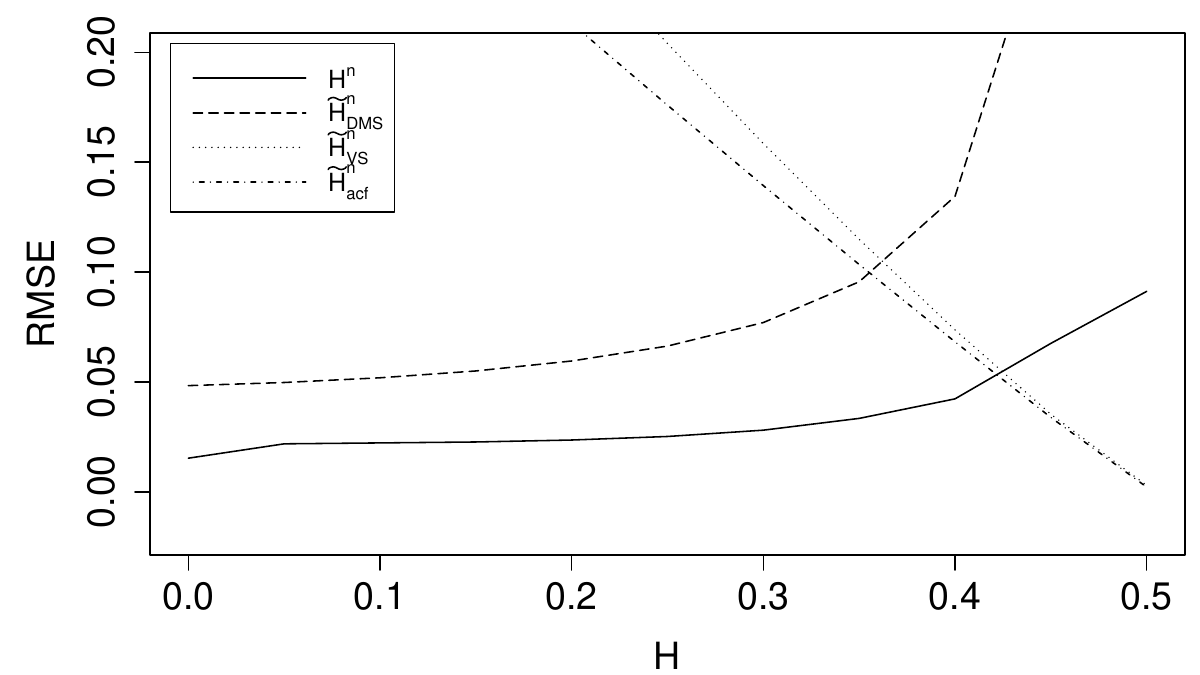}
	\caption{Bias,  SE and RMSE of $H^{n}$, $\wt H^{n}_{\text{DMS}}$, $\wt H^n_{\text{VS}}$ and $\wt H^{n}_{\text{acf}}$ in absolute numbers.}\label{fig:H}
\end{figure}

As the first plot of Figure~\ref{fig:H} shows, the estimators based on   volatility signature plots and   first-order autocorrelation   have large  upward biases except when $H$ is very close to $\frac12$. The estimator by \cite{Dozzi15} is essentially bias-free except for $H=0.5$ where it shows a large upward bias. On top of that, as the second plot shows, the SE of this estimator explodes as $H$ approaches $0.5$, confirming \cite{Dozzi15}'s observation that this estimator is highly unstable. For most values of $H$, our estimator $H^n$ achieves the best RMSE results, which confirms the benefit of  initially testing for the presence of noise as a bias--variance trade-off.

Next, we turn to  volatility estimation.
Having obtained an estimate of $H$, we estimate $C_T-C_{T-1}$ and $\Pi_T-\Pi_{T-1}$, that is, price and noise volatility on the last trading day by repeating steps (ii) and (iii) in Section~\ref{sec:fsc} but with $H$ fixed at the previously obtained estimate. 

\begin{table}[ht!]
	\setlength{\tabcolsep}{0.20cm}
	\spacingset{1.2}
	\begin{center}\small 
		\caption{Bias,  SE and RMSE of $C^n/\si^2$ and $\Pi^n/\rho^2$}\label{tab:sigrho}
		\begin{tabularx}{\textwidth}{X|ccccccccccc}
			\toprule
			$H$    & \multicolumn{5}{c}{Quantiles of $C^n/\si^2$} && \multicolumn{5}{c}{Quantiles of $\Pi^n/\rho^2$}  \\
			\cline{2-6}  \cline{8-12}  
			& 2.5\% & 25\%& 50\% & 75\%& 97.5\%&  & 2.5\% & 25\%& 50\% & 75\%& 97.5\% \\
			\cline{2-12}
			0 &0.9545 &0.9809 &0.9943 &1.0076 &1.0321& &0.9513 &1.0159 &1.0696 &1.3740& 2.4227\\ 
			0.05 & 0.9517& 0.9827& 0.9993 &1.0166 &1.0460& & 0.3849& 0.7162 &0.9946& 1.3753& 2.4168 \\ 
			0.10&0.9460 &0.9810& 0.9997& 1.0184 &1.0512 & & 0.3752 &0.7144& 0.9898 &1.3751 &2.4606\\
			0.15&0.9371& 0.9782& 0.9997& 1.0212 &1.0578& & 0.3600& 0.7091& 0.9902 &1.3895 &2.5175\\
			0.20 & 0.9234& 0.9743& 1.0003 &1.0249& 1.0670 & &0.3394& 0.6899 &0.9891 &1.4025 &2.6643\\
			0.25 & 0.8997& 0.9667 &1.0006 &1.0314 &1.0816& & 0.3034& 0.6625 &0.9831 &1.4453 &2.9399\\
			0.30 & 0.8488 &0.9534 &1.0013 &1.0411 &1.1057 & & 0.2553& 0.6195& 0.9831 &1.5328 &3.5292\\
			0.35 & 0.7008 &0.9245 &1.0027 &1.0613& 1.1495& & 0.1810& 0.5461& 0.9797 &1.7186 &5.2889\\
			0.40 & 0.4371 &0.8262 &1.0043 &1.1116 &1.2342& &0.0830 &0.4092& 0.9685& 2.3267& 6.6144\\
			0.45 &0.5378 &0.6161 &1.0215 &1.2578 &1.3914& & 0.0054 &0.1464 &0.9661 &2.8184 &3.1781\\
			0.50 &0.9758 &0.9930& 0.9996& 1.0060 &1.0174& & -& -& -& - &-  \\
			\bottomrule
		\end{tabularx}
	\end{center}
\end{table}

Table~\ref{tab:sigrho} summarizes the performance of $\Pi^n$ and $C^n$ as respective estimators of $\Pi_T-\Pi_{T-1}$ and $C_T-C_{T-1}$. For all values of $H\leq 0.30$, the performance of $C^n$ is quite good, with a relative error of less than 5\% (resp., 16\%) in 50\% (resp., 95\%) of the cases. For $H=0.5$, the performance is very good, too, with less than 3\% error in 95\% of the cases. The most difficult case is when $H$ is relatively close to  but not equal to $0.5$. While $C^n$ remains essentially unbiased in this case, the relative error increases with $H$. This is expected as it becomes increasingly more difficult as $H\to0.5$   to statistically distinguish a semimartingale process from a fractional one (the two being indistinguishable in the limit $H=0.5$). The results for the noise volatility estimator $\Pi^n$ are qualitatively similar, except that the dispersion of estimates is generally higher in comparison with $C^n$. One explanation is that noise is smaller than volatility in our simulation (and typically in practice as well).

An interesting observation is that even for $H\leq 0.25$, our estimator $C^n$ yields very precise estimates in the simulation study, although $C_T$ cannot be consistently estimated for $H<\frac14$ according to Proposition~\ref{prop:imposs}. This is because we have fixed the same noise-to-signal ratio for all values of $H$, which implies that $\rho$ is smaller for smaller values of $H$. If we had fixed the same $\rho$ for all values of $H$, then as $\Den\to0$, the percentage of log-return variance explained by noise increases very fast to  100\% for small values to $H$, which is quite different from what has been observed in practice \citep{AY09}.

\section{Empirical analysis}\label{sec:emp}

We   apply our estimators $H^n$, $C^n$ and $\Pi^n$  to  SPY transaction data from 2013--2022.  In  Appendix~\ref{sec:quote} of the supplement, we carry out a similar   analysis for transaction data of single-name stocks.
For each trading day in the ten-year period, we collect all  trades 
from 9:30am to 4:00pm Eastern Time from the TAQ database. We apply mild data cleaning procedures and sample in calendar time every  second using the previous-tick method.\footnote{\spacingset{1}\footnotesize We exclude trades with  exchange codes \texttt{D} and \texttt{S}. Furthermore, we only keep trades with trade condition indicator equal to \texttt{0}, \texttt{00}, \texttt{1} or \texttt{01} and a trade condition that is either empty or equal to  \texttt{@}, \texttt{C}, \texttt{E}, \texttt{F}, \texttt{I}, \texttt{M}, \texttt{O}, \texttt{Q}, \texttt{6} or any combination thereof. To account for outliers due to, for instance, big jumps, we further remove log-returns exceeding in absolute value three times the standard deviation of log-returns of the same day.} 
All tuning parameters are chosen exactly as in the Monte Carlo. We compute estimates of $H$ on a moving window of five business days and use these estimates to compute daily estimates of integrated price and noise volatility.

\begin{table}[htb!]
	\setlength{\tabcolsep}{0.09cm}
	\spacingset{1}
		\vspace{\baselineskip}
	\begin{center}\small 
		\caption{Quantiles, mean and standard deviation of daily estimates of $H$ and of the NSR}\label{tab:summary}
		\begin{tabularx}{\textwidth}{X|ccccc|cc}
			\toprule
			& 2.5\% Qu. & 25\% Qu.& 50\% Qu.& 75\% Qu.& 97.5\% Qu.     & Mean & SD \\
			\hline
			Estimates of $H$     & 0.0010 &0.2344 &0.3136 &0.3780  &0.5000  &0.2963 & 0.1285  \\
			Estimates of NSR &0.0000 &0.1478 &0.3386& 0.5722& 0.8973 &   0.3677 & 0.2537   \\
			\bottomrule
		\end{tabularx}
	\end{center}
	\vspace{-\baselineskip}
\end{table}
Table~\ref{tab:summary} shows summary statistics of the daily estimators of $H$ and of the  noise-to-signal ratio (NSR)  $\Pi^n\Den^{2H^n-1}/(\Pi^n\Den^{2H^n-1}+C^n)$ for the whole ten-year period. Next, we show in Figure~\ref{fig:H:nsr}  the empirical distributions of the estimates of $H$ and the NSR separately for each year. While the average estimate of the noise roughness parameter remains in the region $[0.25, 0.35]$ for all years, without a prominent trend, the shape of the distribution does change over the years, from a more concentrated distribution in earlier years towards a more spread out one in recent years. Also, the number of days with no or almost no noise (i.e., $H$ close to $0.5$) tends to be much higher at the end than at the beginning of the considered period. This is in line with the histograms of the NSR estimates, which show a concentration around smaller values in recent years. In summary, while the average roughness of noise appears to be relatively stable, the average magnitude of noise (relative to volatility) seems to decrease over time. This is in agreement with other research (see e.g., \cite{AitSahalia19}) showing that the level of noise in high-frequency return data has been decreasing in recent years.

\begin{figure}[htb!]
	\centering
	\begin{subfigure}{.5\textwidth}
		\centering
		\includegraphics[width=0.9\linewidth]{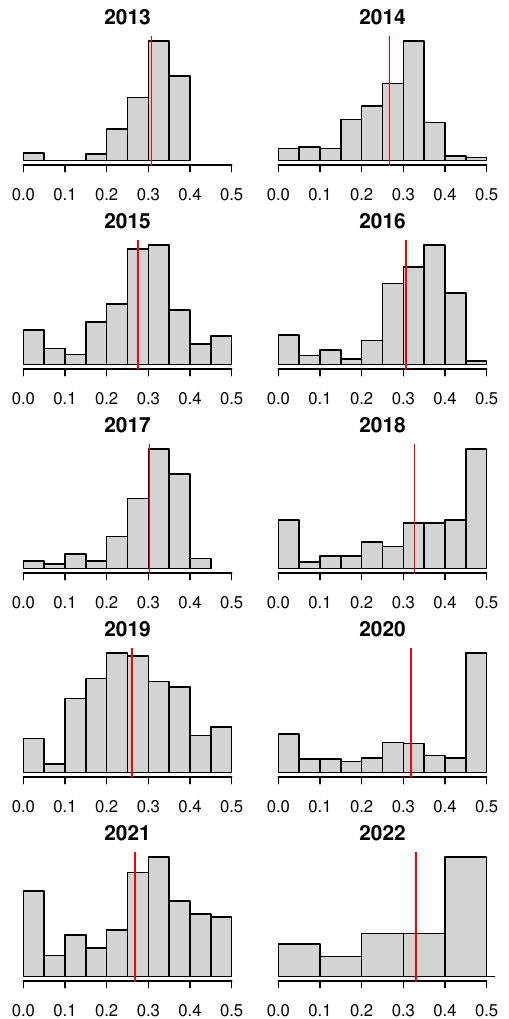}
		\caption{Estimates of $H$}
	\end{subfigure}%
	\begin{subfigure}{.5\textwidth}
		\centering
		\includegraphics[width=0.9\linewidth]{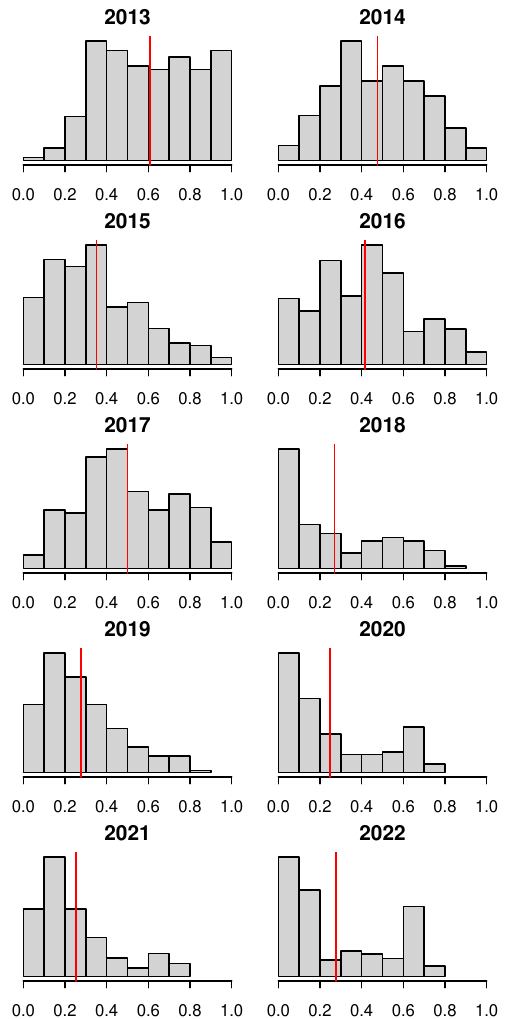}
		\caption{Estimates of NSR}
	\end{subfigure}
	\caption{Histogram of daily estimates of $H$ (a) and  of the NSR (b) based on 1 second SPY transaction data over a period of ten years, with the mean indicated by a red line.}\label{fig:H:nsr}
\end{figure}

To further understand the time-dependence of our estimates, we plot as a function of time
the daily estimates of $H$ (including 95\%-confidence intervals) in Figure~\ref{fig6}
\footnote{\spacingset{1}\footnotesize To reduce oscillations, we show moving averages of $H$-estimates obtained as follows: for each day $i$, we compute an estimate $H^n_i$ using data of the immediate  past five days. If $H^n_i=0.5$ (i.e., no noise is detected), we plot this estimate of $H$ without confidence intervals. If $H^n_i<0.5$ (i.e., there is noise), we plot  $\sum_{j=0}^3 a^n_{ij}H^n_{i-5j}$ as a point estimate of $H$ together with  corresponding confidence intervals. The weights $a^n_{ij}$ are chosen to sum up to one and inversely proportional to the estimated asymptotic variance of $H^n_{i-5j}$. As each $H^n_{i-5j}$ is an asymptotically unbiased estimator of $H$ and $H^n_{i}$, $H^n_{i-5}$, $H^n_{i-10}$ and $H^n_{i-15}$ are asymptotically independent, this choice minimizes the asymptotic mean-squared error among all convex combinations of $H^n_{i}$, $H^n_{i-5}$, $H^n_{i-10}$ and $H^n_{i-15}$. If any of the estimates $H^n_{i-5j}$ with $j=1,2,3$ equals $0.5$ (i.e., the test in (i) of Section~\ref{sec:fsc} fails to detect noise), we exclude it by setting its weight to $0$.}
and of volatility and the NSR in Figure~\ref{fig:vol:nsr}.
While the estimates of $H$ exhibit time-dependence in all considered years, the time variation is stronger in later years, confirming our earlier observation that the distribution of $H$   becomes less concentrated around its mean recently. At the same time, the confidence intervals for $H$ are typically wider in the second half of the considered data. This is line with our previous observation that the NSR decreases over time, which makes inference of $H$ harder. We also note that for most of the time, $H$ is significantly different from $0$ (white noise case) and from $\frac12$ (noise-free case), indicating the presence of rough noise in the data. A notable exceptionis the period around the onset of the COVID-19 pandemic in spring 2020, where the data almost appears as noise-free.
\begin{figure}[tb!]
	\centering
	\includegraphics[width=\linewidth]{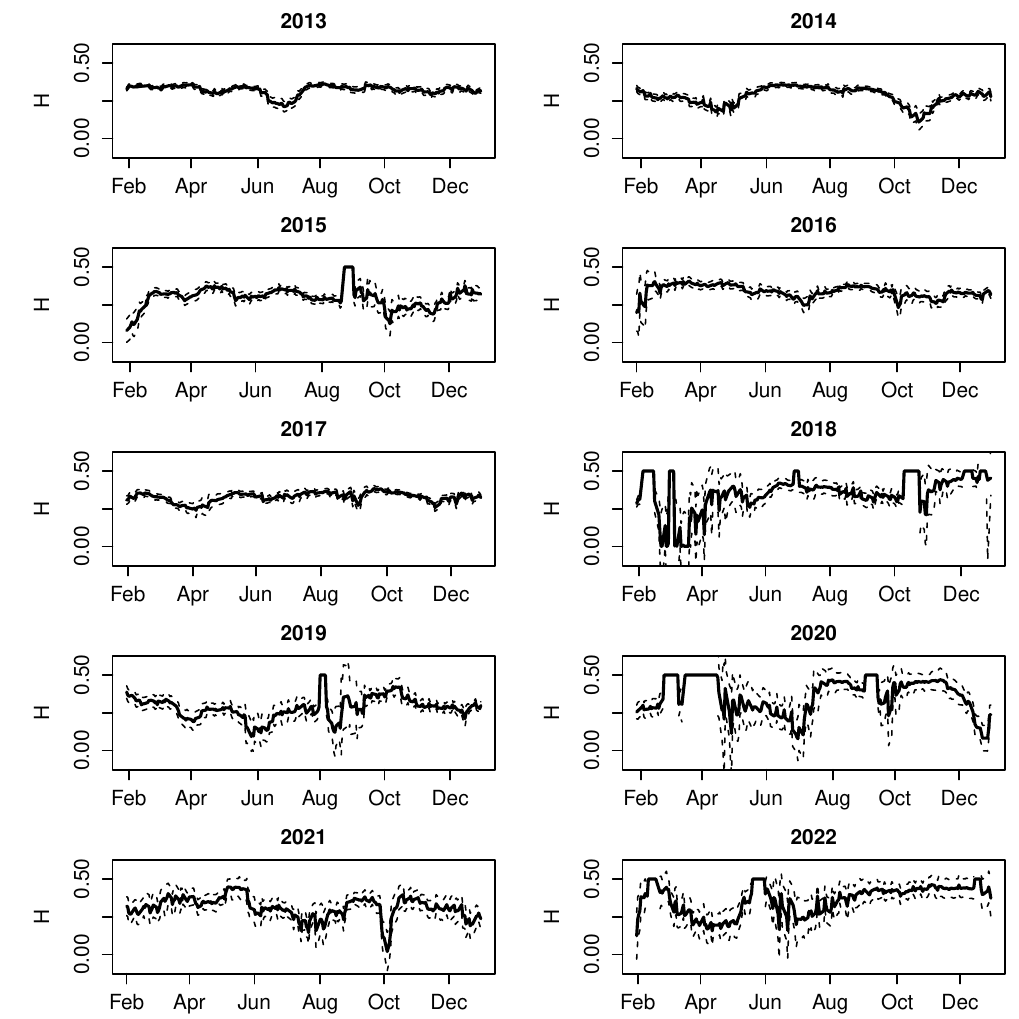}
	\caption{Daily estimates of $H$ with asymptotic 95\%-confidence intervals.}\label{fig6}
\end{figure}
\begin{figure}[htb!]
	\centering
	\includegraphics[width=\linewidth]{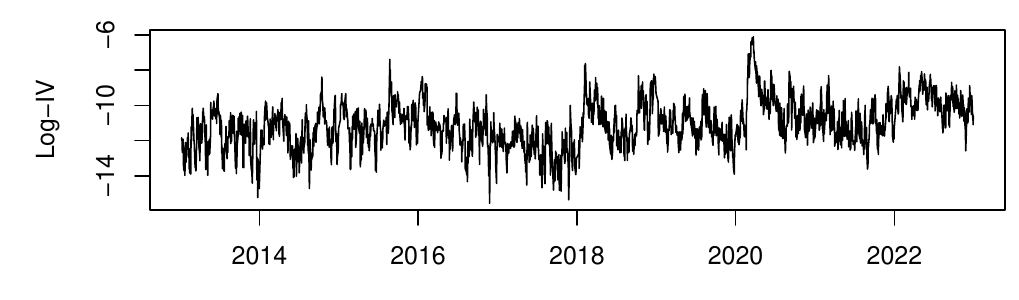}
	\includegraphics[width=\linewidth]{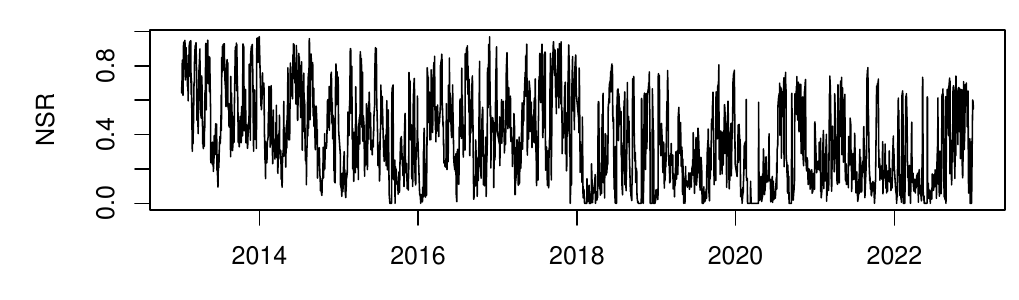}
	\caption{Daily estimates of log-integrated variance and of the NSR.}\label{fig:vol:nsr}
\end{figure}

Let us mention that the empirical evidence of rough noise reported in this section is not of universal nature and depends strongly on the considered asset and the time period. For instance, when analyzing single-name stocks in Appendix~\ref{sec:quote}, we find substantial variation in the roughness of noise, with some stocks such as Intel (INTC) or Coca-Cola (KO) exhibiting rough noise throughout most days in 2019, while other stocks such as Travelers (TRV) or United Technologies (UTX) showed little roughness during the same period. Finally, let us stress that our empirical findings also depend on the data cleaning method applied to the raw data. For instance, if we only kept trades of SPY on NYSE and NASDAQ (which is also a common cleaning procedure), then the resulting price series is almost noise-free in the most recent years of the sample.

\section{Conclusion and future directions}\label{sec:disc}
Volatility estimation based on high-frequency return data is often impeded by the presence of market microstructure noise.
In this paper, we propose to model microstructure noise as a continuous-time rough stochastic process. A distinctive feature of these mixed semimartingale models is a non-shrinking noise component with shrinking increments, which can explain a rich variety of scaling exponents in volatility signature plots.

Using
 CLTs for variation functionals and a GMM approach, we construct consistent and asymptotically mixed normal estimators for the roughness parameter $H$ of the noise and the integrated price and noise volatilities, whenever these quantities are identifiable. In an empirical application, we find  evidence of rough noise in high-frequency return data.

In this first paper, we do not examine the effect of jumps \citep{AitSahalia09,Jacod14} or irregular observation times \citep{BN05,Chen20,Jacod17, Jacod19} on our estimators. Similarly, the current mixed semimartingale model does not capture rounding effects in observed prices \citep{AitSahalia14,Delattre97,Robert10,Robert12}, which are particularly relevant at the highest sampling frequencies.  We leave it to future research to develop estimators that are robust to the aforementioned features of high-frequency data.

A current shortcoming of the mixed semimartingale model is that price volatility cannot  be consistently estimated  for $H<\frac14$. At the same time, our simulation study shows that the estimators of volatility perform very well, even for $H<\frac14$, at practically relevant levels of the noise-to-signal ratio. Therefore, an interesting future direction of research is to examine whether, and how, price volatility can be consistently estimated for all values of $H$ if the  noise volatility coefficient is assumed to be shrinking.

\appendix

\section{Does microstructure noise exist in continuous time?}\label{sec:noisecont}

In the classical  \cite{Roll84} model of transaction prices, deviations of the observed from the efficient price are due to bid--ask bounces associated to each single trade. This raises the  question whether Assumption~\ref{ass:Z}, which postulates the existence of noise  in continuous time, is appropriate. Moreover, another important source of noise  is the discreteness of prices (see \cite{Harris90a, Harris90b} and \cite{Delattre97, Li07, Robert10, Robert12,Rosenbaum09}), which is clearly  not satisfied by   \eqref{mix:SM:mod}.

These seeming contradictions between classical market microstructure theory and our  mixed semimartingale model   can be resolved by taking into account the time scale at which prices are observed. At low to medium frequency (e.g., if $\Delta_n\geq 5\,\mathrm{min}$) and for liquid assets, it is a well established practice  to consider noise as negligible and observed prices as essentially following semimartingale processes.\footnote{\spacingset{1} \footnotesize  This property can be realized in our model: The size of increments of $Z$ over large time intervals is determined by the behavior of the kernel $g_0$ in \eqref{kernel:g} for large $t$, which is not further specified in our model. For instance, if $Z$ is a standard fBM with $H\in(0,\frac12)$, $Z_{s+t}-Z_s$ is of lower order than $X_{s+t}-X_s$ for large $t$, so the effect of noise is negligible.} As $\Den$ enters a high-frequency regime, noise becomes noticeable and even dominates when $\Den$ approaches a few seconds. Finally, at ultra-high frequency, eventually all trades are recorded tick by tick and both transaction times and observed prices become discrete. 

Without doubt, estimating volatility using tick-by-tick data (see, for example, \cite{Jacod19,Li14, Robert10,Robert12}) necessitates a careful modeling of rounding effects and bid--ask bounces in prices. However, as we can see from Figure~\ref{fig:path}, prices sampled at 1 second in our 2019 SPY data do not show much discreteness or flat periods as opposed  to, for example, a typical price path in 1999, which was before the decimalization on US stock exchanges. This is in agreement with our previous observation from Figure~\ref{fig1} (b) that price increments are still shrinking\footnote{\spacingset{1} \footnotesize  An important detail: to calculate the  variance of increments, we exclude periods of no observations (as they would artificially lower the  variance) but include zero returns between  identical observed prices.} at the frequencies we consider (rounding errors would induce a flattening in variance plots). As a result, rounding effects and bid--ask bounces do not seem to be the dominant source of noise in the data and at the frequency we consider.

\setcounter{figure}{7} 
\begin{figure}[htb!]
	\centering
	\begin{subfigure}{.5\textwidth}
		\centering
		\includegraphics[width=0.95\linewidth]{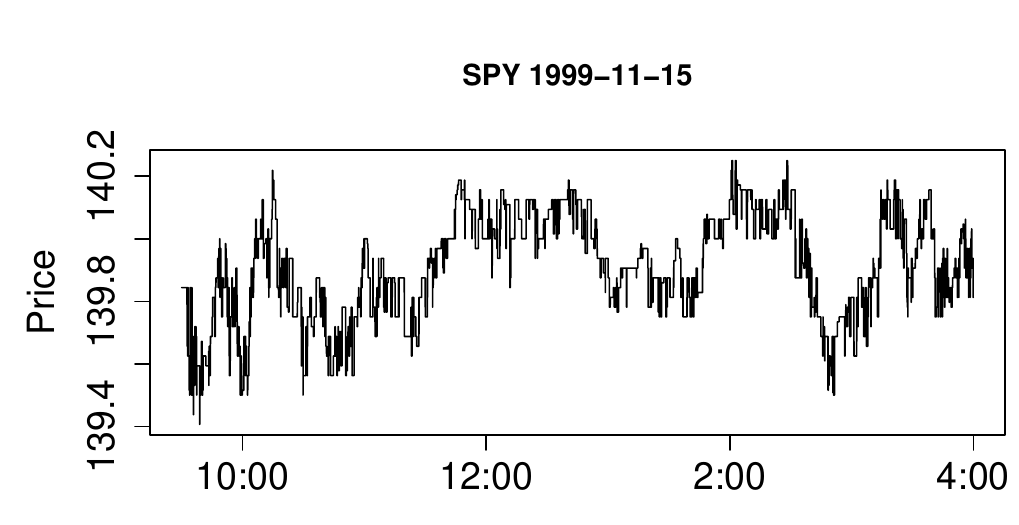}
	\end{subfigure}%
	\begin{subfigure}{.5\textwidth}
		\centering
		\includegraphics[width=0.95\linewidth]{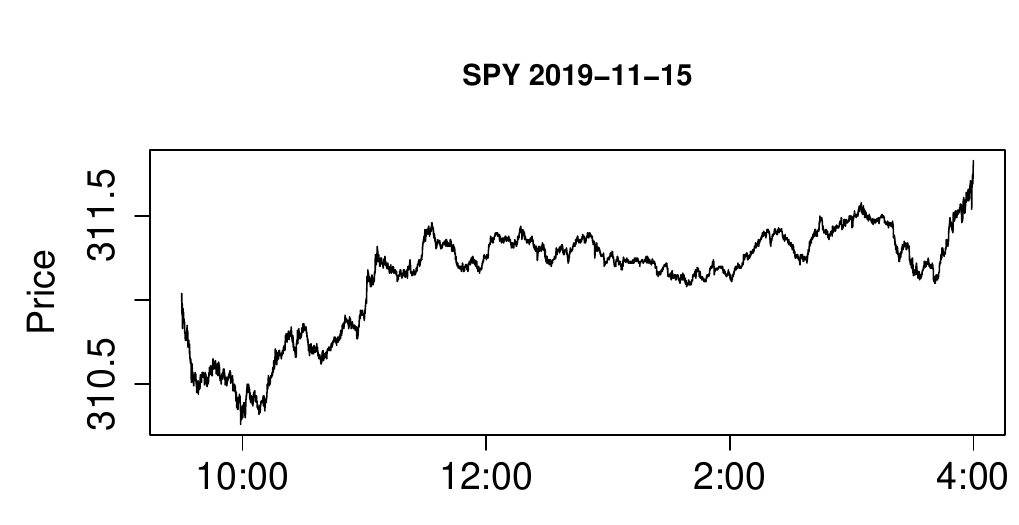}
	\end{subfigure}
	\caption{Two paths of SPY transaction prices, one from 1999 and one from 2019.}\label{fig:path}
\end{figure}

Next, we give two possible explanations for the existence of microstructure noise in continuous time. Both are related to the very reason why the efficient price  $X$ is typically assumed to be a semimartingale. First, according to the fundamental theorem of asset pricing, the absence of arbitrage in an idealized frictionless market implies that prices  must be semimartingales \citep{Delbaen94}. Real markets, of course, have transaction costs (e.g.,  bid--ask spreads and commissions). Transaction costs  do not only generate trade-specific noise  in the form of bid--ask bounces (as in the \cite{Roll84} model), but have the effect that the absence of arbitrage no longer implies the semimartingale property for prices. For example, both fBM and mfBM (which are special cases of our model) are known to not admit arbitrage in the presence of transaction costs  \citep{Cherny08, Guasoni08,Jarrow09}. In other words, even if noise due to trading mechanisms is taken away, transaction costs may lead to an additional continuous noise component.

Second, as shown by \cite{AitSahalia20},  many microscopic models of tick-by-tick data are  compatible (i.e., functionally converge in law to) macroscopic  semimartingale models as time is stretched out. In this framework, microstructure noise can be viewed as the difference between the limiting semimartingale process $X$ and the microscopic tick-by-tick observed price process $Y$ (which evolves as a continuous-time but piecewise constant process). In this approach, the microstructure noise process $Z=Y-X$ is, by definition, a continuous-time process. Moreover, since it bridges a microscopic model with a classical white or colored noise as in \eqref{eq:nonvan} (``$H=0$'') and a noise-free macroscopic model (``$H=\frac12$''), it seems reasonable to assume a locally fractional nature for $Z$ with some $H\in(0,\frac12)$.

Finally, let us remark that including both a discrete and a continuous noise component would probably yield the most satisfying solution; but this is beyond the scope of the current paper. Also, a theoretical substantiation of the arguments in the previous paragraph (e.g., by exhibiting a  tick-by-tick price model that converges to a mixed semimartingale on an intermediate time scale) remains open and is left to future research.

\section{A multivariate central limit theorem for variation functionals}\label{sec:mult}

Theorem~\ref{thm:CLT:mixedSM} can be extended to a multivariate setting that covers variation functionals of the form
\begin{equation*}
	V^n_f(Y, t) = \Delta_n \sum_{i = 1}^{[t/\Del] - L +1} f \bigg( \frac{\un{\Delta}_i^n Y}{\Den^H} \bigg),
\end{equation*}
where $f \colon \bbr^{d \times L} \to \bbr^M$ is some test function ($L, M \in \bbn$), $Y$ is a $d$-dimensional process and 
\begin{equation} \label{not:tensor}
	\begin{split}
		\Delta_i^n Y = Y_{i \Delta_n}- Y_{(i-1) \Delta_n} \in \bbr^d, \quad 
		\un{\Delta}^n_i Y = (\Delta_i^n Y, \Delta_{i+1}^n Y, \ldots, \Delta_{i+L-1}^n Y) \in \bbr^{d \times L}.
	\end{split}
\end{equation} 
In the following set of hypotheses, which is a direct multivariate extension of Assumption~\ref{Ass:A-1}, $\Vert \cdot \Vert$ denotes the Euclidean norm (in $\bbr^n$ if applied to vectors and in $\bbr^{nm}$ if applied to a matrix in $\bbr^{n \times m}$).
\settheoremtag{(CLT$_d$)}
\begin{Assumption} \label{Ass:A} The observation process $Y$ is given by the sum of $X$ from \eqref{eq:X} and $Z$ from \eqref{eq:Z} with the following specifications:
	\begin{enumerate}
		\item[(i)]
		The function $f \colon \bbr^{d \times L} \to \bbr^M$ is even and infinitely differentiable. Moreover, all its derivatives (including $f$ itself) have at most polynomial growth.
		\item[(ii)] Both $B$ and $W$ are independent  standard $\bbf$-Brownian motions in $\R^d$, the drift $a$ is $d$-dimensional, locally bounded and $\F$-adapted, and $\si$ is an $\F$-adapted locally bounded   $\R^{d\times d}$-valued process such that for every $T>0$, there is $K_1\in(0,\infty)$ with
		\begin{equation} 
			\bbe \Big[ 1\wedge  \Vert \si_t - \si_s   \Vert  \Big] \leq K_1 \vert t -s \vert^{\frac{1}{2}},\qquad s,t\in[0,T].
		\end{equation}
		\item[(iii)] The noise volatility process $\rho$ takes the form
		\begin{equation} \label{repr:si-2}
			\rho_t = \rho^{(0)}_t + \int_{0}^{t} \wti{b}_s \, \dd s + \int_{0}^{t} \wti{\rho}_s \, \dd \wt W_s,\qquad t \geq 0.
		\end{equation}
		In \eqref{repr:si-2},  $\wt W$ is standard $\bbf$-Brownian motion in $\R^d$ that is jointly Gaussian with $(B,W)$;  $\wti{b}$ is $d \times d$-dimensional, locally bounded and $\F$-adapted;  $\rho^{(0)}$ is an $\F$-adapted locally bounded $\bbr^{d \times d}$-valued process such that for all $T>0$,
		\begin{equation} \label{mom:ass:si:rho:0-2}
			\bbe \Big[ 1\wedge\Vert \rho_t^{(0)} - \rho_s^{(0)}  \Vert  \Big]  \leq K_2\vert t - s \vert^{\ga},\qquad s,t\in[0,T],
		\end{equation}
		for some $\ga \in  ( \frac{1}{2}, 1  ]$ and $K_2\in(0,\infty)$; and $\wti{\rho}$ is an $\F$-adapted locally bounded $\bbr^{d \times d \times d}$-valued process (e.g., the $(ij)$th component of the stochastic integral in \eqref{repr:si-2} equals $\sum_{k=1}^{d} \int_{0}^{t}\wti{\rho}^{ijk}_s \, \dd \wt W^k_s$) such that for all $T>0$, there exist   $\eps > 0$ and $K_3\in(0,\infty)$ with	
		\begin{equation} \label{reg:ass:si:ti-2}
			\bbe \Big[ 1\wedge\Vert \wti{\rho}_t - \wti{\rho}_s  \Vert   \Big]  \leq K_3 \vert t -s \vert^{\eps},\qquad s,t\in[0,T].
		\end{equation}
		\item[(iv)] We have \eqref{kernel:g} with $H\in(0,\frac12)$ and some $g_0\in C^\infty([0,\infty))$ with $g_0(0) = 0$.	
	\end{enumerate}
\end{Assumption}

Before we can state the multivariate extension of Theorem~\ref{thm:CLT:mixedSM-1}, we need   some more notation.
Define $\mu_f$ as the $\bbr^{M}$-valued function that maps  $v=(v_{k \ell, k' \ell'})\in(\R^{d\times L})^2$ to $\mathbb{E}[f(\calz)]$ where $\calz\in \R^{d\times L}$ follows a multivariate normal distribution with mean $0$ and $\cov(\calz_{k \ell}, \calz_{k' \ell'}) = v_{k \ell, k' \ell'}$. Note that $\mu_f$ is infinitely differentiable because $f$ is. Furthermore, if $\calz'\in\R^{d\times L}$ is such that $\calz$ and $\calz'$ are jointly Gaussian with mean $0$, covariances 
$\cov (\calz_{k\ell}, \calz_{k'\ell'}) = \cov (\calz'_{k\ell}, \calz'_{k'\ell'}) = v_{k\ell, k'\ell'}$ and cross-covariances
$\cov (\calz_{k\ell}, \calz'_{k'\ell'}) = q_{k\ell, k'\ell'}$, we define
$
\ga_{f_{m_1}, f_{m_2}}(v,q) = \cov   ( f_{m_1}(\calz), f_{m_2}(\calz')   )$ for $m_1, m_2= 1, \ldots, M.
$
We further introduce a multi-index notation adapted to the definition of $\mu_f$. For $\chi = ({\chi}_{k\ell,k'\ell'}) \in 
\bbn_0^{(d \times L) \times (d \times L)}$ and $v$ as above,   let
$
\lvert \chi \rvert = \sum_{k,k' = 1}^d \sum_{\ell,\ell' = 1}^{L} \chi_{k \ell,k'\ell'}$,  $\chi! = \prod_{k,k' = 1}^d \prod_{\ell,\ell' = 1}^{L} \chi_{k \ell,k'\ell'}!$,
$v^{\chi} = \prod_{k,k' = 1}^d \prod_{\ell,\ell' = 1}^{L} {v_{k \ell,k'\ell'}}^{{\chi}_{k \ell,k'\ell'}}$ and $\pd^{\chi} \mu_f = \frac{\pd^{\vert \chi \vert} \mu_f}{\pd v_{11,11}^{{\chi}_{11,11}} \cdots \pd v_{dL,dL}^{{\chi}_{dL,dL}}}$.
Finally, recall \eqref{num:Ga} and define,  for all $k, k' \in \{1, \ldots, d\}$,  $\ell, \ell' \in \{1, \ldots, L\}$ and $r \in \bbn_0$, 
\begin{equation} \label{mat:lim:thms:2}
	\pi_r(s)_{k \ell, k' \ell'}  = (\rho_s \rho_s^T)_{k k'} \Ga^H_{\vert \ell - \ell' + r \vert},\quad
	c(s)_{k \ell, k' \ell'}  = (\si_s \si_s^T)_{k k'} \ind_{\{\ell = \ell'\}},\quad
	\pi(s) = \pi_0(s).\!
\end{equation}

\begin{Theorem} \label{thm:CLT:mixedSM}
	Grant Assumption~\ref{Ass:A} and
	let $N(H)=[1/(2-4H)]$. Then  
	\begin{equation} \label{CLT:2} 
		\Delta_n^{- \frac{1}{2}} \Bigg\{ V^n_f(Y,t)
		- V_f(Y,t)  
		- \sum_{j=1}^{N(H)} \Den^{j (1 - 2H)}  
		\sum_{\vert \chi \vert = j} \frac1{\chi!}\int_{0}^{t} \pd^{\chi} \mu_f  (\pi(s) ) {c(s)}^{\chi} \, \dd s \Bigg \} 
		\stackrel{\mathrm{st}}{\Longrightarrow} \mathcal{Z},
	\end{equation}
	where 
	\beq\label{eq:Vft} V_f(Y,t)=\int_0^t \mu_f(\pi(s))\,\dd s\eeq
	and $\mathcal{Z} = (\mathcal{Z}_t)_{t \geq 0}$ is an $\bbr^M$-valued continuous process defined on a very good filtered extension $(\ov{\Om}, \ov{\mathcal{F}}, (\ov{\mathcal{F}}_t)_{t \geq 0}, \ov{\mathbb{P}})$ of   $(\Om, \mathcal{F}, (\mathcal{F}_t)_{t \geq 0}, \mathbb{P})$ which, conditionally on  $\mathcal{F}$, is a centered Gaussian process with independent increments and such that the conditional covariance function 
	$\calc^{m_1 m_2}_t = \ov{\bbe}[\calz^{m_1}_t \calz^{m_2}_t \mid \calf]$, for $m_1, m_2 = 1, \ldots, M$, is given by
	\begin{equation} \label{cov:fct:lim:CLT}
		\calc^{m_1 m_2}_t =\int_{0}^{t} \bigg \{  \ga_{f_{m_1}, f_{m_2}}(\pi(s), \pi(s)) 
		+ \sum_{r=1}^{\infty} \Big( \ga_{f_{m_1}, f_{m_2}} 
		+ \ga_{f_{m_2}, f_{m_1}} \Big) (\pi(s), \pi_r(s)) \bigg \} \,\dd s.
	\end{equation}
\end{Theorem}

\section{Proof of Theorem~\ref{thm:CLT:mixedSM}}\label{Sect:proof}
\subsection{Size estimates}\label{SectA}
We use the notation from the main paper. In addition, we write $A\lec B$ if there is a constant $C$ that is independent of any quantity of interest such that $A\leq CB$.
In the following, we  repeatedly make use of so-called \emph{standard size estimates}  (cf.\ \cite{Chong20supp}, Appendix D).  Under the strengthened hypotheses of Assumption~\ref{Ass:A'}, consider 
for fixed  $j,k\in \{1, \ldots,d\}$ and $\ell \in \{1, \ldots, L\}$ an expression like
\begin{equation} \label{SSE:mod}
	\begin{split}
		S_n(t) & = \Den^{\frac{1}{2}} \sum_{i = \theta_n + 1}^{[t/\Delta_n]} h(\zeta_i^n) 
		\Bigg ( \frac{\Delta^n_{i+\ell -1} A^k}{\Den^H} 
		+ \frac{1}{\Den^H} \int_{(i+\ell - 2) \Delta_n}^{(i+\ell - 1) \Delta_n} \Big( \si_s^{kj} - \si^{kj}_{(i-\theta''_n) \Den} \Big) \, \dd B_s^j 
		\\ & \qquad \qquad + 
		\int_0^{\infty} \frac{\Delta^n_{i+\ell -1} g(s)}{\Den^H} 
		\Big( \rho_s^{kj} - \rho^{kj}_{(i-\theta_n) \Den}\Big  ) \ind_{ ((i-\theta_n)\Den, (i-\theta'_n)\Den) }(s) \, \dd W^j_s \Bigg),
	\end{split}
\end{equation}
where 
$\theta_n = [\Delta_n^{-\theta}]$, $\theta_n' = [\Delta_n^{-\theta'}]$, $\theta_n'' = [\Delta_n^{-\theta''}]$ and $- \infty \leq \theta' ,\theta''< \theta \leq \infty$. In addition, $h$ is a function such that
$\lvert h(x) \rvert \lesssim 1 + \Vert x\Vert^{p }$ for some $p >1$, and $\zeta^n_i$ are random variables with
\begin{equation*}
	\sup_{n \in \bbn} \sup_{i=1, \ldots, [T/\Den]} \bbe [\Vert \zeta^n_i \Vert^p] < \infty.
\end{equation*}
For any $q\geq1$, because $a$ is uniformly bounded by Assumption~\ref{Ass:A'}, Minkowski's integral inequality yields
\begin{equation} \label{SSE:A'}
	\bbe \bigg[ \bigg \Vert \frac{\Delta^n_{i+\ell - 1} A}{\Den^H} \bigg \Vert^q \bigg]^{\frac{1}{q}} \leq 
	\frac{1}{\Den^H} \int_{(i+\ell - 2)\Den}^{(i+\ell - 1) \Den} \bbe  [ \Vert a_s \Vert^q  ]^{\frac{1}{q}} \, \dd s \lesssim \Den^{1-H}.
\end{equation}
Similarly, by the Burkholder--Davis--Gundy (BDG) inequality and Assumption~\ref{Ass:A'}, 
\begin{equation} \label{SSE:BM}
	\begin{split}
		& \bbe \bigg[ \bigg \vert \frac{1}{\Den^H} \int_{(i+\ell - 2) \Delta_n}^{(i+\ell - 1) \Delta_n} \Big( \si^{kj}_s - \si^{kj}_{(i-\theta''_n) \Den} \Big) \, \dd B^{j}_s  \bigg \vert^q \bigg]^{\frac{1}{q}} \lesssim (\theta''_n \Den)^{\frac{1}{2}} \Den^{\frac{1}{2}-H}.
	\end{split}
\end{equation}
Combining  Assumption~\ref{Ass:A'} with Lemma~\ref{lem:ker}, we deduce that
\begin{equation} \label{SSE:fBM}
	\begin{split}
		& \bbe \bigg[ \bigg \vert \int_0^{\infty} \frac{\Delta^n_{i+\ell -1} g(s)}{\Den^H} 
		\Big( \rho_s^{kj} - \rho^{kj}_{(i-\theta_n) \Den} \Big) \ind_{ ((i-\theta_n)\Den, (i-\theta'_n)\Den) }(s) \, \dd W^j_s \bigg \vert^q \bigg]^{\frac{1}{q}}   
		\\ & \qquad
		\lesssim  (\theta_n \Den  )^{\frac{1}{2}}
		\bigg(\frac{1}{\Den^{2H}} \int_0^{(i - \theta_n') \Den} \Delta^n_{i+\ell -1} g(s)^2 \, \dd s \bigg)^{\frac{1}{2}} \lesssim
		(\theta_n \Den  )^{\frac{1}{2}} \Den^{\theta'(1-H)}.
	\end{split}
\end{equation}	
Finally, using H\"older's inequality to separate $h(\zeta^n_i)$ from the subsequent expression in \eqref{SSE:mod}, we have shown that 
\begin{equation} \label{SSE:all}
	\begin{split}
		\bbe \bigg[ \sup_{t \leq T} \big \vert S_n(t) \big \vert \bigg] & \lesssim 
		\DenOneHalf \sum_{i=\theta_n +1}^{[T/\Den]}  \Big\{ \Den^{1-H} + \Den^{1-H} (\theta''_n)^{\frac{1}{2}} 
		+  (\theta_n \Den  )^{\frac{1}{2}} \Den^{\theta'(1-H)}
		\Big\} \\
		&\lesssim
		\Den^{\frac12-H} + \Den^{\frac{1}{2}-H - \frac{\theta''}{2}} + \Den^{\theta'(1-H) - \theta}.
	\end{split}
\end{equation}

The upshot of this example is that the absolute moments of sums and products of more or less complicated expressions can always be bounded term by term: for example, in \eqref{SSE:mod}, the terms
\begin{align*} &\sum_{i = \theta_n + 1}^{[t/\Delta_n]},\qquad h(\zeta^n_i),\qquad \Delta^n_{i+\ell-1} A^k,\qquad \int_{(i+\ell - 2) \Delta_n}^{(i+\ell - 1) \Delta_n} (\cdots) \, \dd B_s^j, \qquad   \si_s^{kj} - \si^{kj}_{(i-\theta''_n) \Den},\\
	& \int_0^{(i-\theta'_n)\Den} \frac{\Delta^n_{i+\ell -1} g(s)}{\Den^H} (\cdots)\, \dd W^j_s,\qquad 
	\rho_s^{kj} - \rho^{kj}_{(i-\theta_n) \Den}    \end{align*}
have \emph{sizes} (i.e., the $L^q$-moments, for any $q$, are uniformly bounded by a constant times)
$$ \Den^{-1},\qquad 1,\qquad \Den,\qquad \sqrt{\Den},\qquad (\theta''_n\Den)^{\frac12},\qquad \Den^{\theta'(1-H)},\qquad (\theta_n\Del)^{\frac12}, $$
respectively. The final estimate \eqref{SSE:all} is then obtained by combining these bounds. Clearly, size estimates can be applied to variants of \eqref{SSE:mod}, too, for example, when   the stochastic integral in \eqref{SSE:mod} is squared, when we have products of integrals, when  $S_n(t)$ is matrix-valued, etc. 

Even though size estimates are optimal in general, better estimates may be available in specific cases. One such case occurs when  sums have a martingale structure. To illustrate this, let $\calf^n_i = \calf_{i \Delta_n}$ and consider  
\begin{equation*}
	S_n'(t) = \Den^{\frac{1}{2}} \sum_{i=1}^{[t/\Den]-L+1} \varpi^n_i
\end{equation*}
with random variables $\varpi^n_i$ that are $\calf^n_i$-measurable and satisfy $\bbe[\varpi^n_i \mid \calf^n_{i - \theta'''_n}]=0$, where $\theta'''_n = [\Den^{-\theta'''}]$ for some $0 < \theta''' < 1$. Suppose that 
$\bbe [\vert \varpi^n_i \vert^2]^{1/2} \lesssim \Den^{\varpi}$ uniformly in $i$ and $n$ for some  $\varpi > 0$. Writing
$$ S'_n(t)=\sum_{j=1}^{\theta'''_n} S'_{n,j}(t),\qquad S'_{n,j}(t)=\Delta_n^{\frac12} \sum_{k=1}^{[([t/\Delta_n]-L+1)/\theta'''_n]} \varpi^n_{j+(k-1)\theta'''_n},$$
we observe that each $S'_{n,j}$ is a martingale in $t$ (albeit relative to different filtrations), so the BDG inequality   and the triangle inequality yield
\begin{equation} \label{MSE}
	\bbe \bigg[ \sup_{t \leq T} \vert S_n'(t) \vert \bigg] \lesssim  (\theta'''_n)^{\frac{1}{2}} \Den^{\varpi}.
\end{equation}
Very often, $\varpi^n_i$ will actually only be $\calf^n_{i + L - 1}$-measurable. However, a shift by $L$ increments will not change the value of the above estimate. Following \cite{Chong20a}, Section~4, we refer to \eqref{MSE} as a \emph{martingale size estimate}.

\subsection{Estimates for fractional kernels}\label{SecD}

Here we gather some useful results about the kernel $g(t)=K_H^{-1}t^{H-1/2}$ introduced in \eqref{kernel:g} (we consider the case $g_0\equiv0$ here). 
We write $g(t)=0$ for $t\leq 0$ and define for all $s, t \geq 0$ and $i, n \in \bbn$,
\begin{equation} \label{def:ker}
	\begin{split}
		\Delta^n_i g(s) & = g(i \Delta_n-s) - g((i - 1) \Delta_n-s),\\ 
		\uDeni g(s)  &=  ( \Deni g (s), \ldots, \Delta^n_{i + L - 1} g (s)  ), 
	\end{split}
\end{equation}
\begin{Lemma} \label{lem:ker}
	Recall \eqref{num:Ga}.
	\begin{enumerate}
		\item[(i)]  For any $k, n \in \bbn$,
		\begin{equation}  \label{est:ker:1}
			\begin{split}
				\int_0^{\infty} \Delta^n_k g(t)^2 \, \dd t & = K_H^{-2} \bigg\{ \frac{1}{2H} + \int_{1}^{k}  \Big( r^{H -\frac{1}{2}} - (r-1)^{H -\frac{1}{2}} \Big)^2 \, \dd r \bigg\} \Delta_n^{2H} \leq \Delta_n^{2H}.
			\end{split}
		\end{equation}
		\item[(ii)] For any $k, \ell, n \in \bbn$ with $k < \ell$,
		\begin{equation} \label{est:ker:2}
			\int_{- \infty}^{\infty} \Delta^n_k g(t) \Delta^n_\ell g(t) \, \dd t = 
			\Den^{2H} \Ga^H_{\ell - k}	
			\lesssim \Delta_n^{2H} \ov{\Ga}^H_{\ell-k},
		\end{equation}
		where $\ov{\Ga}^H_1 = \Ga^H_1$ and 
		$
		\ov{\Ga}^H_{r} =  (r - 1)^{-2(1-H)}
		$ for $r\geq2$.
		\item[(iii)] For any $\theta \in (0,1)$, setting $\theta_n = [\Delta_n^{-\theta}]$, we have for any $i > \theta_n$ and $r \in \bbn$,
		\begin{equation} \label{est:ker:tr}
			\int_{-\infty}^{(i - \theta_n) \Den} \Delta^n_i g(s) \Delta^n_{i + r} g(s) \, \dd s \lesssim \Den^{2H} \Den^{2\theta(1 - H)}.
		\end{equation}
	\end{enumerate}	
	
\end{Lemma}

\begin{proof}
	Let $k \leq \ell$. By direct calculation,
	\begin{align*}
		& \int_0^{\infty} \Delta^n_k g(t) \Delta^n_\ell g(t) \, \dd t  \\ 
		& \qquad =
		\Den^{2H} K_H^{-2} 
		\int_{0}^{k}  \Big(r^{H - \frac{1}{2}} - (r - 1)_+^{H - \frac{1}{2}} \Big) 
		\Big( (r + (\ell - k))^{H - \frac{1}{2}} - (r + (\ell - k) - 1)_+^{H - \frac{1}{2}} \Big) \, \dd r,
	\end{align*}
	which shows \eqref{est:ker:1} by setting $k=\ell$.
	Next, let $(B^H)_{t \geq 0}$ be a fractional Brownian motion with Hurst index $H$. Then $B^H$ has the Mandelbrot--van Ness representation
	\begin{equation*}
		B^H_t = K_H^{-1} \int_{\bbr} \Big((t-s)_+^{H-\frac{1}{2}} - (-s)_+^{H-\frac{1}{2}}\Big) \, \dd \ov{B}_s, \quad t \geq 0,
	\end{equation*}	
	where $\ov{B}$ is a two-sided standard Brownian motion. Moreover, 
	$
	\Deni B^H = \int_{\bbr} \Deni g(s) \, \dd \ov{B}_s
	$
	for any $i$. Therefore, by well-known properties of fractional Brownian motion,
	\begin{equation*}
		\begin{split}
			\int_{- \infty}^{\infty} \Delta^n_k g(s) \Delta^n_\ell g(s) \, \dd s &= \bbe  [ \Delta^n_k B^H \Delta^n_\ell B^H  ]    = 
			\bbe  [ B^H_{\Delta_n} B^H_{(\ell - k + 1)\Delta_n}  ] - \bbe  [ B^H_{\Delta_n} B^H_{(\ell - k) \Delta_n}  ] \\ &  
			=\mathtoolsset{multlined-width=0.7\displaywidth} \begin{multlined}[t] \frac{1}{2 } \Big \{ \Delta_n^{2H} + ((\ell - k + 1) \Delta_n)^{2H} - ((\ell - k) \Delta_n)^{2H} \\  -
				\Delta_n^{2H} - ((\ell - k) \Delta_n)^{2H} + ((\ell - k-1) \Delta_n)^{2H} \Big \} \end{multlined}\\
			&=
			\Delta_n^{2H} \Ga^H_{\ell- k},
		\end{split}
	\end{equation*}
	which is the  equality in \eqref{est:ker:2}. Next, use the mean-value theorem twice on $\Ga^H_r$ in order to obtain for all $r \geq 2$,
	\begin{equation*}
		\begin{split}
			\Ga^H_r & = \frac{1}{2} \Big(  \{ (r + 1)^{2H} -  r^{2H}  \} -  \{  r^{2H} - (r - 1)^{2H}  \} \Big) \leq 
			\frac{1}{2}  (2H )  \Big( 
			(r+1)^{2H-1} - (r-1)^{2H-1} \Big)  \\ & \leq 
			H (2H-1) (r-1)^{2H-2},
		\end{split}
	\end{equation*}
	which shows the inequality in  \eqref{est:ker:2}. Finally, 
	\begin{align*}
		& \int_{-\infty}^{(i - \theta_n) \Den} \Delta^n_i g(s) \Delta^n_{i + r} g(s) \, \dd s \\ & \qquad =
		\Den^{2H} K_H^{-2} \int_{\theta_n}^{\infty}
		\Big(t^{H-\frac{1}{2}} - (t - 1)^{H-\frac{1}{2}}\Big) 
		\Big( (t + r)^{H-\frac{1}{2}} - (t + r - 1)^{H-\frac{1}{2}} \Big) \, \dd t \\ & \qquad \lesssim
		\Den^{2H} \int_{\theta_n}^{\infty}
		\Big(t^{H-\frac{1}{2}} - (t - 1)^{H-\frac{1}{2}}\Big)^2 \, \dd t \lesssim 
		\Den^{2H} \int_{\theta_n}^{\infty} (t - 1)^{2H-3} \, \dd t \lesssim
		\Den^{2H} \Den^{\theta  ( 2-2H  )},
	\end{align*}	
	which yields \eqref{est:ker:tr}.
\end{proof}

\subsection{Overview of the proof of Theorem~\ref{thm:CLT:mixedSM}}\label{SectB}

Throughout the proof, by a standard localization argument (cf.\ Lemma~4.4.9 in \cite{JP}), we may and will assume a strengthened version of Assumption~\ref{Ass:A}:
\settheoremtag{(CLT$_d^\prime$)}
\begin{Assumption} \label{Ass:A'} In addition to Assumption~\ref{Ass:A}, there is  $C>0$ such that 
	$$ \sup_{(\om,t)\in\Om\times[0,\infty)} \bigg\{ \lVert a_t(\om)\rVert + \lVert \si_t(\om)\rVert + \lVert \rho_t(\om)\rVert + \lVert \rho^{(0)}_t(\om)\rVert + \lVert \wt b_t(\om)\rVert + \lVert \wt \rho_t(\om)\rVert  \bigg\}<C.  $$
	Moreover, for every $p>0$, there is $C_p>0$ such that for all $s,t>0$,
	\beq\label{eq:Hoelder}\begin{split} \E[ \lVert \si_t-\si_s\rVert^p ]^{\frac 1p}&\leq C_p\lvert t-s\rvert^{\frac12},\qquad  \E[ \lVert \rho^{(0)}_t-\rho^{(0)}_s\rVert^p ]^{\frac 1p}\leq C_p\lvert t-s\rvert^{\ga},\\
		\E[ \lVert \wt\rho_t-\wt\rho_s\rVert^p ]^{\frac 1p}&\leq C_p\lvert t-s\rvert^{\eps}. \end{split}\eeq
\end{Assumption}

\begin{proof}[Proof of Theorem~\ref{thm:CLT:mixedSM}]
	Except for \eqref{CLT:fBM} below, we may and will assume that $M = 1$.
	Recalling the decomposition \eqref{kernel:g}, since $g_0$ is smooth with $g_0(0)=0$, we can use the stochastic Fubini theorem (see \cite{Protter05},  Chapter IV, Theorem 65) to write
	\begin{equation*}
		\int_{0}^{t} g_0(t-r) \rho_r \, \dd W_r  = 
		\int_{0}^{t} \bigg( \int_{r}^{t} g_0'(s-r) \, \dd s  \bigg) \rho_r \,  \dd W_r  
		= \int_{0}^{t} \bigg( \int_{0}^{s} g_0'(s-r) \rho_r \, \dd W_r \bigg) \dd s.
	\end{equation*}
	This is a finite variation process and can be incorporated in the drift process in \eqref{mix:SM:mod}. So without loss of generality, we may assume $g_0 \equiv 0$ and $g(t)=K_H^{-1}t^{H-1/2}$ in the following. Then 
	$
	Y_t = A_t + M_t + Z_t$, where 	$A_t = \int_{0}^{t} a_s \, \dd s$ and $M_t = \int_{0}^{t} \si_s \, \dd B_s$,
	and we have $ \un{\Delta}^n_i Y = \un{\Delta}^n_i A + \un{\Delta}^n_i M + \un{\Delta}^n_i Z$ in the notation of \eqref{not:tensor}.  Recall \eqref{def:ker}, we have, in matrix notation,
	\begin{equation*}
		\un{\Delta}^n_i Z =  \bigg(\int_0^{\infty} \Deni g(s) \rho_s \, \dd W_s, , \ldots, 
		\int_0^{\infty} \Delta^n_{i + L - 1} g (s) \rho_s \, \dd W_s \bigg) 
		= \int_0^{\infty} \rho_s \, \dd W_s \, \uDeni g(s).
	\end{equation*}

	The first step in our proof is to shrink the domain of integration for each $\un{\Delta}^n_i Z$. Let 	\begin{equation} \label{theta:value}
		\theta\in(	\tfrac{1}{4(1-H)} ,  \tfrac{1}{2}),
	\end{equation}
	which is always possible for $H\in(0,\frac12)$,
	and set 
	$\theta_n = [\Delta_n^{-\theta}]$. 
	Further define  
	\begin{equation} \label{1st:tr}
		\un{\Delta}^n_i Y^{\mathrm{tr}} = \un{\Delta}^n_i A + \un{\Delta}^n_i M + \xi_i^{n}, \qquad
		\xi_i^n = \int_{(i - \theta_n)\Delta_n}^{(i + L - 1) \Delta_n} \rho_s \, \dd W_s \, \uDeni g(s).
	\end{equation}
	\blem\label{lem:app:2}
	If $\theta$ is chosen according to \eqref{theta:value}, then 
	\[ \Del^{-\frac12}\bigg\{V^n_f(Y,t)-\Del\sum_{i = \theta_n + 1}^{[t/\Den]-L+1} f \bigg( \frac{\un{\Delta}_i^n Y^{\tr} }{\DenH} \bigg)\bigg\}\stackrel{L^1}{\Longrightarrow}0. \]
	\elem

	The last sum can be further decomposed into three parts:
	\begin{equation} \label{dec:fBM}
		\Delta_n^{\frac{1}{2}} \sum_{i = \theta_n + 1}^{[t/\Den]-L+1} f \bigg( \frac{\un{\Delta}_i^n Y^{\tr} }{\DenH} \bigg) 
		= V^n(t) + U^n(t) + \Delta_n^{\frac{1}{2}} \sum_{i = \theta_n + 1}^{[t/\Den]-L+1} 
		\bbe \bigg[f \bigg( \frac{\un{\Delta}_i^n Y^{\tr} }{\DenH} \bigg) \mathrel{\Big\vert} \calf^n_{i - \theta_n} \bigg], 
	\end{equation}
	where 
	\begin{equation*}
		\begin{split}
			& V^n(t) = \sum_{i=\theta_n + 1}^{[t/\Den]-L+1} \Xi^n_i, \qquad 
			\Xi^n_i = \DenOneHalf \bigg( f \bigg( \frac{ \xi^n_i }{\DenH} \bigg) - 
			\bbe \bigg[f \bigg( \frac{ \xi^n_i }{\DenH} \bigg) \mathrel{\Big\vert}  \calf^n_{i - \theta_n} \bigg] \bigg), \\
			& U^n(t) =  \DenOneHalf \sum_{i=\theta_n + 1}^{[t/\Den]-L+1} \Bigg\{ 
			f \bigg( \frac{\un{\Delta}_i^n Y^{\tr}}{\DenH} \bigg) - f \bigg( \frac{ \xi^n_i }{\DenH} \bigg) - 
			\bbe \bigg[f \bigg( \frac{\un{\Delta}_i^n Y^{\tr}}{\DenH} \bigg) - f \bigg( \frac{ \xi^n_i }{\DenH} \bigg) \mathrel{\Big\vert}  \calf^n_{i - \theta_n} \bigg] \Bigg\}. 
		\end{split}
	\end{equation*}
	\begin{Lemma} \label{approx:lem:1}
		For all $H < \frac12$, we have that $U^n\limL0$.
	\end{Lemma}
	In other words, in the limit $\Del\to0$, the impact of the semimartingale component is negligible, except for its contributions to the conditional expectations in \eqref{dec:fBM}. As we mentioned above, this is somewhat surprising: It is true that the $L^2$-norm of the semimartingale increment $\un\Delta^n_i A+\un\Delta^n_i M$, divided by $\Del^H$, converges to $0$. But the rate $\Del^{1/2-H}$ at which this takes place can be arbitrarily slow if $H$ is close to $\frac12$. So Lemma~\ref{approx:lem:1} implies that there is a big gain in convergence rate if one considers the sum of the centered differences $f(\un\Delta^n_i Y^\tr/\Del^H)-f(\xi^n_i/\Del^H)$. In the proof, we will need for the first time that $f$ has at least $2(N(H)+1)$ continuous derivatives.
	
	The process $V^n$ only contains the fractional part and is responsible for the limit $\calz$ in \eqref{CLT:2}. For the sake of brevity, we borrow a result from \cite{Chong20}: For each $m \in \bbn$, consider the sums 
	\begin{equation*}
		\begin{split}
			V^{n,m,1}(t) & = \sum_{j=1}^{J^{n,m}(t)} V^{n,m}_j, \qquad  
			V^{n,m}_j = \sum_{k=1}^{m \theta_n} \Xi^n_{(j-1)((m+1)\theta_n + L - 1) + k}, \\
			V^{n,m,2}(t) & = \sum_{j=1}^{J^{n,m}(t)} \sum_{k=1}^{\theta_n + L - 1} \Xi^n_{(j-1)((m+1)\theta_n + L - 1) + m\theta_n + k}, \\
			V^{n,m,3}(t) & = \sum_{j = ((m+1)\theta_n + L - 1)J^{n,m}(t) + 1 }^{[t/\Den]-L+1} \Xi_j^n,
		\end{split}
	\end{equation*}
	where $J^{n,m}(t) = [ ([t/\Den]-L+1)/((m+1)\theta_n + L - 1) ]$. We then have $V^n(t) = \sum_{i=1}^{3} V^{n,m,i}(t)$. This is very similar to the decomposition on p.\ 1161 in \cite{Chong20}. With essentially the same proof, we  infer that
	$V^n(t) \stabKonv \calz$ and, hence, 
	\begin{equation} \label{CLT:fBM}
		\Delta_n^{\frac{1}{2}} \Bigg\{   \sum_{i = 1}^{[t/\Den]-L+1} f \bigg( \frac{\un{\Delta}_i^n Y}{\Delta_n^H} \bigg) 
		-   \sum_{i = \theta_n + 1}^{[t/\Den]-L+1} 
		\bbe \bigg[f \bigg( \frac{\un{\Delta}_i^n Y^{\tr} }{\DenH} \bigg) \mathrel{\Big|} \calf^n_{i - \theta_n} \bigg] \Bigg\} \stabKonv \calz,
	\end{equation}
	where $\calz$ is exactly as in \eqref{CLT:2}. Therefore, in order to complete the proof of Theorem~\ref{thm:CLT:mixedSM}, it remains to show that (recall $N(H) = [1/(2-4H)]$)
	\begin{equation*} 
		\begin{split}
			& \Delta_n^{-\frac{1}{2}} \Bigg \{ \Delta_n \sum_{i = \theta_n + 1}^{[t/\Den]-L+1} 
			\bbe \bigg[f \bigg( \frac{\un{\Delta}_i^n Y^{\tr} }{\DenH} \bigg) \mathrel{\Big|} \calf^n_{i - \theta_n} \bigg] 
			- \int_{0}^{t} \mu_f (\pi(s)) \, \dd s \\ & \qquad \qquad
			- \sum_{j=1}^{N(H)} \Den^{j (1 - 2H)}  
			\sum_{\vert \chi \vert = j} \frac{1}{\chi!} \int_{0}^{t} \pd^{\chi} \mu_f (\pi(s) ) {c(s)}^{\chi} \, \dd s  
			\Bigg \} \LeinsKonv 0.
		\end{split}
	\end{equation*}
	To this end, we will discretize the volatility processes $\si$ and $\rho$ in $\un{\Delta}_i^n Y^{\tr}$. The proof is technical (as it involves another multiscale analysis) and will be divided into further smaller steps in Appendix~\ref{sect:details}. 
	\blem\label{lem:approx:6} Assuming \eqref{theta:value}, we have that 
	\begin{equation*}
		\DenOneHalf \sum_{i=\theta_n + 1}^{[t/\Den]-L+1} \bigg\{ \bbe \bigg[f  \bigg(\frac{\un{\Delta}_i^n Y^{\tr}}{\DenH } \bigg) \mathrel{\Big|} \calf^n_{i - \theta_n} \bigg] 
		- \mu_f ( \Upsilon^{n,i} ) \bigg\} \LeinsKonv 0,
	\end{equation*}		
	where $\Upsilon^{n,i} \in (\bbr^{d \times L})^2$ is defined by
	\begin{equation} \label{cov:v1}
		\begin{split}
			(\Upsilon^{n,i})_{k\ell,k'\ell'} & =
			c( (i-1)\Den )_{k\ell,k'\ell'} \,\Delta_n^{1-2H} \\ & \quad +
			( \rho_{(i-1)\Den} \rho_{(i-1)\Den}^T  )_{k k'}
			\int_{(i - \theta_n)\Delta_n}^{(i + L - 1) \Delta_n} \frac{\Delta^n_{i + \ell- 1} g(s) \Delta^n_{i + \ell' - 1} g(s)}{\Den^{2H}}  \,  \dd s.
		\end{split}
	\end{equation}	
	\elem

	The last part of the proof consists of evaluating $$\DenOneHalf \sum_{i=\theta_n + 1}^{[t/\Den]-L+1} \mu_f ( \Upsilon^{n,i} ).$$ This is the place where the asymptotic bias terms arise and which is different from the pure (semimartingale or fractional) cases. Roughly speaking, the additional terms are due to the fact that in the LLN limit \eqref{eq:Vft}, there is a contribution of magnitude $\Den^{1-2H}c(s)$ coming from the semimartingale part that is negligible on first order but not at a rate of $\sqrt{\Del}$. Expanding $\mu_f ( \Upsilon^{n,i} )$ in a Taylor sum   up to order $N(H)$,  we obtain
	\begin{align*} \mu_f ( \Upsilon^{n,i} )&=\mu_f ( \pi( (i-1)\Den ) ) + \sum_{j=1}^{N(H)}    
		\sum_{\vert \chi \vert = j}\frac{1}{\chi!}  \pd^{\chi} \mu_f  ( \pi( (i-1)\Den )  ) 
		(\Upsilon^{n,i} - \pi( (i-1)\Den ) )^{\chi}  \\
		&\quad+\sum_{\vert \chi \vert = N(H) + 1} 	\frac{1}{\chi!}\pd^{\chi} \mu_f( \upsilon^n_i ) (\Upsilon^{n,i} - \pi( (i-1)\Den ) )^{\chi}, \end{align*}
	where  $\upsilon^n_i$ is a point between $\Upsilon^{n,i}$ and $\pi( (i-1)\Den )$. The next lemma shows two things: first, the term of order $N(H)+1$ is negligible, and second, for $j=1,\dots,N(H)$, we may replace $\Upsilon^{n,i} - \pi( (i-1)\Den )$ by $\Den^{1-2H}c((i-1)\Del)$.
	\blem\label{lem:Taylor}
	We have that $\mathbb{X}_1^{n}\limL0$ and $\mathbb{X}_2^{n}\limL0$, where
	\begin{equation} \label{CLT:ZGW:2}
		\begin{split}
			\mathbb{X}_1^{n} (t)& =\mathtoolsset{multlined-width=0.8\displaywidth} \begin{multlined}[t]
				\DenOneHalf \sum_{i=\theta_n + 1}^{[t/\Delta_n-L+1}
				\sum_{j = 1}^{N(H)} 
				\sum_{\vert \chi \vert = j} \frac{1}{\chi!} \pd^{\chi} \mu_f(\pi( (i-1)\Den )) \\   \times 
				\Big\{ (\Upsilon^{n,i} - \pi( (i-1)\Den ) )^{\chi} 
				- \Den^{j(1-2H)} c( (i-1)\Den )^{\chi} \Big\},\end{multlined} \\
			\mathbb{X}_2^{n} (t)& = \DenOneHalf \sum_{i=\theta_n + 1}^{[t/\Delta_n]-L+1}
			\sum_{\vert \chi \vert = N(H) + 1} 	\frac{1}{\chi!}\pd^{\chi} \mu_f( \upsilon^n_i ) (\Upsilon^{n,i} - \pi( (i-1)\Den ) )^{\chi}. 
		\end{split}
	\end{equation}
	\elem
	
	In a final step, we remove the discretization of $\si$ and $\rho$.
	\begin{Lemma} \label{lem:app:10} If $\theta$ is chosen according to \eqref{theta:value}, then		
		\begin{equation} \label{app:ZGW}
			\Del^{-\frac12}\Bigg\{ \Del\sum_{i=\la_n + 1}^{[t/\Delta_n]-L+1} \mu_f ( \pi( (i-1)\Den )) -  \int_{0}^{t} \mu_f (\pi( s )) \, \dd s\Bigg\} \LeinsKonv 0
		\end{equation}
		and
		\begin{equation} \label{app:ZGW:2}
			\begin{split}
				& \DenMinOneHalf\Bigg\{	\Delta_n\sum_{i=\theta_n + 1}^{[t/\Delta_n]-L+1}
				\sum_{j=1}^{N(H)}  \Den^{j(1-2H)}
				\sum_{\vert \chi \vert = j}\frac{1}{\chi!} \pd^{\chi} \mu_f  ( \pi( (i-1)\Den )  ) 
				c( (i-1)\Den )^{\chi}  \\ & \qquad
				- \int_{0}^{t}  \sum_{j=1}^{N(H)}  
				\sum_{\vert \chi \vert = j}\frac{1}{\chi!} \pd^{\chi} \mu_f  (\pi( s ) )  
				\Den^{j(1-2H)} c(s)^{\chi} 
				\, \dd s\Bigg\} \LeinsKonv 0.
			\end{split}
		\end{equation}
	\end{Lemma}

	By the properties of stable convergence in law (see Equation~(2.2.5) in \cite{JP}), the CLT in \eqref{CLT:2} follows by combining Lemmas~\ref{lem:app:2}--\ref{lem:app:10}.
\end{proof} 

\subsection{Details of the proof of Theorem~\ref{thm:CLT:mixedSM}}\label{sect:details}

Assumption~\ref{Ass:A'} is in force throughout this section.

\begin{proof}[Proof of Lemma~\ref{lem:app:2}]
	By the calculations in \eqref{SSE:A'}--\eqref{SSE:fBM}, we have
	$\bbe [ \Vert \uDeni Y/\Den^H \Vert^p ]^{{1}/{p}} \lesssim 1$ for all $p\geq1$. As $f$ grows at most polynomially, we see that  
	$\bbe [ \vert f ( \un{\Delta}_i^n Y/\Den^H  ) \vert ]$ is of size $1$. Hence, 
	$\bbe[   \vert \Delta_n^{1/2} \sum_{i = 1}^{\theta_n} f ( \un{\Delta}_i^n Y/\Den^H ) \vert ] 
	\lesssim \Den^{1/2 - \theta}$, which implies
	$
	\Delta_n^{ {1}/{2}} \sum_{i = 1}^{\theta_n} f  (  {\un{\Delta}_i^n Y}/{\Delta_n^H}  ) \to 0
	$ in $L^1$
	since $\theta < \frac12$ by \eqref{theta:value}. As a result, omitting the first $\theta_n$ terms in the definition of $V^n_f(Y,t)$ does no harm asymptotically.
	Next, we define
	\begin{equation} \label{1st:tr:vec}
		\begin{split}
			\La_i^{n} &= f \bigg( \frac{\un{\Delta}_i^n Y}{\DenH} \bigg) - 
			f \bigg( \frac{\un{\Delta}_i^n Y^{\mathrm{tr}}}{\DenH} \bigg), \qquad
			\ov	\La_i^{n} = \La_i^{n} - \bbe [\La_i^{n} \mid \calf^n_{i - \theta_n}].
		\end{split}
	\end{equation}
	By our choice \eqref{theta:value} of $\theta$ and since $H < \frac12$, the lemma is proved once 
	\begin{align}\label{eq:toprove}
		\bbe \bigg[\sup_{t \leq T} \bigg \vert \DenOneHalf \sum_{i = \theta_n + 1}^{[t/\Delta_n]-L+1} \ov\La^{n}_i \bigg \vert \bigg] 
		&\lesssim \Den^{\theta  (\frac{1}{2}-H )},\\
		\bbe \bigg[\sup_{t \leq T} \bigg \vert \DenOneHalf \sum_{i = \theta_n + 1}^{[t/\Delta_n]-L+1} \bbe [ \La^{n}_i \vert \calf^n_{i - \theta_n}] \bigg \vert \bigg] &\lesssim
		\Den^{\theta  (\frac{1}{2}-H )} + \Den^{ 2\theta(1-H) - \frac{1}{2} }
		\label{eq:toprove2}
	\end{align}
	are established. To this end, let
	$\lambda^n_i =  {\un{\Delta}_i^n Y - \un{\Delta}_i^n Y^{\tr}}/{\DenH}  =
	\int_{0}^{(i - \theta_n)\Den}  \rho_s \, \dd W_s \, \frac{\uDeni g(s)}{ \DenH}$.
	By Assumption \ref{Ass:A}, we have $\lvert f(z) - f(z') \rvert \lesssim (1 + \Vert z \Vert^{p-1} + \Vert z' \Vert^{p-1} ) \Vert z - z' \Vert$.  In addition, 
	$\bbe [ \vert f ( \un{\Delta}_i^n Y/\Den^H  ) \vert ]$ is of size 1, so
	$
	\bbe [ (\ov\La^{n}_i)^2 ] \lesssim \bbe [ (\La^n_i)^2 ]  
	\lesssim \bbe [ \Vert \lambda^n_i \Vert^{2}]
	\lesssim \Den^{2 \theta (1-H)},
	$
	where we used \eqref{SSE:fBM} for the last estimation.
	By construction, $\ov\La^{n}_i$ is $\calf^n_{i + L -1}$-measurable and has conditional expectation $0$ given $\calf^n_{i - \theta_n}$. Therefore, we can further use an estimate of the kind \eqref{MSE} to show that the left-hand side of \eqref{eq:toprove} is bounded, up to constant, by $ \sqrt{\theta_n} \Den^{\theta (1-H)} \lesssim
	\Den^{\theta (1-H) - \theta/2} = 
	\Den^{\theta  ({1}/{2}-H )}$.

	Next, let
	$
	\psi^n_i  = \si_{(i- \theta_n) \Den} \uDeni B + \int_{(i - \theta_n) \Den}^{(i + L - 1) \Den} \rho_{(i - \theta_n) \Den} \, \dd W_s \, \uDeni g(s)
	$.
	Since $f$ is smooth, applying Taylor's theorem twice yields $\La^{n}_i =	\La^{n,1}_i + \La^{n,2}_i + \La^{n,3}_i$, where
	\begin{align*}
		\La^{n,1}_i & 
		=
		\sum_{\lvert \chi \rvert = 1} \pd^{\chi} f \biggl( \frac{\psi^n_i}{\DenH} \biggr) (\lambda^n_i)^{\chi},\qquad  \La^{n,2}_i =
		\sum_{\lvert \chi \rvert ,\lvert \chi' \rvert = 1} \pd^{\chi + \chi'} f (\wt\eta_i^{n}) 
		\biggl( \frac{ \uDeni Y^{\tr} - \psi^n_i }{\DenH} \biggr)^{\chi'} (\lambda^n_i)^{\chi}, \\
		\La^{n,3}_i&= 	
		\sum_{\lvert \chi \rvert = 2} \frac{\pd^{\chi} ( \eta_i^{n} )}{\chi!}    (\lambda^n_i)^{\chi} 
	\end{align*}
	and $\chi, \chi' \in	\bbn_0^{d \times L}$ are multi-indices and $\eta_i^{n}$ (resp., $\wt\eta_i^{n}$) is a point on the line between 
	$\uDeni Y / \DenH $ and $\uDeni Y^{\tr} /\DenH $ (resp., $\uDeni Y^{\tr} / \DenH $ and $\psi^n_i / \DenH $). Accordingly, we split
	\begin{equation*}
		\DenOneHalf \sum_{i=\theta_n + 1}^{[t/\Delta_n]-L+1} \mathbb{E} [ \La^{n}_i \mid \mathcal{F}^n_{i - \theta_n} ] 
		= \sum_{j=1}^3 \bbl_j^n(t), \qquad
		\bbl_j^n(t) = \DenOneHalf \sum_{i=\theta_n + 1}^{[t/\Delta_n]-L+1} \mathbb{E} [ \La^{n,j}_i \mid \mathcal{F}^n_{i - \theta_n} ].
	\end{equation*}	
	Note that
	$
	\mathbb{E}  [ \pd^{\chi} f  ( \frac{\psi^n_i}{\DenH}  ) (\lambda^n_i)^{\chi} \mid \mathcal{F}^n_{i - \theta_n}  ]  =
	(\lambda^n_i)^{\chi}\, \mathbb{E}  [  \pd^{\chi} f  ( \frac{\psi^n_i}{\DenH}  ) \mid \mathcal{F}^n_{i - \theta_n}  ] = 0
	$
	because $\lambda^n_i$ is $\calf^n_{i - \theta_n}$-measurable, $\psi^n_i$ is centered normal given $\mathcal{F}^n_{i - \theta_n}$ and $f$ has odd partial derivatives of first orders (since $f$ is even). It follows that $\bbl_1^n(t) = 0$ identically.
	Writing $$\bone^n_i(s)=(\bone_{((i-1)\Del,i\Del)}(s),\dots,\bone_{((i+L-2)\Del,(i+L-1)\Del)}(s)),$$ we can   decompose $	\uDeni Y^{\tr} - \psi^n_i$ as
	\begin{equation*}
		\uDeni A + \int_{0}^{t} (\si_s - \si_{(i- \theta_n) \Den})\, \dd B_s\, \bone^n_i(s)
		+ \int_{(i - \theta_n) \Den}^{(i + L - 1) \Den} (\rho_s - \rho_{(i - \theta_n) \Den}) \, \dd W_s \, \uDeni g(s).
	\end{equation*}
	By a standard size estimate, it follows that
	\begin{equation*}
		\begin{split}
			\bbe \bigg[ \sup_{t \leq T} \big \vert \bbl^n_2(t) \big \vert \bigg] & \lesssim (\DenOneHalf \Den^{-1}) 
			( \Delta_n^{1-H} + \theta_n^{\frac{1}{2}} \Den^{1-H} + (\theta_n \Den)^{\frac{1}{2}}  )\Den^{\theta(1-H)}  \\ &
			\lesssim \Den^{-\frac{1}{2}} \Den^{\theta(1-H)} (\theta_n \Den)^{\frac{1}{2}}
			= \Den^{\theta  (\frac{1}{2} - H ) }, \\
			\bbe \bigg[ \sup_{t \leq T} \big \vert \bbl^n_3(t) \big \vert \bigg] & \lesssim 
			\DenMinOneHalf  (\Den^{ \theta(1-H) }  )^2 = 
			\Den^{ 2\theta(1-H) - \frac{1}{2} },
		\end{split}
	\end{equation*}
	proving  \eqref{eq:toprove2} and thus the lemma.
\end{proof}

\begin{proof}[Proof of Lemma \ref{approx:lem:1}]	
	Let 
	$ \xi^{n,\di}_i=\int_{(i-\theta_n)\Del}^{(i+L-1)\Del} \rho_{(i-\theta_n)\Del} \,\dd W_s\,\un\Delta^n_i g(s)$ and recall the definition of $\xi^n_i$ from  \eqref{1st:tr}.
	In a first step, we   show that $U^n$ can be approximated by 
	\begin{multline*} \ov U^n(t)= \DenOneHalf \sum_{i=\theta_n + 1}^{[t/\Den]-L+1} \Bigg\{ 
		f \bigg( \frac{\si_{(i-1)\Del} \un\Delta^n_i B + \xi^{n,\di}_i}{\DenH} \bigg) - f \bigg( \frac{ \xi^{n,\di}_i }{\DenH} \bigg) \\
		- 	\bbe \bigg[f \bigg( \frac{\si_{(i-1)\Del} \un\Delta^n_i B + \xi^{n,\di}_i}{\DenH} \bigg) - f \bigg( \frac{ \xi^{n,\di}_i }{\DenH} \bigg) \mathrel{\Big\vert}  \calf^n_{i - \theta_n} \bigg] \Bigg\}. \end{multline*}
	By \eqref{eq:Hoelder} and a size estimate as in \eqref{SSE:fBM}, the difference $\xi^n_i-\xi^{n,\di}_i$ is of size $(\theta_n\Del)^{1/2}$. Together with \eqref{SSE:A'} and \eqref{SSE:BM}, we further have that $\un\Delta^n_i Y^\tr- \si_{(i-1)\Del} \un\Delta^n_i B - \xi^{n,\di}_i$ is of size $\Del+\sqrt{\Del}+(\theta_n\Del)^{1/2}$. By the mean-value theorem, these size bounds imply that  
	$$ \E\bigg[\bigg\lvert	f \bigg( \frac{\uDeni Y^{\tr}}{\DenH} \bigg)- f \bigg( \frac{ \si_{(i-1)\Del} \un\Delta^n_i B + \xi^{n,\di}_i }{\DenH} \bigg) \bigg\rvert^p 
	+ \bigg\lvert f \bigg( \frac{\xi^n_i}{\DenH} \bigg) - f \bigg( \frac{ \xi^{n,\di}_i }{\DenH} \bigg)  \bigg\rvert^p\bigg]^{\frac1p}\lec(\theta_n\Del)^{\frac12} $$		
	for any $p>0$. 
	Moreover, the $i$th term in the definition of $\ov U^n(t)$  is $\calf^n_{i+L-1}$-measurable with zero mean conditionally on $\calf^n_{i-\theta_n}$. Therefore, employing a martingale size estimate as in \eqref{MSE}, we obtain
	$ \E [\sup_{t\leq T} \lvert U^n(t)-\ov U^n(t)\rvert ] \lec \sqrt{\theta_n}(\theta_n\Del)^{1/2}\leq \Del^{1/2-\theta},  $
	which converges to $0$ by \eqref{theta:value}.
	
	Next, because $B$ and $W$ are independent, we can apply It\^o's formula with $\xi^{n,\di}_i$ as starting point and write 
	\beq\label{Ito1}\begin{split}
		&f \bigg( \frac{\si_{(i-1)\Del} \un\Delta^n_i B + \xi^{n,\di}_i}{\DenH} \bigg)- f \bigg( \frac{ \xi^{n,\di}_i }{\DenH} \bigg) \\
		&\qquad= \Del^{-H}\sum_{j,k=1}^d\sum_{\ell=1}^L \int_{(i+\ell-2)\Del}^{(i+\ell-1)\Del} \frac{\pd}{\pd z_{k\ell}} f\bigg(\frac{\un\Delta Y^{n,\di}_i(s)}{\Del^H}\bigg)\si^{kj}_{(i-1)\Del}\,\dd B^j_s   \\
		&\qquad\quad+\frac12\Del^{-2H}\sum_{k,k'=1}^d\sum_{\ell=1}^L \int_{(i+\ell-2)\Del}^{(i+\ell-1)\Del} \frac{\pd^2}{\pd z_{k\ell}\pd z_{k'\ell}} f\bigg(\frac{\un\Delta Y^{n,\di}_i(s)}{\Del^H}\bigg)(\si\si^T)^{kk'}_{(i-1)\Del}\,\dd s,
	\end{split}\eeq
	where $\un\Delta Y^{n,\di}_i(s)=\int_{(i-1)\Del}^s \si_{(i-1)\Del}\,\dd B_r\,\bone^n_i(r)+\xi^{n,\di}_i$.
	Clearly, the stochastic integral is $\calf^n_{i+L-1}$-measurable and   conditionally centered given  $\calf^n_{i-1}$. Therefore, by a martingale size estimate, its contribution to $\ov U^n(t)$ is of  magnitude $\Del^{1/2-H}$, which is negligible because $H<\frac12$. For the Lebesgue integral, we   apply It\^o's formula again and write 
	\begin{align*}
		&\frac{\pd^2}{\pd z_{k\ell}\pd z_{k'\ell}} f\bigg(\frac{\un\Delta Y^{n,\di}_i(s)}{\Del^H}\bigg)=	\frac{\pd^2}{\pd z_{k\ell}\pd z_{k'\ell}}f\bigg(\frac{\xi^{n,\di}_i}{\Del^H}\bigg)\\
		&\quad +\Del^{-H}\sum_{j_2,k_2=1}^d\sum_{\ell_2=1}^L \int_{(i+\ell_2-2)\Del}^{s\wedge(i+\ell_2-1)\Del} \frac{\pd^3}{\pd z_{k\ell}\pd z_{k'\ell}\pd z_{k_2\ell_2}} f\bigg(\frac{\un\Delta Y^{n,\di}_i(r)}{\Del^H}\bigg)\si^{k_2j_2}_{(i-1)\Del}\,\dd B^{j_2}_r\\
		&\quad+\frac{\Del^{-2H}}{2}\sum_{k_2,k'_2=1}^d\sum_{\ell_2=1}^L \int_{(i+\ell_2-2)\Del}^{s\wedge(i+\ell_2-1)\Del} \frac{\pd^4}{\pd z_{k\ell}\pd z_{k\ell}\pd z_{k_2\ell_2}\pd z_{k'_2\ell_2}} f\biggl(\frac{\un\Delta Y^{n,\di}_i(r)}{\Del^H}\biggr)(\si\si^T)^{k_2k'_2}_{(i-1)\Del}\,\dd r.
	\end{align*} 
	By the same reason as before, the stochastic integral (even after we plug it into the drift in \eqref{Ito1}) is $\calf^n_{i+L-1}$-measurable with zero $\calf^n_{i-1}$-conditional mean and therefore negligible. The Lebesgue integral is essentially of the same form as the one in \eqref{Ito1}. Because $f$ is smooth, we can repeat this procedure as often as we want. What is important, is that we gain a net factor of $\Del^{1-2H}$ in each step (we have $\Del^{-2H}$ times a Lebesgue integral over an interval of length at most $\Del$). After $N$ applications of It\^o's formula, the final drift term yields a contribution of size
	$\sqrt{\theta_n}\Del^{N(1-2H)}$
	to $\ov U^n(t)$. As $\theta<\frac12$, it suffices to take $N=N(H)+1$ to make this convergent to $0$.
\end{proof}

\begin{proof}[Proof of Lemma \ref{lem:approx:6}]
	We begin by discretizing $\rho$ on a finer scale and let
	\begin{equation} \label{dis:CLT} 
		\Theta_i^n  = \int_{(i - \theta_n)\Delta_n}^{(i + L - 1) \Delta_n} 
		\sum_{k=1}^{Q} \rho_{(i- \theta_n^{(q-1)}) \Den } 
		\ind_{ ((i-\theta_n^{(q-1)})\Den, (i-\theta_n^{(q)})\Den) }(s)   \, \dd W_s \, \uDeni g(s),
	\end{equation}
	where $\theta_n^{(q)} = [\Delta_n^{-\theta^{(q)}}]$ for $q=0, \ldots, Q-1$, $\theta_n^{(Q)} = - (L - 1)$ and the
	numbers $\theta^{(q)}$, $q=0, \ldots, Q-1$ for some $Q \in \bbn$, are chosen such that $\theta = \theta^{(0)}> \dots > \theta^{(Q-1)} > \theta^{(Q)} = 0$ and 
	\begin{equation} \label{it:CLT}
		\theta^{(q)} > \frac{\ga}{1 - H} \theta^{(q - 1)} - \frac{\ga - \frac12}{1-H}, \qquad q=1, \ldots, Q,
	\end{equation}
	where $\ga$ describes the regularity of the volatility process $\rho^{(0)}$ in \eqref{repr:si-2}. Because $H<\frac12$ and we can make $\ga$ arbitrarily close to $\frac12$ if we want, there is no loss of generality to assume that $\ga/(1-H)<1$. In this case, the fact that a choice as in \eqref{it:CLT} is possible can be verified by solving the associated linear recurrence equation.  Defining
	$
	\un{\Delta}^n_i Y^{\mathrm{dis}}  = \si_{(i - 1) \Den} \un{\Delta}^n_i B + \Theta_i^n,
	$
	we will show in Lemma~\ref{lem:help1} below that
	\begin{equation} \label{app:lem6:1}
		\DenOneHalf \sum_{i=\theta_n + 1}^{[t/\Den]-L+1} \bigg\{ \bbe \bigg[f  \bigg(\frac{\un{\Delta}_i^n Y^{\tr}}{\DenH } \bigg) \mathrel{\Big\vert} \calf^n_{i - \theta_n} \bigg] - 
		\bbe \bigg[f \bigg(\frac{\un{\Delta}^n_i Y^{\mathrm{dis}}}{\DenH } \bigg) \mathrel{\Big\vert} \calf^n_{i - \theta_n} \bigg] \bigg\} \LeinsKonv 0.
	\end{equation}

	Next, we define another  matrix $\Upsilon_i^{n,0} \in (\bbr^{d \times L})^2$ by
	\begin{equation} \label{cov:v0}
		\begin{split}
			(\Upsilon_i^{n,0})_{k\ell,k'\ell'} & = c((i-1)\Del) \Delta_n^{1-2H}  
			+ \sum_{q=1}^{Q} 
			\Big( \rho_{(i- \theta^{(q-1)}_n) \Den} \rho_{(i- \theta^{(q-1)}_n) \Den}^T \Big)_{k k'} \\ & \quad\times
			\int_{(i - \theta_n)\Delta_n}^{(i + L - 1) \Delta_n} \frac{\Delta^n_{i+\ell - 1} g(s) \Delta^n_{i+\ell' - 1} g(s)}{\Den^{2H}} 
			\ind_{ ((i-\theta_n^{(q-1)})\Den, (i-\theta_n^{(q)})\Den) }(s) \,  \dd s.
		\end{split}
	\end{equation}
	If $c$ and $\rho$ are deterministic, this is the covariance matrix of $\un{\Delta}^n_i Y^{\tr}/\Den^H$. Also notice that the only difference to $\Upsilon^n_i$ are the discretization points of $\rho$. Next, we show that 
	\begin{equation} \label{app:lem6:4}
		\DenOneHalf \sum_{i=\theta_n + 1}^{[t/\Den]-L+1} \bigg\{ 
		\bbe \bigg[ f \bigg(\frac{\un{\Delta}^n_i Y^{\di}}{\DenH } \bigg) \mathrel{\Big|} \calf^n_{i - \theta_n} \bigg]
		- \mu_f ( \bbe  [ \Upsilon_i^{n,0} \mid \calf^n_{i - \theta_n}  ]  ) \bigg\} \LeinsKonv 0,
	\end{equation}
	where $\mu_f$ is the mapping defined after Assumption~\ref{Ass:A}. This will be achieved through successive conditioning in Lemma~\ref{lem:help2}. Finally, as we  show in Lemma~\ref{lem:help3}, we have
	\begin{align} \label{app:lem6:2}
		\bbe \bigg[ \sup_{t \leq T} \bigg\vert
		\DenOneHalf \sum_{i=\theta_n + 1}^{[t/\Den]-L+1} \bigg\{ 
		\mu_f \bigg( \bbe [ \Upsilon_i^{n,0} \mid \calf^n_{i - \theta_n} ] \bigg)
		- \mu_f ( \Upsilon_i^{n,0} )
		\bigg\} \bigg\vert \bigg] &\ra 0,
		\\
		\label{app:lem6:3}
		\bbe \bigg[ \sup_{t \leq T} \bigg\vert \DenOneHalf \sum_{i=\theta_n + 1}^{[t/\Den]-L+1}  \{ \mu_f ( \Upsilon_i^{n,0} ) - \mu_f ( \Upsilon^{n,i} )
		\} \bigg\vert\bigg] &\ra 0,
	\end{align}
	which completes the proof of the current lemma.
\end{proof}

\blem\label{lem:help1} The convergence \eqref{app:lem6:1} holds true.
\elem
\bpr
By Taylor's theorem, the left-hand side of \eqref{app:lem6:1} is $\bbq^n_1(t) + \bbq^n_2(t)$ with
\begin{equation*}
	\begin{split}
		\bbq^n_1(t) & = \Delta_n^{\frac{1}{2}} \sum_{i = \theta_n + 1}^{[t/\Den]-L+1} 
		\sum_{\vert \chi \vert = 1} 
		\mathbb{E} \bigg[ \pd^{\chi} f \bigg(\frac{\un{\Delta}^n_i Y^{\di}}{\DenH }\bigg) (\kappa_i^n)^{\chi} \mathrel{\Big\vert}  \mathcal{F}^n_{i - \theta_n} \bigg], \\
		\bbq^n_2(t) & =  \Delta_n^{\frac{1}{2}} \sum_{i = \theta_n + 1}^{[t/\Den]-L+1} 
		\sum_{\vert \chi \vert = 2}  \frac1{\chi!}
		\mathbb{E}  [ \pd^{\chi} f  ( \ov\kappa_i^n  ) (\kappa_i^n)^{\chi}\mid \mathcal{F}^n_{i - \theta_n}  ],	
	\end{split}
\end{equation*}
where $\kappa_i^n =(\un{\Delta}^n_i Y^{\tr} - \un{\Delta}^n_i Y^{\di})/\DenH$ and $ \ov\kappa_i^n$ is some point on the line between $\un{\Delta}^n_i Y^{\tr}/\DenH$ and $\un{\Delta}^n_i Y^{\di}/\DenH$. By definition,
\begin{equation} \label{dec:6}
	\begin{split}
		\!\!(\kappa^n_i)_{k\ell} & = \frac{\Delta^n_{i+\ell -1} A^k}{\DenH} 
		+ \frac{1}{\DenH} \int_{(i+\ell - 2) \Delta_n}^{(i+\ell - 1) \Delta_n} \sum_{\ell' = 1}^{d'} \Big( \si_s^{k\ell'} - \si^{k\ell'}_{(i - 1) \Den} \Big) \, \dd B_s^{\ell'} \\ & \quad
		+ \sum_{q=1}^{Q} \int_{(i-\theta_n^{(q-1)})\Den}^{(i-\theta_n^{(q)})\Den} \frac{\Delta^n_{i+\ell -1} g(s)}{\DenH}  
		\sum_{\ell' = 1}^{d'}
		\Big( \rho_s^{k\ell'} - \rho^{k\ell'}_{(i-\theta_n^{(q-1)})\Den}  \Big)  \, \dd W^{\ell'}_s.
	\end{split}
\end{equation}
Using H\"older's inequality, the estimates \eqref{SSE:A'}, \eqref{SSE:BM} and \eqref{SSE:fBM} and the polynomial growth assumption on $\partial^\chi f$, we see that $\un{\Delta}^n_i Y^{\di}/\DenH$ is of size one and, since $0 < \theta^{(q)} < \frac12$,  
\begin{equation}\label{eq:Q}
	\bbe \bigg[ \sup_{t \leq T} \bv \bbq^n_2(t) \bv \bigg] \lesssim \Del^{-\frac12}
	\bigg( \Den^{2(1-H)} + \Den^{2(1-H)} + \sum_{q=1}^{Q} \Den^{(1-\theta^{(q-1)}) + 2\theta^{(q)}(1-H)} \bigg) 
	\ra 0.
\end{equation}

Next, we further split $\bbq^n_1(t) = \bbq^n_{11}(t) + \bbq^n_{12}(t) + \bbq^n_{13}(t)$ into three terms according to the decomposition \eqref{dec:6}. Using again \eqref{SSE:A'} and \eqref{SSE:BM}, we see that both $\bbq^n_{11}(t)$ and $\bbq^n_{12}(t)$ are of size $\Den^{-1/2 + (1-H)} = \Den^{1/2-H}$.
We first tackle the term $\bbq^n_{13}(t)$, which requires a more careful analysis. Here we need assumption  \eqref{repr:si-2} on the noise volatility   $\rho$. Since   $t \mapsto \int_{0}^{t} \wt{b}_s \, \dd s$ satisfies a better regularity condition than \eqref{eq:Hoelder}, we may incorporate the drift term  in $\rho^{(0)}$ for the remainder of the proof. Then we further write $\bbq^n_{13}(t)=\bbr^{n}_1(t) +\bbr^{n}_2(t)$ where $\bbr^{n}_1(t)$ and $\bbr^{n}_2(t)$ correspond to taking only $\rho^{(0)}$ and
$\int_{0}^{t} \wt{\rho}_s \, \dd \wt{W}_s$ instead of $\rho$, respectively.
By  \eqref{mom:ass:si:rho:0-2}, \eqref{SSE:fBM} and \eqref{it:CLT},  $\bbr^{n}_1(t)$ is of size 
\beq\label{eq:help}\sum_{q=1}^{Q} \Den^{-\frac12 + \ga ( 1-\theta^{(q-1)}) + \theta^{(q)}( 1 - H)}\to0.\eeq

For $\bbr^{n}_2(t)$, we write 	$\bbr^{n}_2(t)		=\sum_{\lvert \chi\rvert=1} (\bbr^{n,\chi}_{21}(t) + \bbr^{n,\chi}_{22}(t)  + \bbr^{n,\chi}_{23}(t))$, where, if $\chi_{k\ell} = 1$, 
\begin{align*}  
	\bbr^{n,\chi}_{21}(t) & = \Delta_n^{\frac{1}{2}} \sum_{i = \theta_n + 1}^{[t/\Den]-L+1} \sum_{\ell', \ell'' = 1}^{d}
	\mathbb{E} \bigg [ \pd^{\chi} f \bigg(\frac{\un{\Delta}^n_i Y^{\di}}{\DenH }\bigg) 	
	\sum_{q=1}^{Q} \int_{(i-\theta^{(q-1)}_n)\Den}^{(i-\theta^{(q)}_n)\Den} \frac{\Delta^n_{i+\ell -1} g(s)}{\DenH}  
	\\ &  \quad \times  
	\int_{(i-\theta_n^{(q-1)})\Den}^{s} \Big(\wt{\rho}_r^{k,\ell',\ell''} - \wt{\rho}_{(i-\theta_n^{(q-1)})\Den}^{k,\ell',\ell''} \Big) \, 
	\dd {\wt{W}^{\ell''}_r} \, \dd W^{\ell'}_s
	\mathrel{\Big|}\mathcal{F}^n_{i - \theta_n} \bigg ], \\
	\bbr^{n,\chi}_{22}(t) &= \Delta_n^{\frac{1}{2}} \sum_{i = \theta_n + 1}^{[t/\Den]-L+1} \sum_{q=1}^{Q} \sum_{\ell', \ell'' = 1}^{d}
	\mathbb{E} \bigg [ \bigg\{\pd^{\chi} f \bigg(\frac{\un{\Delta}^n_i Y^{\di}}{\DenH }\bigg) - \pd^{\chi} f \bigg(\frac{\un{\Delta}^n_i Y^{\di,q}}{\DenH }\bigg)\bigg\}	
	\\ &  \quad \times		\int_{(i-\theta^{(q-1)}_n)\Den}^{(i-\theta^{(q)}_n)\Den} \frac{\Delta^n_{i+\ell -1} g(s)}{\DenH}  	
	\int_{(i-\theta_n^{(q-1)})\Den}^{s} \wt{\rho}_{(i-\theta_n^{(q-1)})\Den}^{k,\ell',\ell''}  \,\dd {\wt{W}^{\ell''}_r} \, \dd W^{\ell'}_s
	\mathrel{\Big|}\mathcal{F}^n_{i - \theta_n} \bigg ], \\ 
	\bbr^{n,\chi}_{23}(t)  & =\Delta_n^{\frac{1}{2}} \sum_{i = \theta_n + 1}^{[t/\Den]-L+1} \sum_{q=1}^{Q} \sum_{\ell', \ell'' = 1}^{d}
	\mathbb{E} \bigg [ \pd^{\chi} f \bigg(\frac{\un{\Delta}^n_i Y^{\di,q}}{\DenH }\bigg) 
	\int_{(i-\theta^{(q-1)}_n)\Den}^{(i-\theta^{(q)}_n)\Den} \frac{\Delta^n_{i+\ell -1} g(s)}{\DenH}  	
	\\ & \quad \times 
	\int_{(i-\theta_n^{(q-1)})\Den}^{s} \wt{\rho}_{(i-\theta_n^{(q-1)})\Den}^{k,\ell',\ell''}  \, \dd{\wt{W}^{\ell''}_r} \, \dd W^{\ell'}_s
	\mathrel{\Big|}\mathcal{F}^n_{i - \theta_n} \bigg ]
\end{align*}
and 
$\un{\Delta}^n_i Y^{\di,q}  = \int_{(i-\theta_n^{(q-1)})\Den}^{(i + L - 1) \Delta_n} 
\rho_{(i- \theta_n^{(q-1)}) \Den } 
\, \dd W_s \, \uDeni g(s)$.
Using the BDG  and Minkowski integral inequality alternatingly,  we obtain, for any $p\geq2$,
\begin{align*}  
	& \bbe \bigg [ \bigg\vert \int_{(i-\theta^{(q-1)}_n)\Den}^{(i-\theta^{(q)}_n)\Den} \frac{\Delta^n_{i+\ell -1} g(s)}{\DenH}   \bigg(
	\int_{(i-\theta_n^{(q-1)})\Den}^{s} \Big(\wt{\rho}_r^{k,\ell',\ell''} - \wt{\rho}_{(i-\theta_n^{(q-1)})\Den}^{k,\ell',\ell''} \Big) \, 
	\dd {\wt{W}^{\ell''}_r} \bigg) \dd W^{\ell'}_s
	\bigg\vert^p \bigg ]^{\frac{1}{p}} \\ 
	& \quad\lec
	\bigg(\int_{(i-\theta^{(q-1)}_n)\Den}^{(i-\theta^{(q)}_n)\Den}  \frac{\Delta^n_{i+\ell -1} g(s)^2}{\Del^{2H}}  		 
	\bbe \bigg[  \bigg\vert
	\int_{(i-\theta_n^{(q-1)})\Den}^{s} \Big(\wt{\rho}_r^{k,\ell',\ell''} - \wt{\rho}_{(i-\theta_n^{(q-1)})\Den}^{k,\ell',\ell''} \Big) \, 
	\dd {\wt{W}^{\ell''}_r}	
	\bigg\vert^{2p} \bigg]^{\frac{1}{p}} \, \dd s \bigg)^{\frac{1}{2}} \\ 
	& \quad\lec
	\bigg ( \int_{(i-\theta^{(q-1)}_n)\Den}^{(i-\theta^{(q)}_n)\Den}  \frac{\Delta^n_{i+\ell -1} g(s)^2}{\Del^{2H}}  
	\int_{(i-\theta_n^{(q-1)})\Den}^{s}  \bbe \bigg[  \Big\vert
	\wt{\rho}_r^{k,\ell',\ell''} - \wt{\rho}_{(i-\theta_n^{(q-1)})\Den}^{k,\ell',\ell''}
	\Big\vert^{2p} \bigg]^{\frac{1}{p}} \, \dd r \, \dd s \bigg )^{\frac{1}{2}} \\  & \quad
	\lesssim (\theta_n^{(q-1)} \Den)^{\frac{1}{2}(1 + 2\eps')}\bigg( \int_{(i - \theta_n^{(q-1)})\Del}^{(i - \theta_n^{(q)}) \Den} \frac{\Delta^n_{i+\ell -1} g(s)^2}{\Del^{2H}}  \,\dd s
	\bigg)^{\frac{1}{2}}
	\lesssim \Den^{ (\frac{1}{2} + \eps' ) (1 - \theta^{(q-1)}) + \theta^{(q)}(1-H)},
\end{align*}
where $\eps'$ is as in \eqref{reg:ass:si:ti-2}. Thus, $\bbr^{n,\chi}_{21}(t)$ is of size 
$\sum_{q=1}^{Q} \Den^{-\frac12 + (\frac12 + \eps') (1 - \theta^{(q-1)}) + \theta^{(q)}(1-H)}$, which  is almost the same as \eqref{eq:help};  the only difference is that $\ga$ is replaced by $\frac12 + \eps'$. Since we can assume without loss of generality that $\frac12 + \eps' < \ga$, the formula \eqref{it:CLT} implies that we have
$-\frac12 + (\frac12+ \eps') (1 - \theta^{(q-1)}) + \theta^{(q)}(1-H) >0$ for all $q=1,\ldots,Q$, which means that $\bbr^{n,\chi}_{21}(t)$ is asymptotically negligible.

Next, using Lemma~\ref{lem:ker} (iii) and a similar estimate to the previous display, we see that $(\Theta^{n}_i - \un{\Delta}^n_i Y^{\di,q})/\DenH$ is of size $\Del^{\theta^{(q-1)}(1-H)}+\Den^{(1-\theta^{(q-1)})/2}$. Hence, with the two estimates \eqref{SSE:A'} and \eqref{SSE:BM} at hand, we deduce that $\bbr^{n,\chi}_{22}(t)$ is of size  
\begin{align*}&\sum_{q=1}^{Q} \Den^{-\frac12} (  \Den^{\frac12-H} + \Del^{\theta^{(q-1)}(1-H)}+\Den^{\frac12(1-\theta^{(q-1)})}) \Den^{ \theta^{(q)}(1-H)+ \frac12(1 - \theta^{(q-1)})}\\
	&\qquad\leq \sum_{q=1}^Q \Big(\Del^{\frac12-H-(\ga-\frac12)(1-\theta^{(q-1)})}+\Del^{(\ga+\frac12-H)\theta^{(q-1)}-(\ga-\frac12)}+\Del^{\theta^{(q)}(1-H)+(\frac12-\theta^{(q-1)})}\Big).
\end{align*} 
The last term clearly goes to $0$ because $\theta^{(q-1)}\leq\theta<\frac12$ by \eqref{theta:value}.
Without loss of generality, we can assume that $\ga>\frac12$ is sufficiently close to $\frac12$ such that the first term is negligible as well. With this particular value, we then make sure that 
$$ \frac{\ga-\frac12}{\ga+\frac12-H}<\theta^{(Q-1)}<\frac{\ga-\frac12}{\ga}, $$
which, on the one hand, is in line with \eqref{it:CLT} and, on the other hand, guarantees that the second term in the preceding display tends to $0$ for all $q=1,\dots,Q$.

Finally, to compute $\bbr^{n,\chi}_{23}(t)$, we  first condition on $\calf^n_{i-\theta_n^{(q-1)}}$. Because $f$ is even and 
$\un{\Delta}^n_i Y^{\di,q}/\DenH$ has a centered normal distribution given $\calf^n_{i-\theta_n^{(q-1)}}$, if follows that 
$\pd^{\chi} f (\Theta^{n,q}_i/\DenH )$ is an element of the direct sum of all odd-order Wiener chaoses. At the same time, the double stochastic integrals in $\bbr^{n,\chi}_{23}(t)$ belongs to the second Wiener chaos; see  Proposition 1.1.4 in \cite{Nual}. Since Wiener chaoses are mutually orthogonal, we obtain  $\bbr^{n,\chi}_{23}(t) = 0$.  Because this reasoning is valid for all multi-indices   with $\lvert \chi \rvert =1$, we have  shown that $\bbr^n_2(t)$ is asymptotically negligible.
\epr

\blem\label{lem:help2} The convergence \eqref{app:lem6:4} holds true.
\elem
\bpr
For $r=0,\dots, Q$ (where $Q$ is as in Lemma~\ref{lem:help1}), define
\begin{align*}
	\mathbb{Y}^{n,r}_i & = 
	\int_{(i - \theta_n)\Delta_n}^{(i + L - 1) \Delta_n} 
	\bigg( \sum_{q=1}^{r} \rho_{(i-\theta^{(q-1)}_n) \Den }
	\ind_{((i-\theta_n^{(q-1)})\Den, (i-\theta_n^{(q)})\Den)}(s) 
	\bigg) \, \dd W_s \, \frac{\uDeni g(s)}{\DenH}, \\
	\Upsilon^{n,r}_i & = c((i-1)\Den) \Den^{1-2H} 
	+ \sum_{q=r+1}^{Q} (\rho \rho^T)_{(i-\theta^{(q-1)}_n) \Den }  
	\int_{(i-\theta_n^{(q-1)})\Den}^{ (i-\theta_n^{(q)})\Den} 
	\frac{\uDeni g(s)^T \uDeni g(s)}{\Den^{2H}}  \, \dd s.
\end{align*}
Note that $\mathbb{Y}^{n,r}_i \in \bbr^{d \times L}$, $\Upsilon^{n,r}_i \in \bbr^{(d \times L) \times (d \times L)}$ and that 
$\mathbb{Y}^{n,Q}_i= \Theta^n_i/\DenH$ by \eqref{dis:CLT}. 
In order to show \eqref{app:lem6:4}, we need the following approximation result for each $r=1, \ldots,Q-1$:
\begin{equation} \label{NR9}
	\DenOneHalf \sum_{i=\theta_n + 1}^{[t/\Den]-L+1}
	\bbe \Big[ \mu_{f(\mathbb{Y}^{n,r}_i + \cdot)}  ( \Upsilon^{n,r,r}_i )  
	-  \mu_{f(\mathbb{Y}^{n,r}_i + \cdot)}  ( \Upsilon^{n,r,r-1}_i  ) 
	\mathrel{\big|} \calf^n_{i - \theta_n} \Big] \limL 0,
\end{equation}
where $\Upsilon^{n,r,q}_i=\bbe[ \Upsilon^{n,r}_i \mid \calf_{i - \theta_n^{(q)}} ]$.
Let us proceed with the proof of \eqref{app:lem6:4}, taking the previous statement for granted. Defining
\begin{equation*}
	\ov{\mathbb{Y}}^{n}_i  = \int_{(i - \theta_n)\Delta_n}^{(i  - 1) \Delta_n} 
	\sum_{q=1}^{Q} \rho_{(i-\theta^{(q-1)}_n) \Den }
	\ind_{((i-\theta_n^{(q-1)})\Den, (i-\theta_n^{(q)})\Den)}(s) 
	\, \dd W_s \, \frac{\uDeni g(s)}{\DenH},
\end{equation*}	
we can use the tower property of conditional expectation to derive
\begin{align*}
	& \bbe \bigg[f \bigg(\frac{\uDeni Y^{\di}}{\Den^H}\bigg) \mathrel{\Big|} \calf_{i- \theta_n} \bigg] 
	= \bbe \bigg[ \bbe \bigg[ f \bigg(\frac{\uDeni Y^{\di}}{\Den^H}\bigg) \mathrel{\Big|} \calf_{i- 1} \bigg] \mathrel{\Big|} \calf_{i- \theta_n} \bigg] \\ 
	& \qquad
	= \bbe \bigg[ \bbe \bigg[ \mu_{f(\ov{\mathbb{Y}}^{n}_i + \cdot)} 
	\bigg( c((i-1)\Den)\Den^{1-2H} \\ & \qquad \quad +	
	(	\rho  \rho^T)_{(i-\theta^{(Q-1)}_n) \Den } 
	\int_{(i - 1)\Delta_n}^{(i+L-1) \Delta_n} 
	\frac{\uDeni g(s)^T \uDeni g(s)}{\Den^{2H}}\,\dd s \bigg) \mathrel{\Big|} \calf_{i- \theta^{(Q-1)}_n} \bigg] \mathrel{\Big|} \calf_{i- \theta_n} \bigg] \\ & \qquad
	= \bbe\Big [ \mu_{f(\mathbb{Y}^{n,Q-1}_i + \cdot)}  ( \Upsilon^{n,Q-1}_i  ) \mathrel{\big|}  \calf_{i- \theta_n}  \Big].
\end{align*}
Thanks to \eqref{NR9},  we can replace $\Upsilon^{n,Q-1}_i=\Upsilon^{n,Q-1,Q-1}_i$ in the last line by $\Upsilon^{n,Q-1,Q-2}_i$. We can then further compute
\begin{equation} \label{NR13}
	\begin{split}
		& \bbe \Big[ \bbe \Big[ \mu_{f(\mathbb{Y}^{n,Q-1}_i + \cdot)} \Big( \Upsilon^{n,Q-1,Q-2}_i \Big) 
		\mathrel{\big|} \calf_{i- \theta^{(Q-2)}_n} \Big] \mathrel{\big|} \calf_{i- \theta_n} \Big] \\ & \qquad
		= \bbe \Big[ \mu_{f(\mathbb{Y}^{n,Q-2}_i + \cdot)} \Big( \Upsilon^{n,Q-2,Q-2}_i\Big) 
		\mathrel{\big|} \calf_{i- \theta_n} \Big] \\ & \qquad
		= \bbe \Big[ \bbe \Big[ \mu_{f(\mathbb{Y}^{n,Q-2}_i + \cdot)} \Big( \Upsilon^{n,Q-2,Q-2}_i \Big) 
		\mathrel{\big|} \calf_{i- \theta^{(Q-3)}_n} \Big] \mathrel{\big|} \calf_{i- \theta_n} \Big].
	\end{split}
\end{equation}
Again by \eqref{NR9}, we may replace $\Upsilon^{n,Q-2,Q-2}_i$ by $\Upsilon^{n,Q-2,Q-3}_i$ in \eqref{NR13}. Repeating this procedure $Q$ times, we obtain 
$\mu_{f(\mathbb{Y}^{n,0}_i + \cdot)}(\bbe [\Upsilon^{n,0}_i \mid \calf_{i- \theta^{(0)}_n} ])
= \mu_f(\bbe [\Upsilon^{n,0}_i \mid \calf_{i- \theta_n} ])$ in the end, which shows \eqref{app:lem6:4}.

It remains to prove \eqref{NR9}. For  $(u,v) \mapsto \mu_{f(u+\cdot)}(v)$, we use $\partial^{{\chi'}}$ to denote differentiation with respect to $u$ (where ${\chi'}\in \bbn_0^{d \times L}$) and $\partial^{{\chi''}}$  to denote differentiation with respect to $v$ (where ${\chi''}\in \bbn_0^{(d \times L) \times (d \times L)}$). By a Taylor expansion of
$\mu_{f(\mathbb{Y}^{n,r}_i + \cdot)} (\cdot)$ around the point 
$(\mathbb{Y}^{n,r}_i, \Upsilon^{n,r,r-1}_i)$, the difference inside $\E[\cdot\mid \calf^n_{i-\theta_n}]$ in \eqref{NR9} equals
\begin{equation} \label{NR10}
	\begin{split}
		& \DenOneHalf \sum_{i=\theta_n + 1}^{[t/\Den]-L+1} \sum_{\lvert {\chi''} \rvert = 1}
		\bbe \Big[ 
		\pd^{{\chi''}}  \mu_{f(\mathbb{Y}^{n,r}_i + \cdot)}(\Upsilon^{n,r,r-1}_i) 
		(\Upsilon^{n,r,r}_i- \Upsilon^{n,r,r-1}_i )^{\chi'' }
		\mathrel{\big|} \calf^n_{i - \theta_n} \Big] \\ & \quad +
		\DenOneHalf \sum_{i=\theta_n + 1}^{[t/\Den]-L+1} \sum_{\lvert{\chi''} \rvert = 2} \frac{1}{\chi''!}
		\bbe \Big[ 
		\pd^{{\chi''}}  \mu_{f(\mathbb{Y}^{n,r}_i + \cdot)}(\ov{\upsilon}^n_i)   ( \Upsilon^{n,r,r}_i- \Upsilon^{n,r,r-1}_i  )^{{\chi''}}
		\mathrel{\big|}  \calf^n_{i - \theta_n} \Big]
	\end{split}
\end{equation}
for some
$\ov{\upsilon}^n_i$ between $\Upsilon^{n,r,r}_i$ and $\Upsilon^{n,r,r-1}_i$. 	Write
\begin{equation}\label{eq:cond}
	\begin{split}
		& \bbe \Big [(\rho \rho^T)_{(i - \theta_n^{(q-1)}) \Den} \mathrel{\big|} \mathcal{F}^n_{i-\theta^{(r)}_n}  \Big] - 
		\mathbb{E} \Big [ (\rho  \rho^T)_{(i - \theta_n^{(q-1)}) \Den} \mathrel{\big|} \mathcal{F}^n_{i-\theta^{(r-1)}_n} \Big ] \\ & \qquad =
		\bbe \Big [(\rho \rho^T)_{(i - \theta_n^{(q-1)}) \Den} 
		- (\rho \rho^T)_{(i - \theta_n^{(r-1)}) \Den} \mathrel{\big|} \mathcal{F}^n_{i-\theta^{(r)}_n} \Big ] 	
		\\ &\quad \qquad -
		\mathbb{E}  \Big[(\rho \rho^T)_{(i - \theta_n^{(q-1)}) \Den} - 
		(\rho  \rho^T)_{(i - \theta^{(r-1)}_n) \Den} \mathrel{\big|} \mathcal{F}^n_{i-\theta^{(r-1)}_n} \Big ],
	\end{split}
\end{equation}
and note that, because of Assumption~\ref{Ass:A'} and the identity
\begin{equation} \label{ident:3}
	xy-x_0y_0=y_0(x-x_0)+x_0(y-y_0)+(x-x_0)(y-y_0),
\end{equation}	
the two conditional expectations on the right-hand side of \eqref{eq:cond} are both of size $(\theta^{(r-1)}_n \Den)^{1/2}$. The same holds true if we replace $\rho_{(i-\theta_n^{(q-1)})\Del}$ by $\si_{(i-1)\Del}$. Therefore, 
\begin{equation} \label{NR12}
	\begin{split}
		& \bbe \Big [   \Vert \Upsilon^{n,r,r}_i- \Upsilon^{n,r,r-1}_i   \Vert^p\Big  ]^{\frac{1}{p}}
		\lesssim (\theta^{(r-1)}_n \Den)^{\frac12}.
	\end{split}
\end{equation}
Thus, the second expression in \eqref{NR10} is of size $\Den^{-1/2} ((\theta^{(r-1)}_n \Den)^{1/2})^2 = 
\Den^{1/2-\theta^{(r-1)}}$ which goes to 0 as $\nto$ since all numbers $\theta^{(r)}$ are chosen to be smaller than $\frac12$; see \eqref{it:CLT}.

Next, we expand $\pd^{\chi}  \mu_{f(\mathbb{Y}^{n,r}_i + \cdot)}(\cdot)$ around  
$(0, \Upsilon^{n,r,r-1}_i)$ and write the first expression in \eqref{NR10} as $\bbs^{n}_1(t)+ \bbs^{n}_2(t) + \bbs^{n}_3(t)$, where  
\begin{align*} 
	\bbs^{n}_1(t)& = 
	\DenOneHalf \sum_{i=\theta_n + 1}^{[t/\Den]-L+1} \sum_{\vert {\chi''} \vert = 1}
	\bbe \Big [ 
	\pd^{{\chi''}}  \mu_{f}(\Upsilon^{n,r,r-1}_i)  
	(\Upsilon^{n,r,r}_i- \Upsilon^{n,r,r-1}_i )^{{\chi''}}
	\mathrel{\big|} \calf^n_{i - \theta_n} \Big], \\
	\bbs^{n}_2(t)& = 
	\DenOneHalf \sum_{i=\theta_n + 1}^{[t/\Den]-L+1} \sum_{\vert {\chi'} \vert =   \vert {\chi''} \vert = 1}
	\bbe \Big [ 
	\pd^{\chi'}\pd^{\chi''}  \mu_{f}(\Upsilon^{n,r,r-1}_i)  
	(\mathbb{Y}^{n,r}_i )^{{\chi'}}
	(\Upsilon^{n,r,r}_i- \Upsilon^{n,r,r-1}_i  )^{{\chi''}}
	\mathrel{\big|} \calf^n_{i - \theta_n} \Big ], \\ 
	\bbs^{n}_3(t)& =  \mathtoolsset{multlined-width=0.9\displaywidth} \begin{multlined}[t]
		\DenOneHalf \sum_{i=\theta_n + 1}^{[t/\Den]-L+1} \sum_{\vert {\chi'} \vert = 2, \, \vert {\chi''} \vert = 1}
		\frac{1}{\chi'!}	\bbe \Big [ 
		\pd^{\chi'}\pd^{\chi''} \mu_{f(\varsigma^n_i +\cdot)}(\Upsilon^{n,r,r-1}_i) \\
		\times			 (\mathbb{Y}^{n,r}_i )^{{\chi'}}
		( \Upsilon^{n,r,r}_i- \Upsilon^{n,r,r-1}_i )^{{\chi''}}
		\mathrel{\big|} \calf^n_{i - \theta_n} \Big ],\end{multlined}
\end{align*}
and $\varsigma^n_i$ is a point between 0 and $\mathbb{Y}^{n,r}_i$.
Observe that $\pd^{{\chi''}}  \mu_{f}(\Upsilon^{n,r,r-1}_i)$ is $\calf_{i - \theta_n^{(r-1)}}$-measurable and that the $\calf^n_{i - \theta^{(r-1)}_n}$-conditional expectation of $\Upsilon^{n,r,r}_i - \Upsilon^{n,r,r-1}_i$ is 0. Hence,
$$
\bbe  [ 
\pd^{{\chi''}}  \mu_{f}(\Upsilon^{n,r,r-1}_i) 
( \Upsilon^{n,r,r}_i -\Upsilon^{n,r,r-1}_i )^{{\chi''}}
\mid \calf^n_{i - \theta_n}  ] =0
$$
and it follows that $\bbs^{n}_1(t)$   vanishes.
Next,
by \cite{Chong20supp}, Equation~(D.46),
given  $\lvert {\chi'} \rvert = \lvert {\chi''} \rvert =1$, there are $\al, \beta, \ga \in \{1, \ldots, d\} \times \{1, \ldots, L\}$ such that	
\begin{equation*}
	\pd^{\chi'}\pd^{\chi''} \mu_{f(u+ \cdot)}(v) = \frac{\partial \mu_{f(u+ \cdot)}}{\partial u_{\ga} \partial v_{\al, \beta}} (v) 
	= \frac{1}{2^{\ind_{\{\al = \beta \}}}} \mu_{\pd_{\al \beta \ga} f(u+ \cdot)}(v).
\end{equation*}
If $u=0$, since $f$ has odd third derivatives, we have that $\mu_{\pd_{\al \beta \ga} f}(v) = 0$. Therefore, 
the $\pd^{\chi'}\pd^{\chi''}  \mu_{f}$-expression in $\bbs^{n}_2(t)$ is equal to 0, so   $\bbs^{n}_2(t)$ vanishes as well. 
Finally, we use the generalized H\"older inequality and the estimates \eqref{NR12} and \eqref{SSE:fBM} to see that 
\begin{equation*}
	\begin{split}
		\bbe \bigg[ \sup_{t \leq T} \big\vert \bbs^{n}_3(t) \big\vert \bigg]  & \lesssim 
		\DenOneHalf \sum_{i=\theta_n + 1}^{[T/\Den]-L+1} 
		\bbe   [ \Vert \mathbb{Y}^{n,r}_i  \Vert^4  ]^{\frac{1}{2}}
		\bbe  \Big[  \Vert \Upsilon^{n,r,r}_i- \Upsilon^{n,r,r-1}_i \Vert^4 \Big]^{\frac{1}{4}} \\ & \lesssim
		\Den^{-\frac{1}{2}}    \Den^{2\theta^{(r)}(1-H)}   (\theta^{(r-1)}_n \Den)^{\frac12}.
	\end{split}
\end{equation*}
This converges to $0$	as $\nto$ if $2\theta^{(r)}(1-H)-\frac12\theta^{(r-1)}>0$ for all $r=1,\dots,Q-1$, which is equivalent to $\theta^{(r)}>\frac{1}{4(1-H)}\theta^{(r-1)}$. Because $\frac{1}{4(1-H)}<1$, this condition  means that $\theta^{(r)}$ must not decrease to $0$ too fast. By adding more intermediate $\theta$'s between $\theta^{(0)}$ and $\theta^{(Q-1)}$ if necessary, which does no harm to \eqref{it:CLT}, we can make sure   this is satisfied.
\epr

\blem\label{lem:help3} The convergences \eqref{app:lem6:2} and \eqref{app:lem6:3} hold true.
\elem
\begin{proof}
	By Taylor's theorem, $	\mu_f(\Upsilon_i^{n,0}) - \mu_f (\Upsilon^{n,0,0}_i) $ is equal to
	\begin{equation}\label{eq:taylor}
		\begin{split}
			\sum_{\vert \chi \vert = 1} \pd^{\chi} \mu_f  ( \Upsilon^{n,0,0}_i ) 
			( \Upsilon_i^{n,0} - \Upsilon^{n,0,0}_i  )^{\chi}  
			+  \sum_{\vert \chi \vert = 2} \frac{1}{\chi!}\pd^{\chi} \mu_f (\wt{\upsilon}_i^n) 
			( \Upsilon_i^{n,0} - \Upsilon^{n,0,0}_i  )^{\chi}
		\end{split}
	\end{equation}
	for some $\wt{\upsilon}_i^n$ on the line between $\Upsilon_i^{n,0}$ and $\Upsilon^{n,0,0}_i$. The expression  $\Upsilon_i^{n,0} - \Upsilon^{n,0,0}_i$ contains the difference 
	$
	( \rho\rho^T)_{(i - \theta_n^{(q-1)}) \Den} - 
	\mathbb{E}  [ (\rho \rho^T)_{(i - \theta_n^{(q-1)}) \Den}\mid\mathcal{F}^n_{i-\theta_n}  ] 
	$
	and a similar one with  
	$\rho_{(i - \theta_n^{(q-1)}) \Den}$ replaced by 
	$\si_{(i - 1) \Den}$. Inserting $\rho\rho^T$ or $\si\si^T$ at $(i - \theta_n) \Den$ artificially (cf.\ \eqref{eq:cond}), we can use \eqref{ident:3} and Assumption~\ref{Ass:A'} to find that the said difference
	is of size at most $(\theta_n \Delta_n)^{1/2}$. This immediately leads to the bound
	$\mathbb{E}[\Vert \Upsilon_i^{n,0} - \Upsilon^{n,0,0}_i\Vert^{2}]^{1/2} \lesssim (\theta_n \Delta_n)^{1/2}$, which in turn shows that the second-order term in \eqref{eq:taylor} is $o_\P(\sqrt{\Del})$ by \eqref{theta:value}. Therefore, in \eqref{app:lem6:2}, it remains to consider
	$ \Delh \sum_{i=\theta_n+1}^{[t/\Del]-L+1} \sum_{\vert \chi \vert = 1} \pd^{\chi} \mu_f  ( \Upsilon^{n,0,0}_i ) 
	( \Upsilon_i^{n,0} - \Upsilon^{n,0,0}_i  )^{\chi}$.
	For each $i$, the $\sum_{\vert \chi \vert = 1}$-expression is $\mathcal{F}_i^n$-measurable and has a vanishing conditional expectation given $\mathcal{F}_{i - \theta_n}^n$. Thus, by a martingale size estimate of the type \eqref{MSE}, the whole term is of size
	$\sqrt{\theta_n} (\theta_n \Delta_n)^{{1}/{2}}$ at most,
	which tends to $0$ by \eqref{theta:value}.
	This proves \eqref{app:lem6:2}.
	
	For \eqref{app:lem6:3}, recall $\Upsilon^{n,i}$ from \eqref{cov:v1} and note that the difference $( \Upsilon^{n,i} - \Upsilon_i^{n,0})_{k\ell,k'\ell'}$ equals
	\begin{equation*}
		\sum_{q=1}^{Q} \Big( (\rho \rho^T)_{(i - 1)\Delta_n}  - 
		(\rho  \rho^T)_{(i - \theta^{(q-1)}_n)\Delta_n} \Big)_{k k'}	 
		\int_{(i - \theta^{(q-1)}_n)\Delta_n}^{(i - \theta^{(q)}_n)\Delta_n} \frac{\Delta^n_{i+\ell - 1} g(s) \Delta^n_{i+\ell' - 1} g(s)}{\Den^{2H}} 
		\,  \dd s
	\end{equation*}
	for all $k, k' =1, \ldots, d$ and $\ell, \ell' =1, \ldots, L$. Thus, if we expand 
	\begin{align} \nonumber
		& \DenOneHalf \sum_{i=\theta_n + 1}^{[t/\Del]-L+1}  \{ \mu_f ( \Upsilon^{n,i} ) - \mu_f ( \Upsilon_i^{n,0} )
		\} =
		\DenOneHalf \sum_{i=\theta_n + 1}^{[t/\Delta_n]-L+1} \sum_{\vert \chi \vert =1} 
		\pd^{\chi} \mu_f (\Upsilon^{n,i}) (\Upsilon^{n,i}-\Upsilon^{n,0})^{\chi}\\
		&\qquad				+  \DenOneHalf \sum_{i=\theta_n + 1}^{[t/\Delta_n]-L+1} \sum_{\vert \chi \vert =2} \frac{1}{\chi!}
		\pd^{\chi} \mu_f (\wh\upsilon^n_i) (\Upsilon^{n,i}-\Upsilon^{n,0})^{\chi},\label{NR2}
	\end{align}
	where $\wh\upsilon^n_i$ is some point between $\Upsilon^{n,i}$ and $\Upsilon_i^{n,0}$, H\"older's inequality together with the identity \eqref{ident:3} as well as the moment and regularity assumptions on $\rho$ shows that the last sum in the above display is of size
	$
	\Delta_n^{-1/2} \sum_{q = 1}^{Q} (\theta_n^{(q-1)} \Den) \Den^{4\theta^{(q)}(1-H)}
	$,
	which goes to 0 as $\nto$; cf.\ \eqref{eq:Q}. Next, recall the decomposition \eqref{repr:si-2}. As before, we  incorporate the drift $t \mapsto \int_{0}^{t} \wt{b}_s \, \dd s$ into $\rho^{(0)}$ so that  $\rho = \rho^{(0)} + \rho^{(1)}$ with 
	$\rho^{(1)}_t = \int_{0}^{t} \wt{\rho}_s \, \dd {\wt{W}_s}$. By \eqref{ident:3},
	\begin{align*}  
		& \rho_{(i - 1) \Den}^{k\ell} \rho_{(i - 1) \Den}^{k'\ell} - \rho_{(i - \theta_n^{(q-1)}) \Den}^{k\ell} \rho_{(i - \theta_n^{(q-1)}) \Den}^{k'\ell} \\ & \quad = 
		\Big(\rho_{(i - \theta_n^{(q-1)}) \Den}^{k\ell} \Big\{ \rho_{(i - 1) \Den}^{(0),k'\ell} -  \rho_{(i - \theta_n^{(q-1)}) \Den}^{(0),k'\ell} \Big\}
		+ \rho_{(i - \theta_n^{(q-1)}) \Den}^{k'\ell}  \Big\{ \rho_{(i - 1) \Den}^{(0),k\ell} -  \rho_{(i - \theta_n^{(q-1)}) \Den}^{(0),k\ell} \Big\}\Big)\\ & \qquad
		+ \Big(\rho_{(i - \theta_n^{(q-1)}) \Den}^{k\ell} \Big\{ \rho_{(i - 1) \Den}^{(1),k'\ell} -  \rho_{(i - \theta_n^{(q-1)}) \Den}^{(1),k'\ell} \Big\}    
		+ \rho_{(i - \theta_n^{(q-1)}) \Den}^{k'\ell} \Big\{ \rho_{(i - 1) \Den}^{(1),k\ell} -  \rho_{(i - \theta_n^{(q-1)}) \Den}^{(1),k\ell} \Big\} \Big)  \\ & \qquad
		+   \Big(  \rho_{(i - 1) \Den}^{k\ell}-\rho_{(i - \theta_n^{(q-1)}) \Den}^{k\ell} \Big) 
		\Big(  \rho_{(i - 1) \Den}^{k'\ell}-\rho_{(i - \theta_n^{(q-1)}) \Den}^{k'\ell}  \Big).
	\end{align*}
	The remaining term $\Den^{1/2} \sum_{i=\theta_n + 1}^{[t/\Delta_n]-L+1} \sum_{\vert \chi \vert =1} 
	\pd^{\chi} \mu_f (\Upsilon^{n,i}) (\Upsilon^{n,i}-\Upsilon^{n,0})^{\chi}$ in \eqref{NR2} can thus be written as $\bbt^n_1(t)+\bbt^n_2(t)+\bbt^n_3(t)$ according to this decomposition. By H\"older's inequality and the moment and regularity assumptions on $\rho$,  $\bbt^n_3(t)$ is of size at most
	\begin{equation} \label{size:1}
		\DenMinOneHalf \sum_{q = 1}^{Q} (\theta_n^{(q-1)} \Den) \Den^{2\theta^{(q)}(1-H)},
	\end{equation}
	which goes to 0 as $\nto$ as we saw in \eqref{eq:Q}. Similarly,  thanks to the regularity property \eqref{eq:Hoelder} of $\rho^{(0)}$, we further obtain
	$
	\bbe  [ \sup_{t \leq T}  \vert \bbt^n_1(t)  \vert  ] \lesssim \Delta_n^{-1/2} \sum_{q = 1}^{Q} (\theta_n^{(q-1)} \Den)^{\ga} \Den^{2\theta^{(q)}(1-H)},
	$
	and this also goes to 0 as $\nto$ by our choice \eqref{it:CLT} of the numbers $\theta_n^{(q-1)}$.  Finally,
	\begin{equation*}  
		\begin{split}
			\bbt^n_2(t) & = 
			\DenOneHalf \sum_{i=\theta_n + 1}^{[t/\Delta_n]-L+1}\sum_{q=1}^{Q}  \sum_{\vert \chi \vert =1} 
			\pd^{\chi} \mu_f (\Upsilon^{n,i})   \\
			&\quad\times\bigg \{ \pi^{n,i}_{q-1} 
			\int_{(i - \theta_n)\Delta_n}^{(i + L - 1) \Delta_n} \frac{\uDeni g(s)^T \uDeni g(s)}{\Den^{2H}} 
			\ind_{ ((i - \theta^{(q-1)}_n)\Delta_n, (i - \theta^{(q)}_n)\Delta_n) }(s) \,  \dd s \bigg \}^{\chi},
		\end{split}
	\end{equation*}
	where 
	$
	\pi^{n,i}_q = \rho_{(i - \theta_n^{(q)}) \Den}   (\rho_{(i - 1) \Den}^{(1)} -  \rho_{(i - \theta_n^{(q)}) \Den}^{(1)}  )^T   
	+  (\rho_{(i - 1) \Den}^{(1)} -  \rho_{(i - \theta_n^{(q)}) \Den}^{(1)}  ) \rho_{(i - \theta_n^{(q)}) \Den}^T.
	$
	Define $\wt \bbt^n_2(t)$  in the same way as $\bbt^n_2(t)$ except that in the previous display, $\Upsilon^{n,i}$ is replaced by $\wt{\Upsilon}_{q-1}^{n,i}$, obtained from $\Upsilon^{n,i}$ by substituting $(i - \theta^{(q-1)}_n)\Delta_n$ for $(i-1)\Den$  everywhere. By H\"older's inequality and the regularity assumptions on $\rho$ and $\si$,   $ \bbt^n_2(t)-\wt  \bbt^n_2(t)$ is of the same size as exhibited in \eqref{size:1}   and hence asymptotically negligible. Next,
	\begin{equation} \label{dec:L2:p}
		\begin{split}
			\wt  \bbt^n_2(t) & = 
			\sum_{q=1}^{Q} \DenOneHalf \sum_{i=\theta_n + 1}^{[t/\Delta_n]-L+1}  \sum_{\vert \chi \vert =1} 
			\pd^{\chi} \mu_f (\wt{\Upsilon}_{q-1}^{n,i}) \bigg (
			\bigg\{	\Big(\pi^{n,i}_{q-1} - \bbe  [ \pi^{n,i}_{q-1} \mid \calf^n_{i-\theta^{(q-1)}_n}  ] \Big) \\ &  \quad  
			+
			\bbe  [ \pi^{n,i}_{q-1} \mid \calf^n_{i-\theta^{(q-1)}_n}  ]\bigg\} \int_{(i - \theta^{(q-1)}_n)\Delta_n}^{(i - \theta^{(q)}_n)\Delta_n} \frac{\uDeni g(s)^T \uDeni g(s)}{\Den^{2H}} 
			\,  \dd s \bigg )^{\chi} .
		\end{split}
	\end{equation}
	For fixed $q$, the part that involves $\pi^{n,i}_{q-1} - \bbe  [ \pi^{n,i}_{q-1} \mid \calf^n_{i-\theta^{(q-1)}_n}  ]$  is a sum where the $i$th summand is $\calf^n_{i+L-1}$-measurable and has, by construction, a zero $\calf^n_{i-\theta^{(q-1)}_n}$-conditional mean. By a martingale size estimate of the type \eqref{MSE}, that part   is therefore of size 
	$$
	\sum_{q=1}^{Q}  \sqrt{\theta_n^{(q-1)} }(\theta_n^{(q-1)} \Den)^{1/2} \Den^{2\theta^{(q)}(1-H)} =
	\sum_{q=1}^{Q}  \Den^{  {1}/{2} - \theta_n^{(q-1)} + 2\theta^{(q)}(1-H)} \ra 0
	$$
	as $\nto$ since all $\theta_n^{(q)} < \frac12$. Clearly,
	\begin{equation*}
		\bbe \bigg[\rho_{(i - 1) \Den}^{(1),k\ell} -  \rho_{(i - \theta_n^{(q-1)}) \Den}^{(1),k\ell} \mathrel{\Big|} \calf^n_{i-\theta^{(q-1)}_n}\bigg] =
		\sum_{m=1}^{d} \bbe \bigg[  \int_{(i - \theta_n^{(q-1)}) \Den}^{(i - 1) \Den}\wti{\rho}^{k\ell m}_s \, 
		\dd \wt{W}^m_s \mathrel{\Big|} \calf^n_{i-\theta^{(q-1)}_n} \bigg]=0.
	\end{equation*}
	Because $ \rho_{(i - \theta_n^{(q)}) \Den}$ is $\calf^n_{i-\theta_n^{(q)}}$-measurable, we have, in fact, 
	$
	\bbe  [ \pi^{n,i}_{q-1} \mid \calf^n_{i-\theta^{(q-1)}_n}  ] =0.
	$
	Therefore, $\bbt_2^n(t)$ is asymptotically negligible and the proof of \eqref{app:lem6:3} is complete. 	
\end{proof}

\begin{proof}[Proof of Lemma \ref{lem:Taylor} ]
	Recall the expressions $\mathbb{X}_1^{n}(t)$ and $\mathbb{X}_2^{n}(t)$ defined in \eqref{CLT:ZGW:2}. For a given multi-index $\chi  \in \bbn_0^{(d \times L) \times (d \times L)}$, let 
	$Q_{\chi}(x) = x^{\chi}$ for $x \in \R^{(d\times L)\times(d\times L)}$,
	which is a polynomial of degree $\lvert \chi \rvert$. 
	By Taylor's theorem,
	\begin{equation} \label{tay:ZGW}
		\begin{split}
			\mathbb{X}_1^{n}(t) & =
			\DenOneHalf \sum_{i=\theta_n + 1}^{[t/\Delta_n]-L+1}
			\sum_{j = 1}^{N(H)} 
			\sum_{\vert \chi \vert = j} 	\frac{1}{\chi!}\pd^{\chi} \mu_f(\pi( (i-1)\Den )) 	\sum_{k=1}^{j}  \sum_{\vert \chi' \vert = k} \frac{\Den^{(j-k)(1-2H)}}{\chi'!} \\ & \quad \times 
			\pd^{\chi'} Q_{\chi}(c( (i-1)\Den )) 
			\Big \{ \Upsilon^{n,i} - \pi( (i-1)\Den ) - \Den^{1-2H} c( (i-1)\Den )  \Big\}^{\chi'}.
		\end{split}
	\end{equation}
	The key term in \eqref{tay:ZGW} is the expression in braces and we have (recall \eqref{mat:lim:thms:2} and \eqref{num:Ga})
	\begin{equation} \label{ZGW}
		\begin{split}
			& \Upsilon^{n,i} - \pi( (i-1)\Den ) - \Den^{1-2H} c( (i-1)\Den ) \\ & \quad =
			(	\rho  \rho^T)_{(i-1) \Den} \bigg\{  \int_{(i - \theta_n)\Delta_n}^{(i + L - 1) \Delta_n} \frac{\uDeni g(s)^T \uDeni g(s)}{\Den^{2H}} \,  \dd s
			-  (\Ga^H_{\vert \ell - \ell'\vert} )_{\ell,\ell'=1}^{L,L} \bigg\} \\ & \quad =
			-(\rho \rho^T)_{(i-1) \Den}  \int_{- \infty}^{(i - \theta_n)\Delta_n} \frac{\uDeni g(s)^T \uDeni g(s)}{\Den^{2H}} \,  \dd s,
		\end{split}
	\end{equation}
	because $\Ga^H_{\vert \ell - \ell'\vert} = \Den^{-2H}
	\int_{- \infty}^{\infty} \Delta^n_{i+\ell} g(s) \Delta^n_{i+\ell'} g(s) \,  \dd s$ by \eqref{est:ker:2}. The size of the last integral is 
	$\Den^{2\theta(1-H)}$ by Lemma \ref{lem:ker} (iii). Consequently, if we apply   H\"older's inequality to \eqref{tay:ZGW}, we obtain that
	$
	\bbe  [ \sup_{t \leq T}   \vert \mathbb{X}_1^{n}(t)   \vert  ] \lesssim
	\Den^{- {1}/{2}} \sum_{j = 1}^{N(H)}  \sum_{k=1}^{j}\Den^{(j-k)(1-2H)} \Den^{k2\theta(1-H)}
	\lesssim \Den^{- {1}/{2}+2\theta(1-H)}\to0
	$
	by \eqref{theta:value}.
	Using  \eqref{ZGW} and Assumption~\ref{Ass:A'}, we further see that the magnitude of
	$\Upsilon^{n,i} - \pi( (i-1)\Den )$ is $ \lesssim \Den^{1-2H} + \Den^{2\theta(1-H)}$. Thus, again by   H\"older's inequality, we   deduce that
	$
	\bbe  [ \sup_{t \leq T}   \vert \mathbb{X}_2^{n}(t)   \vert  ] \lesssim
	\Den^{- {1}/{2}}   ( \Den^{(N(H)+1)(1-2H)} + \Den^{(N(H)+1)2\theta(1-H)}  )\to0
	$
	by the definition of $N(H)$.
\end{proof}   

\begin{proof}[Proof of Lemma \ref{lem:app:10}] 	The first convergence \eqref{app:ZGW} can be shown analogously to Equation (5.3.24) in \cite{JP} and is omitted. For \eqref{app:ZGW:2}, we
	write the left-hand side  as $\sum_{j=1}^{N(H)} \mathbb{Z}_j^{n}(t) - \ov{\mathbb{Z}}^{n}(t)$ where 
	\begin{align*}
		\mathbb{Z}_j^{n}(t) & =\mathtoolsset{multlined-width=0.9\displaywidth} \begin{multlined}[t] \Delta_n^{- \frac{1}{2} + j (1-2H) } \sum_{i=\theta_n + 1}^{[t/\Delta_n]-L+1}   \sum_{\vert \chi \vert = j}  \frac{1}{\chi!}\int_{(i-1) \Den}^{i \Den}
			\Big \{ \pd^{\chi} \mu_f  (\pi((i-1)\Den) ) {c((i-1)\Den)}^{\chi} \\ 
			- \pd^{\chi} \mu_f  (\pi(s) ) c(s)^{\chi} \Big \} \, \dd s, \end{multlined}\\
		\ov{\mathbb{Z}}^{n} (t)& = \Delta_n^{- \frac{1}{2}} \bigg( \int_0^{\theta_n \Den } + \int_{([t/\Delta_n] +L-1)\Den}^t \bigg)
		\sum_{j=1}^{N(H)}  
		\sum_{\vert \chi \vert = j} \frac{1}{\chi!}\pd^{\chi} \mu_f  (\pi( s ) )  
		\Den^{j(1-2H)} c(s)^{\chi} 
		\, \dd s.
	\end{align*}
	Using the moment assumptions on $\si$ and $\rho$, since $t - ([t/\Delta_n]-L+1)\Den \leq L\Den$, we readily see that 
	$ \bbe [\sup_{t \leq T}  \vert \ov{\mathbb{Z}}^{n} (t) \vert  ] \lesssim \Delta_n^{- 1/2} (\theta_n \Den + L\Den) \lesssim \Den^{ 1/2 - \theta} + \Den^{ 1/2}\to0$.
	
	Let $j = 1, \ldots, N(H)$ (in particular, everything in the following can be skipped if $H<\frac14$) and consider, for $\chi \in 
	\bbn_0^{(d \times L) \times (d \times L)}$,  again the polynomial $Q_{\chi}$  introduced in proof of Lemma~\ref{lem:Taylor}. Using the mean-value theorem, we can write
	\begin{align*}
		\mathbb{Z}_j^{n}(t) & = \Delta_n^{- \frac{1}{2} + j (1-2H) } \sum_{i=\theta_n + 1}^{[t/\Delta_n]-L+1}   \sum_{\vert \chi \vert = j} \frac{1}{\chi!} \int_{(i-1) \Den}^{i \Den}
		\sum_{\vert \chi_1 + \chi_2 \vert = 1} \pd^{\chi + \chi_1} \mu_f (\zeta^1_{n,i}) \pd^{\chi_2} Q_{\chi}(\zeta^2_{n,i}) \\ & \qquad \times
		\{ \pi((i-1)\Den) - \pi(s) \}^{\chi_1} \{ c((i-1)\Den) - c(s) \}^{\chi_2} \, \dd s
	\end{align*}
	for some    $\zeta^1_{n,i}$ and $\zeta^2_{n,i}$. By  H\"older's inequality and Assumption~\ref{Ass:A}, we deduce that  
	$\bbe [\sup_{t \leq T}  \vert \mathbb{Z}_j^{n}(t) \vert ] \lesssim \Delta_n^{- 1/2 + j(1-2H)} \Den^{-1} \Den \Den^{1/2} = \Delta_n^{j(1-2H)}\to0$ for any $H<\frac12$. 
\end{proof}

\section{Proofs of Theorem \ref{thm:main} and Corollary~\ref{cor:feas}}\label{SecC}

\begin{proof}[Proof of Theorem~\ref{thm:main}]
	We first consider the case where $H\in(\frac14,\frac12)$ and would like to apply Theorem~A.2 in \cite{MP23}. To this end,  define $A_n(\theta) = \Den^{1/2-2H} B_n(\theta)$ and $C_n(\theta)=B_n(\theta)^T$, where 
	\begin{equation}\label{eq:Bn} 
		B_n(\theta)=\begin{pmatrix} \Den^{1-2H} & 2\Den^{1-2H}\lvert \log \Den\rvert \Pi &0\\ 0& \Den^{1-2H} & 0\\ 0 & 0 & 1\end{pmatrix},
	\end{equation}
	and note that $F_n(\theta)=-2D_\theta \mu_n(\theta)^T\wh{\mathcal{W}}_n(\wh V^n_t-\mu_n(\theta))$ with $\mu_n(\theta)=\Den^{2H-1}\Pi\Ga^H+Ce_1$. As $D_\theta\mu_n(\theta)=(\Den^{2H-1}\Pi(\partial_H \Ga^H-2\lvert \log\Den\rvert \Ga^H),\Den^{2H-1}\Ga^H, e_1) \in \R^{(1+R)\times 3}$, it follows from Corollary~\ref{cor:CLT:lags} that
	\begin{align*}
		A_n(\theta_t)F_n(\theta_t)	&=-2\Den^{1/2-2H}\begin{pmatrix}\Den^{1-2H}  & 2\Den^{1-2H}\lvert \log \Den\rvert \Pi_t& 0\\ 0& \Den^{1-2H} & 0\\ 0 & 0 & 1\end{pmatrix}\\
		&\quad\times(\Den^{2H-1}\Pi_t(\partial_H \Ga^H-2\lvert \log\Den\rvert \Ga^H),\Den^{2H-1}\Ga^H, e_1)^T\wh{\mathcal{W}}_n (\wh V^n_t-\mu_n(\theta_t)) \\
		&=-2 ( \Pi_t\partial_H \Ga^H , \Ga^H, e_1)^T\wh{\mathcal{W}}_n\Den^{-1/2}(\Den^{1-2H}\wh V^n_t-\Pi_t\Ga^H-\Den^{1-2H}C_t e_1 ) \\
		&\lims 2 ( \Pi_t\partial_H \Ga^H , \Ga^H, e_1)^T\mathcal{W} \calz = 2E(t)\mathcal{W}^{1/2}\calz. 
	\end{align*}
	This shows property (E.1) in Appendix A of \cite{MP23}.
	
	Next, we verify (E.2)' and note that $F_n(\theta)$ is continuously differentiable around $\theta_t$ with $D_\theta F_n(\theta)=-2D^2_\theta\mu_n(\theta)^T\wh{\mathcal{W}}_n(\wh V^n_t-\mu_n(\theta))+2D_\theta\mu_n(\theta)^T\wh{\mathcal{W}}_nD_\theta\mu_n(\theta)$. Using Corollary~\ref{cor:CLT:lags}, one can show that $-2B_n(\theta_t)D^2_\theta\mu_n(\theta)^T\wh{\mathcal{W}}_n(\wh V^n_t-\mu_n(\theta))C_n(\theta_t)\limp 0$ locally uniformly in a neighborhood of radius $r_n=1/(\log \Den)^2$ around $\theta_t$. This implies that 
	\begin{equation}\label{eq:BCE} 
		\sup_{\theta:\lvert \theta-\theta_t\rvert \leq r_n} \lVert B_n(\theta_t)D_\theta F_n(\theta)C_n(\theta_t) - \ov E(t)\rVert \limp 0, 
	\end{equation} 
	where $\ov E(t)=2( \Pi_t\partial_H \Ga^H , \Ga^H, e_1)^T \mathcal{W} ( \Pi_t\partial_H \Ga^H , \Ga^H, e_1)$. Let $\dot\Ga^H$ and $\dot\partial_H\Ga^H$ be obtained from $\Ga^H$ and $\partial_H\Ga^H$ by omitting the first entry. Since $\dot\Ga^H$ and $\dot\partial_H\Ga^H$ are not collinear and $\partial_H \Ga^H_0=0$, it follows that $( \Pi_t\partial_H \Ga^H , \Ga^H, e_1)$ has full rank, which in turn shows that $\ov E(t)$ is a regular matrix.
	Since $\lVert C_n(\theta_t)\rVert_2\lVert B_n(\theta_t)A_n(\theta_t)^{-1}\rVert_2/r_n\lec \Den^{2H-1/2}\lvert \log \Den\rvert^2\to0$, we have shown (E.2)' in \cite{MP23}.
	Because $A_n(\theta)B_n(\theta)^{-1} = \Den^{1/2-2H}\mathrm{Id}_3$  and $\Den^{2H-1/2}C_n(\theta_t)=D_n(t)$, Theorem~A.2 of \cite{MP23}  implies \eqref{eq:conv-theta}.
	
	The convergence statement in \eqref{eq:conv-theta-2} can be shown analogously,  by omitting the last row and/or column (corresponding to $C_t)$ in all the matrices in the computations above. The local uniqueness statement in (iii) is implied by Theorem~A.2 of \cite{MP23}.
	For the last statement in (iii),   let $\mu'_n(\theta')=\Den^{2H-1}\Pi\Ga^H$ for $\theta'=(H,\Pi)$, so that $F'_n(\theta')=\nabla_{\theta'}\lVert \wh{\mathcal{W}}_n^{1/2}(\wh V^n_t-\Den^{2H-1}\Pi\Ga^H)\rVert^2_2 = -2D_{\theta'} \mu'_n(\theta')^T\wh{\mathcal{W}}_n(\wh V^n_t-\mu'_n(\theta'))$. Furthermore, let $X_\times$ denote the upper left $2\times 2$-matrix of a matrix $X\in\R^{3\times 3}$ and $x_\times$ denote the vector consisting of the first two entries of a vector $x\in\R^3$. Define  $X'_n(\theta) = X_n(\theta)_\times$ for $X\in\{A,B,C\}$. Because  $F_n(\wh \theta^n_t)=0$ (with high probability) by assumption, 
	\[
	0	=  F_n(\wh \theta^n_t) = -2 D_\theta\mu_n(\wh \theta^n_t)^T\wh{\mathcal{W}}_n(\wh V^n_t - \mu_n(\wh \theta^n_t)),
	\] 
	so taking $(\cdot)_\times$ and multiplying by $A'_n(\theta'_t)$ on both sides yield
	\begin{align*}
		0&= -2A'_n(  \theta'_t)D_{\theta'}\mu'_n(\ov \theta^n_t)^T\wh{\mathcal{W}}_n (\wh V^n_t - \mu'_n(\ov \theta^n_t) - \wh C^n_te_1) \\
		&=-2A'_n(  \theta'_t)D_{\theta'}\mu'_n(\ov \theta^n_t)^T\wh{\mathcal{W}}_n(\wh V^n_t - \mu'_n(\ov \theta^n_t) )+o_\P(1)\\
		&= A'_n(  \theta'_t)F'_n(\ov\theta^n_t)+o_\P(1),
	\end{align*}
	where the second step follows from the assumption $H\in(0,\frac14)$ and the fact   that $A'_n(  \theta'_t)$ and $D_{\theta'}\mu'_n(\ov \theta^n_t)$ have matrix norms of magnitude $\Den^{3/2-4H}$ and $\Den^{2H-1}\lvert \log \Den\rvert$, respectively. Since $F'_n(\wh\theta^{\prime n}_t)=0$ by construction, Taylor's theorem yields some $\wt \theta^n_t$ between $\ov\theta^n_t$ and  $\wh\theta^{\prime n}_t$ such that
	\begin{align*}
		0	&= A'_n(\theta'_t)D_{\theta'}F'_n(\wt \theta^n_t)(\ov\theta^n_t-\wh\theta^{\prime n}_t)+o_\P(1)\\ &=A'_n(\theta'_t)B'_n(\theta'_t)^{-1}\bigl(B'_n(\theta'_t)D_{\theta'}F'_n(\wt \theta^n_t)C'_n(\theta'_t)\bigr)C'_n(\theta_t)^{-1}(\ov\theta^n_t-\wh\theta^{\prime n}_t)+o_\P(1).
	\end{align*}
	Similarly to \eqref{eq:BCE}, we have $B'_n(\theta'_t)D_{\theta'}F'_n(\wt \theta^n_t)C'_n(\theta'_t)\to \ov E'(t)$, where $\ov E'(t)=2( \Pi_t\partial_H \Ga^H , \Ga^H)^T\linebreak \mathcal{W} ( \Pi_t\partial_H \Ga^H , \Ga^H)$. Because $\Den^{2H-1/2}C'_n(\theta'_t)=D'_n(t)$ and $A'_n(\theta'_t)(B'_n(\theta'_t))^{-1}=\Den^{1/2-2H}\mathrm{Id}_2$, this shows that 
	$0= \ov E'(t)D'_n(t)^{-1}(\ov\theta^n_t-\wh\theta^{\prime n}_t)+o_\P(1)$.
	Since $E'(t)$ is invertible, we have $D'_n(t)^{-1}(\ov\theta^n_t-\wh\theta^{\prime n}_t)\limp0$, which completes the proof. 
\end{proof}

\begin{proof}[Proof of Corollary~\ref{cor:feas}]
	We only consider the case $H\in(\frac14,\frac12)$; if $H\in(0,\frac14)$, the proof is similar thanks the third part of Theorem~\ref{thm:main}.
	Combining classical spot volatility estimation techniques (see e.g.,  Chapter~8 in \cite{AitSahalia14}) with the assumption that  $\rho$ and $\si$ are $\frac12$-H\"older continuous in squared mean, one can  show that $\Den^{-2H}(\wh m^{n,r}_i -\ov m^{n,r}_i) = O_\P(k_n^{-1/2}\vee \sqrt{k_n\Den})$ uniformly in $r$, where $\ov m^{n,r}_i=\rho^2_{(i-1)\Den}\Ga^H_r +\si^2_{(i-1)\Den}\bone_{\{r=0\}}\Den^{1-2H}$. Therefore, if we define $\ov \eta^{(i)}_r$ (resp., $\ov\Sig^{(\ell)}_n$, $\ov \Sig_n$) in the same way as $ \eta^{(i)}_r$ (resp., $\wh\Sig^{(\ell)}_n$, $\wh \Sig_n$) except that $\wh m^{n,r}_i$ (resp., $\eta^{(i)}$, $\wh \Sig_n^{(\ell)}$) is replaced by $\ov m^{n,r}_i$ (resp., $\ov\eta^{(i)}$, $\ov \Sig_n^{(\ell)}$), then because $K$ is uniformly bounded, we have $\Den^{-4H}(\wh \Sig_n - \ov\Sig_n) = O_\P(\ell_n(k_n^{-1/2}\vee \sqrt{k_n\Den}))\limp 0$. Next, similarly to Lemma~\ref{approx:lem:1}, one can show  that $\Den^{-4H} \ov\Sig_n^{(\ell)} -  \int_0^t a^{(\ell)}_n(s) \,\dd s = O_\P(\sqrt{\Den})$ uniformly in $\ell$, where
	\begin{align*}
		a^{(\ell)}_n(s)_{r,r'} &= \rho_s^4 v^{H,\ell}_{r,r'}+\si_s^4 \Den^{2-4H}(2\bone_{\{\ell=r=r'=0\}}+\bone_{\{\ell=0,r=r'>0\}})\\
		&\quad+\rho_s^2\si_s^2\Den^{1-2H}[\Ga^H_{\lvert r-r'\rvert}( \bone_{\{\ell=r'-r\}}+\bone_{\{\ell=0\}}) + \Ga^H_{r+r'}(\bone_{\{\ell=r'\}} + \bone_{\{\ell=r=0\}})].
	\end{align*}
	Hence, we have $\Den^{-4H}\ov\Sig_n =  \int_0^t [a^{(0)}_n(s)+\sum_{\ell=1}^{\ell_n} K(\ell,\ell_n)(a^{(\ell)}_n(s)+a^{(\ell)}_n(s)^T)]\,\dd s + O_\P(\ell_n\sqrt{\Den})= (v^{H,0}+\sum_{\ell=1}^{\ell_n} K(\ell,\ell_n) (v^{H,\ell}+(v^{H,\ell})^T))  \int_0^t \rho_s^4 \,\dd s+O_\P(\ell_n\sqrt{\Den} \vee \Den^{1-2H})\limp   \calc^H(t)$, where the latter is defined in \eqref{covMat:lags}. Because $E_n(t)=\Den^{-2H-1/2}(\wh\zeta_n D_n(t))^T\wh {  \mathcal{W}}_n^{1/2}\limp E(t)$, we have
	\begin{align*}
		&D_n(t)^{-1}\mathbb{V}_n(D_n(t)^{-1})^T\\
		&\quad=\Den(D_n(t)^T\wh\zeta_n^T\wh {  \mathcal{W}}_n\wh\zeta_nD_n(t))^{-1}D_n(t)^T\wh\zeta_n^T\wh {  \mathcal{W}}_n\wh\Sig_n\wh {  \mathcal{W}}_n\wh\zeta_nD_n(t)(D_n(t)^T\wh\zeta_n^T \wh {  \mathcal{W}}_n\wh\zeta_nD_n(t))^{-1} \\
		&\quad= (E_n(t)E_n(t)^T)^{-1}E_n(t)\wh {  \mathcal{W}}_n^{1/2}\Den^{-4H}\wh\Sig_n\wh {  \mathcal{W}}_n^{1/2}E_n(t)^T(E_n(t)E_n(t)^T)^{-1}\\
		&\quad\limp (E(t)E(t)^T)^{-1}E(t) \mathcal{W}^{1/2} \calc^H(t) \mathcal{W}^{1/2}E(t)^T(E(t)E(t)^T)^{-1}.
	\end{align*}
	By Theorem~\ref{thm:main}, this shows that $\mathbb{V}_n$ is a consistent estimator of the asymptotic covariance matrix of $\wh \theta^n_t$, which implies the assertion of the corollary.
\end{proof}

\begin{figure}[thb!]
	\centering
	\includegraphics[width=\linewidth]{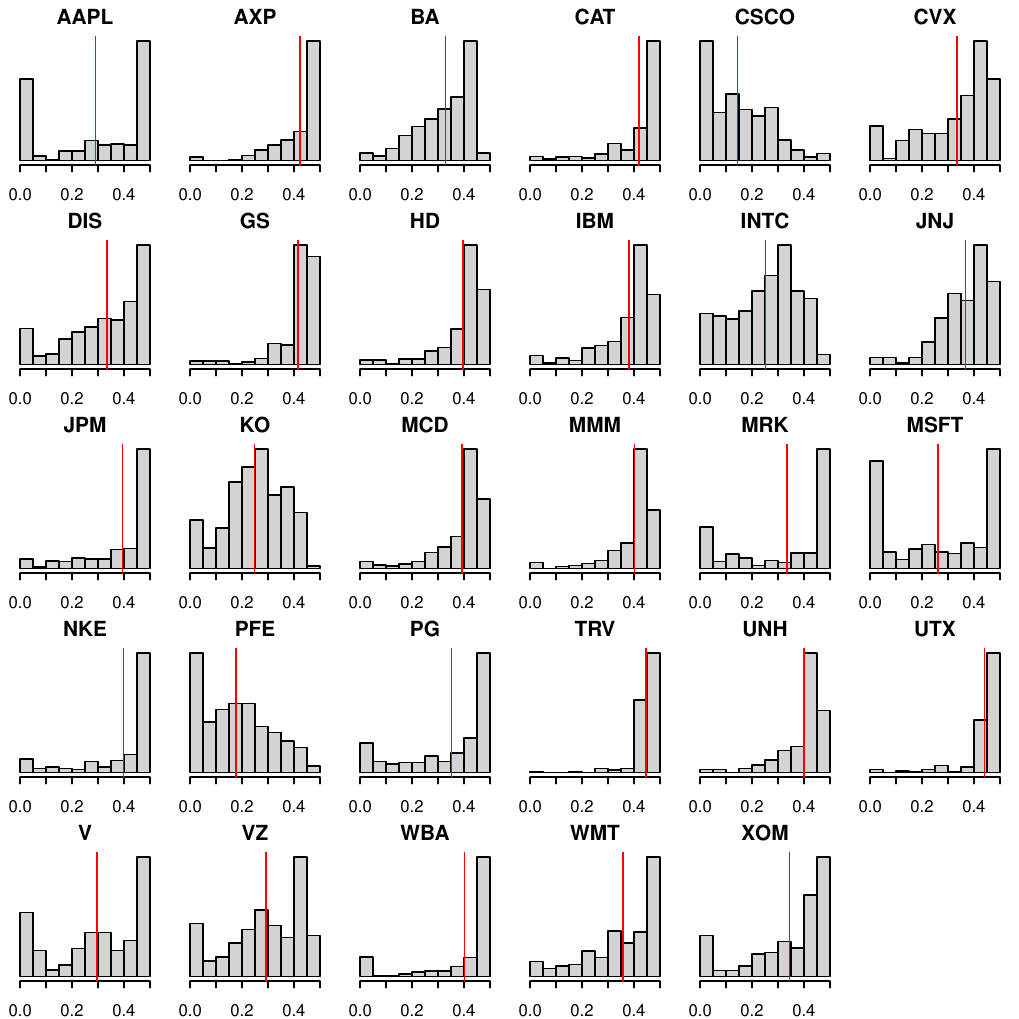}
	\caption{Histogram of daily estimates of $H$ for 29 DJIA companies in 2019, with the mean indicated by a red line.}\label{fig:H29}
\end{figure}

\section{Further empirical results}\label{sec:quote}

We extend the empirical analysis of Section~\ref{sec:emp} to transaction data of single-name stocks. To this end, we consider the 29 stocks that were constituents of the DJIA index for the whole year of 2019. Using the TAQ database, we collect, for each trading day in 2019, all trades 
from 9:30am until 4:00pm Eastern Time. We preprocess the  data in the same way as outlined in Section~\ref{sec:emp}. 
We sample in calendar time every five seconds and choose all tuning parameters  as in the Monte Carlo of Section~\ref{sec:sim}, except that we fix $k_n=60$ (such that $k_n\Den$ remains equal to five minutes).

Figure~\ref{fig:H29} shows a histogram of $H$-estimates for each company, where each estimate corresponds to one trading day and is computed using data of the previous five business days. It is apparent from Figure~\ref{fig:H29} that the distribution of $H$-estimates is highly heterogeneous between different companies. While some histograms appear to be symmetric (e.g., AAPL, KO, MSFT, V), which is similar to what we saw for SPY in 2019 (see Figure~\ref{fig:H:nsr}), others are highly skewed to the left (e.g., AXP, GS, HD, JPM) or to the right (e.g., CSCO, PFE). Also, some companies have only  few days without noise (e.g., CSCO, INTC, KO, PFE), while others have a significant number of noise-free days. This calls for a detailed investigation of the determinants of $H$ in future research.

\bibliographystyle{agsm}

\bibliography{biblio}
\end{document}